\documentclass[11pt]{article}
\usepackage{ae,lmodern} %R�tablissement des polices vectorielles
\usepackage[latin1]{inputenc}
\usepackage[T1]{fontenc}
\usepackage{amsmath}
\usepackage{amsthm}
\usepackage{mathtools}%%new 
\usepackage{amscd}
\usepackage{amssymb}
\usepackage{amsfonts}
\usepackage{graphics}
\usepackage[dvipsnames]{xcolor}
\usepackage{mathtools}
\usepackage{cancel}
\usepackage{geometry}
\usepackage{cite}
\usepackage{enumitem}
\usepackage{soul} 
\setenumerate[2]{label=(\alph*)}
\setenumerate[1]{label=\textup{(\roman*)}}
\setenumerate[3]{label=(\arabic*)}

\newcommand\inner[2]{{\langle #1, #2 \rangle}}
\newcommand{\normW}[1]{{\vert\kern-0.25ex\vert\kern-0.25ex\vert #1 
		\vert\kern-0.25ex\vert\kern-0.25ex\vert}}
\newcommand{\normWbig}[1]{{\Big\vert\kern-0.25ex\Big\vert\kern-0.25ex\Big\vert #1 
		\Big\vert\kern-0.25ex\Big\vert\kern-0.25ex\Big\vert}}

\newcommand{\C}{\mathbb{C}}
\newcommand{\N}{\mathbb{N}}

\newcommand{\R}{\mathbb{R}}

\newcommand{\boA}{\mathcal{A}}
\newcommand{\boB}{\mathcal{B}}
\newcommand{\boC}{\mathcal{C}}
\newcommand{\boD}{\mathcal{D}}
\newcommand{\boE}{\mathcal{E}}

\newcommand{\boH}{\mathcal{H}}

\newcommand{\boL}{\mathcal{L}}

\newcommand{\boN}{\mathcal{N}}

\newcommand{\boQ}{\mathcal{Q}}

\newcommand{\boS}{\mathcal{S}}

\newcommand{\boV}{\mathcal{V}}
\newcommand{\boW}{\mathcal{W}}
\newcommand{\boX}{\mathcal{X}}

\newcommand{\gq}{\mathfrak{q}}
\newcommand{\gc}{\mathfrak{c}}
\newcommand{\NV}{\mathcal{NV}(\mathbb{R})}
\newcommand{\W}{\mathcal{W}}

\renewcommand{\k}{\textup{k}}
\newcommand{\p}{\textup{p}}

\newcommand{\wh}{\widehat }

\newcommand{\Emin}{E_\textup{min}}

\DeclareMathOperator{\supp}{{\rm supp}}

\newcommand{\loc}{\operatorname{loc}}
\providecommand{\abs}[1]{|#1 |}
\providecommand{\norm}[1]{\lVert#1 \rVert}

\renewcommand{\Re}{\operatorname{Re}}
\renewcommand{\Im}{\operatorname{Im}}

\newcommand{\wto}{\rightharpoonup}
\newcommand{\wstar}{\stackrel{\ast}{\rightharpoonup}}
\theoremstyle{plain}
\newtheorem{theorem}{Theorem}[section]
\newtheorem{proposition}[theorem]{Proposition}
\newtheorem{lemma}[theorem]{Lemma}
\newtheorem{corollary}[theorem]{Corollary}

\theoremstyle{definition}

\newtheorem{remark}[theorem]{Remark}

\theoremstyle{remark}

\newtheorem*{merci}{Acknowledgments}
\numberwithin{equation}{section}

\begin{document}
	\title{Existence and decay of traveling waves  for  the nonlocal Gross--Pitaevskii equation}                                 
	\author{
		\renewcommand{\thefootnote}{\arabic{footnote}}
		Andr\'e de Laire\footnotemark[1]~ and Salvador  L\'opez-Mart\'inez\footnotemark[2]}
	\footnotetext[1]{
		Univ.\ Lille, CNRS, Inria, UMR 8524 - Laboratoire Paul Painlev\'e, Inria, F-59000 Lille, France.\\
		E-mail: {\tt andre.de-laire@univ-lille.fr}}
	\footnotetext[2]{
		Departamento de Matem\'aticas, Universidad Aut\'onoma de Madrid, Ciudad Universitaria de Cantoblanco, 28049, Madrid, Spain.\\
		E-mail: {\tt salvador.lopez@uam.es}}
	\date{}
	\maketitle
	%%%%%%%%%%%%%%%%%%%%%%%%%%%%%%%%%%%%%%%%%%%%%%	

	\begin{abstract}
		We consider the nonlocal  Gross--Pitaevskii  equation that models  a Bose gas with general nonlocal interactions between particles in  one spatial dimension, with constant density far away.  We  address the problem of the existence of 
		traveling waves with nonvanishing conditions at infinity, i.e.\ dark solitons. 
		Under general conditions on the interactions, we prove 
		existence of dark solitons for almost every subsonic speed. Moreover, 
		we show existence in the whole subsonic regime for a family of 
		potentials. The proofs rely on a Mountain Pass argument combined with the so-called ``monotonicity trick'',
		as well as on a priori estimates for the Palais--Smale sequences. 
		Finally, we establish properties of the solitons such as exponential decay at infinity and analyticity. 
	\end{abstract}   
	
	\maketitle
	
	\medskip
	\noindent{{\em Keywords:}
		Nonlocal Schr\"odinger equation, Gross--Pitaevskii equation, traveling waves, dark solitons,		nonzero conditions at infinity, decay, analyticity.
		
		\medskip
		\noindent{2010 \em{Mathematics Subject Classification}:}
		%	% 35-XX PARTIAL DIFFERENTIAL EQUATIONS
		35Q55; %NLS-like equations
		35J20; % Variational methods for second-order elliptic equations
		%		35A15 Variational methods
		35C07; % Traveling wave solutions
		%		35B35; % Stability 
		37K05; %  Hamiltonian structures, symmetries, variational principles,		conservation laws
		35C08; % Soliton solutions
		%		35Q53 %KdV-like equations (Korteweg-de Vries) 
		%		82D50;% Superfluids
		%	35R05, %Partial differential equations with discontinuous coefficients or data
		%	35Q60, %PDEs in connection with optics and electromagnetic theory,
		35A01, %Existence problems: global existence, local existence, non-existence
		%	35C06, %Self-similar solutions
		%	35B35, %Stability
		%	35Q55, %NLS-like equations
		%	35Q56, %Ginzburg-Landau equations
		%	35A02, %Uniqueness problems: global uniqueness, local uniqueness, non-uniqueness
		%	53C44. %Geometric evolution equations (mean curvature flow, Ricci flow, etc.)
		%35Q55 NLS-like equations (nonlinear Schr�odinger) 
		37K40 % Soliton theory, asymptotic behavior of solutions
		%		 \tbc 
		%%%%%%%%%%%%%%%%%%%%%%%%%%%%%%%%%%%%%%%%%%%
		\section{Introduction}\label{intro}
		\subsection{The problem}
		%%%%%%%%%%%%%%%%%%%%%%%%%%%%%%%%%%%%%%%%%%%%%%%%
		In order to describe the dynamics of a weakly interacting Bose gas,
		Gross \cite{gross} and 
		Pitaevskii \cite{pitaevskii} found 
		that the wavefunction $\Psi$ governing the condensate satisfies 
		a Schr\"odinger equation, that in dimension one and in its dimensionless form, is given by 
		\begin{equation}
			\label{GP-full}
			i\partial_t\Psi=-\partial_{xx}\Psi+\Psi \int_{\R} \abs{\Psi(y,t)}^2V(x-y)\,dy, 
			\quad 	\text{in } \R\times \R.
		\end{equation}
		Here, $\Psi: \R\times \R\to \C$ and  $V$ describes the interaction between bosons. In their works, they are interested in a function  $\Psi$ satisfying the nonzero condition at infinity:
		\begin{equation}
			%	\label{psi-infinity}
			%	\abs{\Psi}\sim 1, \quad \text{ as }\abs{x}\to \infty,
			\label{nonzero}
			\lim_{\abs{x}\to \infty}\abs{ \Psi(x,\cdot)}=1,
		\end{equation}
		representing the fact that the density is constant far away.
		
		Equation \eqref{GP-full} also appears as the model for the evolution of a one-dimensional optical beam of intensity $\abs{\Psi}^2$ in a self-defocusing nonlocal Kerr-like medium, where $V$ characterizes the nonlocal response of the medium \cite{nikolov2004,krolikowski2000}. 
		In this case, the condition \eqref{nonzero} is natural when studying dark optical solitons.  
		In all of these physical situations, $V$  is assumed to be real-valued and symmetric. Moreover, in the most typical first approximation, $V$ is considered as a Dirac delta function, which leads to the standard  Gross--Pitaevskii equation  with nonvanishing condition at infinity, that  has been intensively investigated (see e.g.\ \cite{JPR1,JPR2,kivshar,coste}).

		To provide a clear mathematical context to the problem, it is useful to perform the change of variables  $\Psi\to e^{-it}\overline{\Psi} $, which leads to the equation 
		%We consider the one-dimensional nonlocal Gross--Pitaevskii equation 
		\begin{equation}
			\label{NGP}
			%\tag{NGP}
			i\partial_{t}\Psi={\partial_{xx}\Psi}+\Psi(\W*(1-|\Psi|^{2}))\quad \text{ in }~\mathbb{R}\times\mathbb{R},
		\end{equation}
		where we assumed that $V*1=1$ and denoted by  $\boW$ the potential to make our following assumptions more clear.
		Here $*$ denotes the convolution in $\R$.  
		%We also make precise the boundary condition at infinity 
		%\begin{equation}
		%	\label{nonzero}
		%	\lim_{\abs{x}\to \infty}\abs{ \Psi(x,\cdot)}=1.
		%\end{equation}
		We assume from now on that  $\W$ is a real-valued even  tempered distribution. In this manner, \eqref{NGP} 
		is Hamiltonian and its energy
		\begin{equation*}
			E(\Psi(t))=\frac12 \int_{\R}\abs{\partial_x\Psi(t)}^2\,dx +\frac 14 \int_{\R}(\W*(1-\abs{\Psi(t)}^2))(1-\abs{\Psi(t)}^2)\,dx,
		\end{equation*}
		is formally conserved. The (renormalized) momentum 
		\begin{equation*}
			p(\Psi(t))=\int_{\mathbb{R}}\langle i \partial_x \Psi(t),\Psi(t) \rangle_{ \C}\left(1-\frac{1}{|\Psi(t)|^2}\right)dx,
		\end{equation*}
		is formally conserved too whenever $\inf_{x\in\R}\abs{\Psi(x,t)}>0$, where $\langle z_1,z_2 \rangle_{\C} =\Re(z_1 \bar z_2)$, for $z_1$, $z_2\in \C$ (see \cite{de2010global}). 
		
		We will be interested in special solutions to \eqref{NGP} with boundary condition \eqref{nonzero}, the so-called dark solitons. Roughly speaking, these are localized density notches that propagate  without spreading \cite{kartashov07}. They have been observed for example in Bose--Einstein condensates \cite{denschlag2000,becker2008}. More precisely, dark solitons in our context will be nontrivial finite energy solutions to \eqref{NGP} of the form
		$$\Psi_c(x,t)=u(x-ct),$$
		which represents a traveling wave with profile $u:\R\to\C$ propagating at speed $c\in\R$. Hence, the soliton $u$ satisfies
		\begin{equation}
			\label{TWc}
			\tag{$S({\W,c})$}
			icu'+u''+u(\mathcal{W}\ast (1-|u|^2))=0\quad\text{ in }\mathbb{R}.
		\end{equation}
		Notice that taking the complex conjugate of $u$ in equation \eqref{TWc}, we are reduced to the case $c\geq 0$. 
		
		By finite energy solution to \eqref{TWc} we mean a solution belonging to the energy space
		$$\mathcal{E}(\mathbb{R})=
		\{v \in H^{1}_{\loc}(\mathbb{R}) : 1-|v|^{2}\in L^{2}(\mathbb{R}), \ v' \in L^{2}(\mathbb{R})\}.$$
		This is justified by assuming that the Fourier transform of $\W$ is bounded, i.e.\ that   $\wh\W\in L^\infty(\R)$. Indeed, by Plancherel's identity, 
		$$E(u)\leq \frac12 \norm{u'}_{L^2(\R)}^2+\frac14 \norm{\wh\W}_{L^\infty(\R)}\norm{1-\abs{u}^2}_{L^2(\R)}^2.$$
		We point out that any function in the energy space satisfies \eqref{nonzero} (see Theorem 1.8 in \cite{gerard3}).
		
		%we say a $u$ is a finite energy solution to \eqref{TWc} if it belongs to 
		% in the energy space 
		%$$\mathcal{E}(\mathbb{R})=
		%\{v \in H^{1}_{\loc}(\mathbb{R}) : 1-|v|^{2}\in L^{2}(\mathbb{R}), \ v' \in L^{2}(\mathbb{R})\}.$$
		% Indeed, by Plancherel's identity, 
		%$$E(u)\leq \frac12 \norm{u'}_{L^2(\R)}^2+\frac14 \norm{\wh\W}_{L^\infty(\R)}\norm{1-\abs{u}^2}_{L^2(\R)}^2.$$

		The simplest case for \eqref{TWc} that one may consider corresponds to the {\em contact} interaction $\W=\delta_0$. In this way, $(S(\delta_0,c))$ becomes the classical Gross--Pitaevskii equation, which is a local equation. In our one-dimensional case, $(S(\delta_0,c))$ can be solved explicitly. More precisely, as explained in \cite{bethuel2008existence}, if  $c\geq \sqrt{2}$ the only solutions in $\boE(\R)$
		are the trivial ones 
		(i.e.\ the constant functions 
		of modulus one). On the contrary, if  \mbox{$0\leq c<\sqrt{2}$}, the nontrivial solutions in $\boE(\R)$ are given, 
		up to invariances (translations and  multiplication by constants of modulus one), by
		\begin{equation}
			\label{sol:1D}
			u_{c}(x)=\sqrt{\frac{2-c^2}{2}}\tanh\left(\frac{\sqrt{2-c^2}}{2}x\right)-i\frac{c}{\sqrt{2}}.
		\end{equation}
		Thus there is a family of dark solitons belonging to the nonvanishing energy space 
		$$ \boN\boE(\R)=\{v \in \boE(\R)   :  \inf_{\R}\abs{v}>0\},$$
		for $c\in (0,\sqrt 2)$. We refer to them as the {\em vortexless} solutions, as usual in nonlinear optics.
		There is also one stationary  soliton that vanishes at exactly one point, associated with the speed $c=0$,
		that is called the black soliton. 
		Notice also that the values of $u_c(\infty)$ and $u_c(-\infty)$ are different if $c\neq 0$, 
		and thus we cannot relax the condition \eqref{nonzero} to $\lim_{\abs{x}\to\infty}{\Psi}=1$ (as in the higher dimensional case, see e.g.\ \cite{bethuel}).

		In the case of spatial dimension equal to two or three, the study of traveling waves for the contact interaction $\W=\delta_0$ started with numerical simulations in the Jones--Roberts program  \cite{JPR1,JPR2}. There, it was observed numerically that finite energy traveling waves should exist for every $c\in [0,\sqrt{2})$, and should not otherwise. Rigorous proofs of these conjectures have been established by 
		B\'ethuel and Saut \cite{bethuel2}, B\'ethuel, Gravejat and Saut \cite{bethuel}, Mari\c{s} \cite{maris2013}, Ruiz and Bellazzini \cite{ruiz-bellazzini}, among others.

		Despite the physical interest of the most realistic case where $\W$ is a more general distribution, there are very few mathematical results concerning  nonlocal interactions with nonzero conditions at infinity. To our knowledge, most of the mathematical results concerning the existence of traveling waves deal with functions  vanishing at infinity (see e.g.\ \cite{lahaye2009,carles2008,antonelli2011,jacopo2016,luo2018}) and  
		the techniques used in these works cannot be adapted to include solutions satisfying \eqref{nonzero}.
		Recently, de Laire and Mennuni \cite{delaire-mennuni} proved the existence of a branch of solutions to \eqref{TWc}, 
		by using a  minimization approach.
		For $\gq\geq 0$, they consider the  (nondecreasing) minimization curve
		\begin{equation}
			\label{Emin}
			E_\text{min}(\gq):=\inf\{E(v)  :   v\in \boE(\R),\ p(v)=\gq \},
		\end{equation}
		and  set 
		\begin{equation}
			\label{q_*}
			\gq_*=\sup\{\gq>0~|~\forall v \in \mathcal{E}(\mathbb{R}), \  E(v)\leq E_{\min}(\gq)\Rightarrow \underset{\mathbb{R}}\inf|v|>0\}.
		\end{equation}
		Under certain technical conditions on $\W$, they show that $\gq_*>0.027$ and that for any $\gq\in(0,\gq_*)$, 
		the minimum  associated with $\Emin(\gq)$  is attained
		and  the corresponding 		 Euler--Lagrange equation 
		satisfied by the   minimizers 
		is exactly \eqref{TWc}, where $c\in (0,\sqrt 2)$ appears as a Lagrange multiplier.
		In addition, their solutions are orbitally stable. Therefore, their result establishes the existence of a family of solutions to \eqref{TWc} parametrized by the momentum.
		This theorem applies for instance to  the potential 
		$\W_{\alpha,\beta}=\frac{\beta}{\beta-2\alpha}(\delta_0-\alpha e^{-\beta\abs{x}})$,
		for  $\beta >2\alpha>0$, which describes a strong repulsion force when two particles are in the same place,  but an attractive force otherwise. However, the results in \cite{delaire-mennuni} do not apply 
		to an interaction of the form $\W(x)=\exp(-\alpha x^2)$, $\alpha> 0$, since its Fourier transform grows  exponentially  in the complex plane.  
		The first goal of this paper is to provide simple conditions on $\W$ that guarantee the existence  of nontrivial finite energy solutions to \eqref{TWc}, covering a large variety of relevant nonlocal interactions, such as the Gaussian potential. The second one is to determine an optimal range for $c$ (depending on $\W$) for which there exist finite energy traveling waves. Finally, a third goal is to establish the regularity and the decay of these solutions, and their nonexistence at critical speed.
		
		\subsection{Main results}
		%%%%%%%%%%%%%%%%%%%%%%%%%%%%%%%%%%%%%%
		From now on we  assume that $\wh\W$ satisfies the following minimal regularity  assumption: 
		\begin{enumerate}[label=(H\arabic*),ref=\textup{({H\arabic*})}]
			\setcounter{enumi}{-1}
			\item\label{H0} $\W$ is an even tempered distribution such that  $\wh\W\in L^\infty(\R)$.
		\end{enumerate}
		Let us remark  that the condition $\wh\W\in L^\infty(\R)$ is equivalent to the continuity of the application $\eta\in L^2(\R) \mapsto  \W*\eta \in L^2(\R)$ (see e.g.\ \cite{grafakos}). The parity assumption is necessary  to have a variational formulation (see Lemma~\ref{lemma:Jsmooth}).   
		
		As explained in \cite{delaire-mennuni}, the Bogoliubov dispersion relation \cite{bogo} is given by
		
		\begin{equation}\label{bogo}
			w(\xi)=\sqrt{\xi^4+2\widehat \W(\xi) \xi^2}.
		\end{equation}
		We formally  get $ w(\xi)\approx {(2\wh \W(0))}^{1/2}\abs{\xi}$, for $\xi\approx 0$. The critical speed
		$c_*(\W)=(2\wh \W(0))^{1/2}$ corresponds to the so-called speed of sound. It is conjectured that there is no nontrivial  solution to \eqref{TWc} with finite energy when $c(\boW)\geq c_*(\W)$.
		Observe also that if $\wh \W$ is continuous at the origin, then we can assume without loss of generality that $\W$ fulfills the normalization condition (see \cite{delaire-mennuni}) 
		%\begin{equation}
		$	\wh \W(0)=1,$
		%\end{equation}
		so that the speed of sound is well-defined and equal to $c_*(\W)=\sqrt 2$.

		%In our some of our results we do need to impose  any kind of continuity for $\W$, so we generalize the definition of critical speed as 
		%$$c_*(\W)=\lim
		%$$

		% Without lost of generality, we impose the 
		%normalization condition   at the origin $\wh\W(0)=1$, so that  the speed of sound is equal to $\sqrt 2$ for all the potentials considered in this paper.
		
		Here and in what follows we use the convention that the Fourier transform of an integrable function is 
		\begin{equation}
			\label{fourier}
			\widehat{ f}(\xi)=\int_{\R} e^{-ix\xi}f(x)dx.
		\end{equation}In particular,  the Fourier transform of the Dirac delta is $\wh \delta_0=1$, 
		so that  $c_*(\delta_0)=\sqrt{2},$ 
		and the nonexistence conjecture holds when $\W=\delta_0$, as explained before.

		Our first result establishes the existence of dark solitons under the following assumption: 
		\begin{enumerate}[label=(H\arabic*),ref=\textup{({H\arabic*})}]
			\item \label{h:lowerboundWbis}
			There exist $\sigma\in (0,1]$ and  $\kappa\in [0,1/2)$ such that  $\wh\W(\xi)\geq \sigma-\kappa \xi^2$ a.e.\ on $\R$.
		\end{enumerate}
		%%%%%%%%%%%%%%%%%%%%
		\begin{theorem}\label{thm:existenceae}
			Assume that $\W$ satisfies \ref{H0} and  \ref{h:lowerboundWbis}. Then there exists a nontrivial solution to \eqref{TWc} in $\boE(\R)$, for almost every $c\in (0,\sqrt{2\sigma})$.
		\end{theorem}

		As an easy consequence of Theorem~\ref{thm:existenceae}, we prove the following existence result for nonnegative potentials satisfying \ref{h:lowerboundWbis} in the critical case $\kappa=1/2$. This critical lower bound was already considered in \cite{delaire-mennuni}.
		
		\begin{corollary}\label{cor:condition12}
			Assume that $\W$ satisfies \ref{H0}, with 	 $\wh \W \geq 0$ a.e. on $\R$,  and that there is  $\sigma\in (0,1]$
			such that  $\wh\W(\xi)\geq \sigma- \xi^2/2$, for a.e.\ $\abs{\xi}\leq \sqrt{2\sigma}$.
			Then there exists a nontrivial solution to \eqref{TWc} in $\boE(\R)$, for almost every $c\in (0,\sqrt{2\sigma})$.
		\end{corollary}

		Notice that there is no assumption on the continuity of $\W$. For instance, 
		Theorem~\ref{thm:existenceae} applies to the potential 
		$$
		\W=3\delta_{0}-\frac{J_1(2 {\abs \cdot }^{1/2})}{ 2{\abs \cdot }^{1/2}},
		\quad \text { so that }\quad
		\wh\W(\xi)=2+\cos(1/\xi),\quad \text{for all }\xi\in \R\setminus\{0\},$$
		where $J_1$ is the Bessel function of  first kind,  with $\sigma=1$ and $\kappa=0$. This gives us  the existence of  nontrivial finite energy solutions to \eqref{TWc} for a.e.\ $c\in (0,\sqrt{2})$.
		
		%In addition, in this case, if the inequality in \ref{h:lowerboundWbis} holds, then $\sigma\leq 1$.
		%Thus, for a normalized potential, condition  \ref{h:lowerboundWbis}  reduces to find $\kappa\in [0,1/2)$ such that  $\wh\W(\xi)\geq 1-\kappa \xi^2$ a.e.\ on $\R$.
		
		We can also apply Theorem~\ref{thm:existenceae} to the potential
		\begin{equation}
			\label{pot:berloff}
			\wh \W_{a,b,\lambda}(\xi)=(1+a\xi^2+b\xi^4)e^{-\lambda\xi^2}, 
		\end{equation}
		that has been proposed in   \cite{berloff0,reneuve2018}
		to describe a quantum fluid exhibiting a roton-maxon spectrum such as Helium~4.
		Indeed, as predicted by the Landau theory, in such a fluid, the dispersion curve \eqref{bogo}
		cannot be monotone, and it should have a local maximum  and a local minimum, the so-called maxon and roton, respectively. In \cite{delaire-mennuni}, some numerical simulations were done for 
		$a=-36$, $b=2687$, $\lambda=30$, and a branch of solitons was found with speeds in $(0,\sqrt 2)$.  These values are relevant because they provide the existence of a  maxon  and a roton.  However, the existence theorem in \cite{delaire-mennuni} does not apply to this potential.  On the other hand, it can be checked that, for these values of $a$, $b$ and $\lambda$, condition \ref{h:lowerboundWbis} is fulfilled with $\sigma=0.175$ and any $\kappa\in (0,1/2)$. Consequently,  Theorem~\ref{thm:existenceae} provides the existence of nontrivial finite energy solutions to $(S(\W_{a,b,\lambda},c))$ for a.e.\ $c\in (0,\sqrt{0.35})$.

		%(see page 433 in the formula 2.5.24.1 on page 433 from Brychkov, Marichev, Prudnikov Integral and Series, vol. 
		%Integrals and series / by A. P. Prudnikov,... Yu. A. Brychkov,... and O. I. Marichev,... ; transl. from the russian by N. M. Queen,.... Volume 1. Elementary functions
		%Auteur
		%Prudnikov, Anatolij Platonovi? 
		
		The first step to prove Theorem~\ref{thm:existenceae} is to show that, if $c>0$, then any solution $u\in \boE(\R)$
		must belong to $\boN\boE(\R)$, which allows us to  lift the function as $u=\rho e^{i\theta}$, for some real-valued functions $\rho$ and $\theta$. Then, we prove in Section~\ref{sec:identities} the crucial fact that \eqref{TWc} is actually equivalent to the nonlocal singular equation 
		\begin{equation}
			\label{eq:rho:intro}
			-\rho''+\frac{c^2}{4}\frac{1-\rho^4}{\rho^3}=\rho\left(\mathcal{W}*(1-\rho^2)\right)\quad\text{in }\mathbb{R}.
		\end{equation}	 
		Therefore, we only need to look for a {\em real} solution $\rho \in \boN\boE(\R)$. Moreover, we know now that $\rho(x)=\abs{u(x)}\to 1$ as $\abs{x}\to\infty$, so that we can
		suppose that $\rho$ recasts as $\rho=1+v$, for some $v\in H^1(\R)$.
		The drawback is that we have introduced a singularity in the equation, we need thus to take care of the possible vanishing of functions on the variational approximation. Indeed, in Section~\ref{sec:aec} we will show  that the solutions to \eqref{eq:rho:intro}
		correspond to critical points of the functional 
		$J_c:H^1(\R)\to\R\cup\{-\infty\}$ given by 
		$$
		J_c(1-\rho)=\boA(1-\rho)-c^2\boB(1-\rho),\quad \text{for } \rho\in 1+H^1(\R),$$
		where
		$$
		\displaystyle\boA(1-\rho)=\frac 1 2 \int_\R (\rho')^2+\frac 1 4 \int_\R (\W\ast(1-\rho^2))(1-\rho^2)
		\quad \text{and}\quad 
		\displaystyle\boB(1-\rho)=\frac 1 8\int_\R\frac{(1-\rho^2)^2}{\rho^2}.
		$$
		More precisely, it will be obtained as a mountain pass point. However, we cannot directly apply the classical version of the mountain pass theorem for several reasons. First, 
		to handle the singularity of $J_c$, we do not work in $H^1$, but in the 
		nonvanishing open set 
		$$ \boN\boV(\R)=\{ v \in H^1(\R)  : 1-v>0 \text{ in } \R \},$$
		i.e.\ we consider  $\rho=1-v$ with $v\in\boN\boV(\R)$. Hence we need to verify that we can adapt the deformation lemma in this setting. Second, our formulation does not give the boundedness of  Palais--Smale sequences. Nevertheless, we can apply the 
		monotonicity trick introduced by Struwe \cite{struwe88} and generalized by Jeanjean \cite{jeanjean99}, 
		that, roughly speaking, will provide bounded Palais--Smale sequences for almost every 
		$c\in (0,\sqrt{2\sigma})$.
		
		% which is used in \cite{rica,aftalion} to the study of supersolids, whose Fourier transform is
		%\begin{equation}
		%%	\label{}
		%\wh \W_a(\xi)=\frac{2\sin(a\xi)}{\xi}.	
		%\end{equation}
		%By using the elementary inequality $\sin(x)/x\geq 1-x^2/6$, for $x\in\R$, we deduce that $\boW_a$ satisfies \ref{H0} and \ref{h:lowerboundWbis} for all $a\in (0,1/2]$. Therefore, we 
		%infer from Theorem~\ref{thm:existenceae} the existence of nontrivial solitons of finite energy to (TW$_{\boW_a,c}$) of speed $c$, for almost every $c\in (0,\sqrt{2a})$, if $a\in (0,1/2]$.
		%
		%{\color{blue} Poner otros ejemplos despu\'es, $(1-\xi^2)^+$}
		
		In order to prove existence in the whole subsonic regime, we have to restrict the potentials we work with; in particular, they will satisfy \ref{h:lowerboundWbis} with $\sigma=1$, so the subsonic interval becomes $(0,\sqrt{2})$. One of the new assumptions on the potentials implies  some a priori estimates on the solutions. To see how this hypothesis is used, observe that, for any $c\in(0,\sqrt2)$, one can apply Theorem \ref{thm:existenceae} to get the existence of  a sequence of speeds $\{c_n\}$ and a sequence of associated solutions $\{u_n\}\subset \boE(\R)$ to $(S(\W,c_n))$, such that $c_n\to c$. At this point, some a priori estimates allow us to conclude that $u_n\to u \in C^k_\text{loc}(\R)$, so that $u$ is a solution to \eqref{TWc}. However, from our estimates, it is not clear that  $u$ has finite energy. Indeed, there are solutions with infinite energy to \eqref{TWc} that we need to avoid, for instance
		\begin{equation*}
			u_r^\pm(x)=r \exp\Big( ix \Big( {\frac{-c\pm\sqrt{c^2+4(1-r^2)}}{2}}\Big) \Big),\quad\text{ for all }r\in (0,1].
		\end{equation*}
		%so that $|u^\pm_r(x)|=r$.
		By imposing more restrictive conditions on the potential $\W$, we  by-pass this difficulty in Section~\ref{sec:allc} by performing a more refined study of the Palais--Smale sequence and using the profile decomposition theorem for bounded functions in $H^1(\R)$. 
		
		For the sake of simplicity, we only state in this introduction the existence result in the whole interval $(0,\sqrt 2)$ for  potentials of the form
			\begin{equation}
				\label{Wf}
				\W_\mu=A_\mu(\delta_0+ \mu),\ \text{ where } \mu \text{ is an even Borel measure with }   \ \norm{\mu^-}<1,\,\,A_\mu=(1+\wh \mu(0))^{-1}.
			\end{equation}
			Here and it what follows, $\mu^-$ and $\mu^+$ denote the negative and positive variations of $\mu$, i.e. $\mu=\mu^+-\mu^-$ for some (unique) positive Borel measures such that $\mu^+\perp\mu^-$ (see \cite{folland}). Also, $\norm{\cdot}$ stands for the total variation of a Borel measure. It can be justified that $\wh\mu(0)=\norm{\mu^+}-\norm{\mu^-}$, so that $1+\wh\mu(0)>0$ and the normalization constant $A_\mu$ is well-defined.  
			
			\begin{theorem}
				\label{thm:existence:Wf}
				Let  $\W_\mu$ be as in \eqref{Wf}. Assume that
				$\wh \mu$ is nondecreasing  on $\R^+$  and that $\wh \mu\in W^{1,\infty}(\R)$.
				Then there exists a nontrivial solution to \eqref{TWc} in $\boN\boE(\R)$ for all $c\in (0,\sqrt{2})$. 
			\end{theorem}
			
			We will show below (see Theorem~\ref{thm:ex2}, case \ref{case:deltas}) that the hypothesis of monotonicity in Theorem~\ref{thm:existence:Wf} can be  relaxed.
		
		In Section~\ref{sec:decay} we study further properties of 
		the solutions 
		by considering the variable $\eta =1-\abs{u}^2$, that satisfies the equation
		\begin{equation}
			\label{eq:ellipticequation:intro}
			-\eta''+2\mathcal{W}\ast\eta-c^2\eta=2\abs{u'}^2+2\eta(\W\ast\eta):= F.
		\end{equation}
		From \eqref{eq:ellipticequation:intro} we can deduce that every finite energy solution is smooth and that, 
		if $c>0$, then $\abs{u}$ does not vanish, i.e.\  $\eta<1$ on $\R$ (see Proposition~\ref{prop:etakidentities}).
		On the other hand, equation  \eqref{eq:ellipticequation:intro} can be written as  
		\begin{equation}
			\label{eq:conv:intro}
			M_c(\xi)\wh \eta(\xi)= \wh F(\xi),\quad \text{with } 
			M_c(\xi)=\xi^2+2\wh \W(\xi)-c^2.
		\end{equation}
		If $c\in [0,\sqrt 2)$ and $\W$ satisfies \ref{h:lowerboundWbis}, then $M_c$ is strictly positive, so \eqref{eq:ellipticequation:intro}
		is an elliptic equation. In this case, \eqref{eq:ellipticequation:intro} can be written as the convolution equation:
		$$\eta =\boL_c *F, \quad \text{ where } \wh \boL_c=M_c^{-1}.$$
		In this way, $\wh \boL_c$
		appears as a Fourier multiplier.	
		
		Let us remark that this kind of convolution formulation has been used in several contexts to get further properties of the solutions, see e.g.\
		\cite{gravejat-decay,delaire2009,deLaire4}. In our case, we can adapt and extend the classical theory of Bona--Li
		\cite{BonaLi1,BonaLi2} to handle the nonlocal function $F$, and to deduce the algebraic or exponential decay, and analyticity of the solutions depending on the properties of $\W$.
		For instance, our main result concerning the exponential decay reads as follows.
		%  for the solution given by theorem 
		%\eqref{thm:existence:Wf} associated with the kernel $\W_f$ we have:
		\begin{theorem}\label{thm:exponentialdecay}
			Assume that $\W$ satisfies \ref{H0}.	Let $c\geq 0$ and let $u\in\mathcal{E}(\mathbb{R})$ be a solution to \eqref{TWc}. Suppose that
			\begin{equation}\label{expLccondition}
				e^{m|\cdot|}\mathcal{L}_c\in L^p(\mathbb{R})\,\,\text{ for some } m>0,\,\, p\in (1,\infty].
			\end{equation}
			Then, for all $\ell\in [0,m)$,  the function  $\eta=1-|u|^2$  has the  exponential decay:  
			$$\lim_{\abs{x}\to\infty}e^{\ell|x|}D^k \eta(x)=0, \quad \text{for all }k\in\N.$$
		\end{theorem}
		\noindent We refer to Section~\ref{sec:decay} for the precise statements for algebraic decay and the real analyticity of $u$, in the sense that $\Re (u)$ and $\Im (u)$ are real analytic functions.

		We now discuss the nonexistence conjecture of nontrivial finite energy solution to \eqref{TWc} for $c\geq \sqrt 2$. In the case $\W=\delta_0$, the proof in  \cite{bethuel2008existence}
		uses the Cauchy--Lipschitz theorem for ODEs. Thus, this argument seems difficult to apply  to the nonlocal equations \eqref{TWc} or \eqref{eq:rho:intro}. In the limit case $c=\sqrt{2}$, we can use \eqref{eq:conv:intro} again to get the following nonexistence result.
		
		\begin{theorem}\label{thm:nonexistence:intro}
			Assume that $\W$ satisfies \ref{H0} and  that   $\wh \W$ 	is of class $\boC^2$
			in a neighborhood of the origin, with  $\wh\W\geq 0$ a.e.\ on $\R$ and
			$\wh \W(0)=1$. Suppose that either   $(\wh \W)''(0)\neq -1$,  or $(\wh \W)''= -1$ on 
			a neighborhood of the origin. Then  $(S(\W,\sqrt 2))$ admits no nontrivial solution in $\mathcal{E}(\mathbb{R})$.
		\end{theorem}
		
		As a consequence of the real analyticity of the solutions to \eqref{TWc}, we can deduce that the solutions obtained by
		minimization by de Laire and Mennuni in \cite{delaire-mennuni} are symmetric.
		%Indeed, they constructed a branch of traveling waves as solutions to the minimization problem \eqref{Emin}, under more restrictive assumptions on $\W$.
		By combining the analyticity with a reflection argument, we get the following result.
		
		\begin{corollary}\label{thm:symmetry}
			Let $\W$ be a potential satisfying the hypotheses in Theorem~1 in \cite{delaire-mennuni}. Let $\gq\in (0,\gq_*)$ and let  $u=\rho e^{i\theta}\in \boN\boE(\R)$ be 
			the nontrivial solution to  \eqref{TWc}, for some $c \in (0,\sqrt 2)$,
			satisfying  $p(u)=\gq$, given by Theorem~1 in \cite{delaire-mennuni}.
			Then, up to translations, $\rho$ is an even function and, up to multiplying $u$ by a constant of modulus one, $\theta$ is an odd function. 
		\end{corollary}

		The uniqueness of solutions to \eqref{TWc} is a difficult problem due to the nonlocal potential. Actually, the uniqueness for nonlocal equations such as \eqref{TWc} can be hard to establish (see e.g.~\cite{albert95,lieb77}).
		We do not know if the solutions to \eqref{TWc} are unique (up to invariances) except in the case $\W=\delta_0$, where the solutions are explicitly given in \eqref{sol:1D}. However, we think that the uniqueness holds, at least for the potentials in the examples in the next subsection.
		
		Another interesting open question is whether the solutions obtained via  Theorem~\ref{thm:existenceae} are orbitally stable. 
		Unlike in \cite{delaire-mennuni}, our solutions are not minimizers but mountain pass critical points.
		This makes the analysis of the stability in our context a nontrivial task. 
		Actually, since uniqueness is not guaranteed, our solutions might in principle be different from those in \cite{delaire-mennuni} and, in consequence, there might exist potentials $\W$ that provide unstable solutions. On the other hand, we performed in \cite{delaire-salva-guillaume} numerical computations for the potentials in the Subsection~\ref{sub:examples}, that suggest the orbital stability
		of the traveling waves in Theorem~\ref{thm:ex2}.

		%
		%
		%One of the ingredients in our a priori estimates is a Pohozaev identity (see Proposition~{prop:pohozaev}).
		%From this identity, we can give some examples of {\em negative} nonlocal potentials such that \eqref{TWc}
		%has only trivial solutions. This is related to the fact that is that case the equation is no longer defocusing. We refer to Remark~\ref{rem:W:neg} for more details.

		%%%%%%%%%%%%%%%%%%%%%%%%%%%%%%%%%%%%%%%%%%%%%
		
		\subsection{Examples}
		\label{sub:examples}
		
		For $\beta>2\alpha>0$, we consider the potential 
		\begin{equation}
			\label{ex:alphabeta}
			\W_{\alpha,\beta}=\frac{\beta}{\beta-2\alpha}(\delta_0-\alpha e^{-\beta|\cdot|}), \quad \text { so that }\quad
			%\end{equation}
			%Clearly, $\W_{\alpha,\beta}\in\boS'(\R)$; in fact, $\W_{\alpha,\beta}$ satisfies \ref{h:Linftybounds}. Its Fourier transform is given by
			%\begin{equation}
			%	\label{ex:Falphabeta}
			\wh\W_{\alpha,\beta}(\xi)=\frac{\beta}{\beta-2\alpha}\left(1-\frac{2\alpha\beta}{\xi^2+\beta^2}\right),\ \xi\in\R.
		\end{equation}
		This kind of potential has been used in \cite{nikolov2004} for the study of dark solitons in a self-defocusing nonlocal Kerr-like
		medium. The kernel $\W_{\alpha,\beta}$ represents a strong repulsive interaction between particles that coincide in space, while the interaction becomes attractive otherwise, being this attraction significant at short distances.
		
		From a mathematical point of view, this potential satisfies all the conditions 
		to apply Theorems~\ref{thm:existence:Wf}, \ref{thm:exponentialdecay}, \ref{thm:nonexistence:intro} and \ref{thm:analyticity}. The result reads as follows.
		
		\begin{theorem}
			\label{thm:ex1}
			Let  $\W_{\alpha,\beta}$ be given by \eqref{ex:alphabeta} with $\beta>2\alpha>0$.
			Then for every $c\in (0,\sqrt{2})$, there exists a nontrivial solution $u_c\in \boN\boE(\R)$ to $(S({\W_{\alpha,\beta},c}))$. 	In addition,  $u_c$ is real-analytic,
			the limits
			$u_c(\pm \infty)$ exist, 
			and there exists $\ell>0$, depending only on $c$, $\alpha$ and $\beta$, such that  $\eta_c=1-\abs{u_c}^2$ satisfies
			\begin{equation}
				\label{eq:decay1}
				%		e^{\ell|\cdot|}D^k\eta\in (L^1\cap L^\infty)(\R),\quad
				\lim_{\abs{x}\to\infty}e^{\ell |x|}D^k\eta_c(x)=0,\quad \text{ for all }k\in\N.
			\end{equation}
			Furthermore,  there is no nontrivial solution to $(S({\W_{\alpha,\beta},\sqrt{2}}))$ in $\boE(\R)$. 
			
		\end{theorem}

		The following three examples will provide similar mathematical results, so we will gather them in a single theorem after some comments. The first one
		was proposed in \cite{veskler2014} as simple model for interactions in a Bose--Einstein condensate. For $\lambda>0$, it is given  by a contact interaction $\delta_0$ and two Dirac delta functions centered at $\pm \lambda$, as
		\begin{equation}
			\label{ex:deltas}
			\W_\lambda=2\delta_0-\frac{1}{2}(\delta_{-\lambda}+\delta_{\lambda}),
			\quad 	\text{ so that } \quad	\wh\W_\lambda(\xi)=2-\cos(\lambda \xi),\ \xi\in\R.
		\end{equation}
		Notice that, as well as \eqref{ex:alphabeta}, \eqref{ex:deltas} models a competition between repulsive and attractive interactions. 
		
		Another interesting example proposed in \cite{Lopez-Aguayo} is the Gaussian function,
		\begin{equation}
			\label{ex:gaussian}
			\W_\lambda(x)=\frac{1}{2\lambda\sqrt{\pi}}e^{-\frac{x^2}{4\lambda^2}},
			\quad 	\text{ so that }\quad	\wh\W_\lambda(\xi)=e^{-\lambda\xi^2},\  \xi\in\R,
		\end{equation}
		where $\lambda>0$. In fact, for $\lambda>0$ small, this potential represents a smooth approximation of the Dirac delta.
		% We also introduce its Fourier transform,

		Finally, we introduce the so-called soft core potential,  which was used in \cite{rica,aftalion} to  study  supersolids. It  can be seen as a nonsmooth approximation of the Dirac delta when $\lambda>0$ is small,
		\begin{equation}
			\label{ex:softcore}
			\W_\lambda(x)=\begin{cases}
				\displaystyle\frac{1}{2\lambda}, & \ \text{ if }\abs{x}<\lambda,\\
				0, & \ \text{otherwise},
			\end{cases}
			\quad 	\text{ so that } \quad	\wh\W_\lambda(\xi)=\frac{\sin(\lambda\xi)}{\lambda\xi}, \ \xi\in\R.
		\end{equation}
		
		Unlike Theorem~\ref{thm:ex1}, for these three potentials we prove existence of nontrivial finite energy traveling waves for {\em almost every} $c\in (0,\sqrt{2})$. We summarize our main results for \eqref{ex:deltas}, \eqref{ex:gaussian} and \eqref{ex:softcore} as follows.
		
		\begin{theorem}
			\label{thm:ex2}
			Assume that one of the following cases holds.
			\begin{enumerate}
				\item\label{case:deltas}  $\W_\lambda$ is given by \eqref{ex:deltas} with $0<\lambda$.
				\item\label{case:gaussian}  $\W_\lambda$ is given by \eqref{ex:gaussian} with $0<\lambda<1/2$.
				\item\label{case:sofcore} $\W_\lambda$ is given by		 \eqref{ex:softcore} with $0<\lambda<\sqrt{3}$.
			\end{enumerate}
			Then, for almost every $c\in (0,\sqrt{2})$,  there exists a nontrivial solution $u_c\in \boN\boE(\R)$ to $(S({\W_\lambda,c}))$. In addition, $u_c$ is  real-analytic, the limits
			$u_c(\pm \infty)$ exist,
			and  there exists $\ell>0$, depending only on $c$  and $\lambda$, such that 
			the function
			$\eta_c=1-|u_c|^2$ satisfies
			\begin{equation}
				\label{eq:decay2}
				%		e^{\ell|\cdot|}D^k\eta\in (L^1\cap L^\infty)(\R),\quad
				\lim_{\abs{x}\to\infty}e^{\ell |x|}D^k\eta_c(x)=0,\quad \text{ for all }k\in\N.
			\end{equation}
			Moreover, in the case \ref{case:deltas}, there exists $\lambda_0\geq \sqrt{2/3}$
				such that if $\lambda\in (0,\lambda_0)$, then  the existence and properties of $u_c$ hold \textbf{for all} $c\in (0,\sqrt{2})$. 
			Finally, in the cases \ref{case:deltas} and \ref{case:gaussian},  there is no nontrivial solution to $(S({\W_\lambda,\sqrt{2}}))$ in $\boE(\R)$. 
		\end{theorem}
		Some comments on this theorem are in order.
		\begin{itemize}
			\item If the Fourier transform of the potential is nonnegative, we can also apply Corollary~\ref{cor:condition12} 
			to study other ranges of $\lambda$. For instance, if  $\W_\lambda$ is given by \eqref{ex:gaussian}, 
			a simple computation shows that for $\lambda\geq 1/2$, we have the estimate $\wh \W_\lambda(\xi)\geq \sigma_\lambda -\xi^2/2$, for all $\xi\in\R$, where $ \sigma_\lambda=\frac{1+\ln(2\lambda)}{2\lambda}$. Therefore, we can deduce the  existence of nontrivial solitons for  a.e.\ $c\in (0,\sqrt{2\sigma_\lambda})$.
			
			\item
			If  $\W_\lambda$ is given by \eqref{ex:softcore}, then 
			$\wh\W_\lambda$ changes sign. For this reason we cannot guarantee  nonexistence of finite energy traveling waves for the critical speed $c=\sqrt{2}$ in Theorem~\ref{thm:ex2}. 
			
			\item Filling the existence in the complete interval $(0,\sqrt 2)$ in Theorem~\ref{thm:ex2} for cases \ref{case:gaussian} and  \ref{case:sofcore}, and even \ref{case:deltas} for $\lambda$ large, would require to prove
			several estimates, a task far from being trivial without our assumptions \ref{h:restrictive} and \ref{W:apriori} (see Section \ref{sec:allc} for details).
			
			\item  
			The arguments in the proof of Theorem~\ref{thm:ex2} also apply to the potential \eqref{pot:berloff}, with the values $a=-36$, $b=2687$, $\lambda=30$.
			Therefore, for a.e.\ $c\in (0,\sqrt{0.35})$, the solution  $u_c$ is  real-analytic and decays exponentially at infinity. Also, there is no nontrivial solution with critical speed $c=\sqrt{2}$ in $\boE(\R)$.
			
		\end{itemize}

		The last example that we consider is given, for $\kappa\in(0,1/2]$, by
		\begin{equation}
			\label{ex:BochnerRiesz}
			\W_{\kappa}(x)=\frac{2\kappa}{\pi x^2}\left(\frac{\sin(x/\sqrt{\kappa})}{x}-\frac1{\sqrt{\kappa}}\cos( x/\sqrt{\kappa})\right),
			\quad 	\text{ so that } 		\wh\W_{\kappa}(\xi)=(1-\kappa\xi^2)^+,
			\ \xi\in\R.
		\end{equation}
		It is simple to check that  $\W_{\kappa}$ is bounded continuous, with $\W_{\kappa}\in L^1(\mathbb{R})$.  From a mathematical point of view, this is an interesting example since
		it represents the limiting case  among the normalized  potentials (i.e.\ $\wh\W(0)=1$) satisfying \ref{h:lowerboundWbis} with nonnegative Fourier transform. This  kernel also  appears in the Bochner--Riesz means, and it is important in the Fourier multipliers theory.
		%  where, for every $\varphi\in L^2(\R)$, the quantities $\wh\W_{\kappa}\wh\varphi$ 
		%  are referred to in the literature as \emph{Bochner-Riesz means} {\color{blue}(see references).}
		
		Due to the lack of regularity of $\wh\W_{\kappa}$, we do not expect an exponential decay of the solution. Nevertheless, we will show that, in this case, $\boL_c$ decays as $1/x^2$, which will lead us to the following result.
		\begin{theorem}
			\label{thm:ex4}
			Let  $\W_{\kappa}$ be given by \eqref{ex:BochnerRiesz}, with $\kappa\in(0,1/2]$.  Then
			for almost every $c\in (0,\sqrt{2})$,  there exists a nontrivial solution $u_c\in \boN\boE(\R)$ to $(S({\W_{\kappa},c}))$. In addition, $u_c$ is real-analytic,  the limits
			$u_c(\pm \infty)$ exist, and the function $\eta_c=1-|u_c|^2$ satisfies the following algebraic decay
			\begin{equation}
				\label{eq:decay3}
				|\cdot|^\ell D^k\eta_c\in L^1(\R),\quad\lim_{\abs{x}\to\infty}|x|^\ell D^k\eta_c(x)=0,\quad\text{ for all }\ell\in [0,1),\,\,\text{ for all }k\in\N.
			\end{equation}
			Moreover, there is no nontrivial solution to $(S({\W_{\kappa},\sqrt{2}}))$ in $\boE(\R)$.
		\end{theorem}
		
		We will study numerically the solitons given in this subsection in the forthcoming paper \cite{delaire-salva-guillaume}.
		
		%%%%%%%%%%%%%%%%%%%%%%%%%%%%%%%%%%%%%%%%
		%Since we are assuming that $\hat W\geq 0$, we 
		%\begin{enumerate}[label=(H\arabic*)]
		%	%	\item \label{h:basic} $\wh\W\in L^\infty(\R)$, $\wh\W\geq 0$  a.e.\ on $\R$, $\wh\W$ is even, continuous at the origin and $\wh\W(0)=1$. 
		%	
		%	\item \label{h:lowerboundW} $\wh\W(\xi)\geq \left(1-\frac{\xi^2}{2}\right)^+$ a.e.\ on $\R$.

\subsection{Notation}

%Throughout the paper, $\boS'(\R)$ will denote the space of real-valued tempered distributions, i.e.\ the dual of the space of real-valued functions in the Schwartz class. On the other hand, 

The usual real-valued Sobolev and Lebesgue spaces will be denoted, respectively, by $W^{k,p}(\R)$ and $L^p(\R)$ for $p\in [1,\infty]$ and $k\in\N$. Moreover, $W^{k,2}(\R)=H^k(\R)$. The notation for the Lebesgue spaces of complex-valued functions will be $L^p(\R;\C)$, and analogously for the Sobolev spaces of complex-valued functions, or simply $L^p(\R)$ if there is no ambiguity.
For a real-valued function $f$, we write  $f^+=\max\{f,0\}$ and $f^-=-\min\{f,0\}$, so that $f=f^+-f^-.$

%We define the space of finite energy functions by
%\[\boE(\R)=\{v\in H^1_{\text{loc}}(\R;\C):\,\, 1-|v|^2\in L^2(\R;\C)\}.\]

%Therefore, for any $u\in \boE(\R)$, we may define the energy as
%\[E_\k(u)=\frac{1}{2}\int_\mathbb{R}K,\quad E_\p(u)=\frac{1}{4}\int_\mathbb{R}(\mathcal{W}\ast\eta)\eta,\quad E(u)=E_\k(u)+E_\p(u).\]

%The space of vortexless finite energy functions will be defined by 
%\[\boN\boE(\R)=\left\{u\in \boE(\R):\,\, \inf_\R |u|>0\right\}.\] 
%In that set, the momentum $p(u)=\frac{1}{2}\int_\mathbb{R}\langle iu',u\rangle\left(1-\frac{1}{|u|^2}\right)$ is well-defined.

%\subsubsection{Assumption on $\W$}

In this paper we {\em always} assume that $\W$
satisfies \ref{H0}. In particular, this implies that 
\begin{equation}
	\label{W-22}
	\norm{\W*f}_{L^2(\R)}\leq  \norm{\wh\W}_{L^\infty(\R)}\norm{f}_{L^2(\R)}, \quad \text{for all }f\in L^2(\R;\C),
\end{equation}
and that Plancherel's identity reads, with our convention for the Fourier transform,
\begin{equation}
	\label{W:even}
	\int_{\R} (\W*f)g=\frac1{2\pi} \int_{\R} \wh \W(\xi) \wh f(\xi)\bar{\wh g}(\xi), \quad \text{for all }f,g\in L^2(\R;\C).
\end{equation}

%{\color{blue} Hay que indicar cu\'al implica cu\'al.}

%\begin{enumerate}[label=(H\arabic*),ref=\textup{({H\arabic*})}]
%%	\item \label{h:basic} $\wh\W\in L^\infty(\R)$, $\wh\W\geq 0$  a.e.\ on $\R$, $\wh\W$ is even, continuous at the origin and $\wh\W(0)=1$. 
%\setcounter{enumi}{1}
%%	\item \label{h:lowerboundW} $\wh\W(\xi)\geq \left(1-\frac{\xi^2}{2}\right)^+$ a.e.\ on $\R$.
%		
%	\item \label{h:derivative} $\wh\W$ is differentiable a.e.\ on $\R$  and the map $\xi\mapsto \xi\big(\wh\W\big)'(\xi)$ is bounded and continuous a.e.\ on $\R$.
%	
%	\item \label{assumption:Linfty}
%	$3\xi^2+2\wh \W(\xi)+2\xi (\wh \W)'(\xi)\geq 2$  a.e.\ on $\R$.
%%	$\xi\big(\wh\W\big)'(\xi)\geq -\xi^2$ a.e.\ on $\R$.
%	
%	\item 
%	\label{h:restrictive} 
%%	$\xi\big(\wh\W\big)'(\xi)\geq 0$ a.e.\ on $\R$.
%		$\wh\W$ is increasing  on $\R^+$.
%	\item 
%	\label{h:Linftybounds}
%	 $\W_f=A_f(\delta_0+ f)$, $f\in L^1(\R)$ a real-valued even function with $\|f^-\|_{L^1(\R)}<1$, $\wh f(0)>-1$ and $A_f=(1+\wh f(0))^{-1}$.
%
%	\item 
%	\label{W:infty} 
%	There is a constant $\norm{\W}_{\infty}$ such that 
%	$\norm{\W*f}_{L^\infty(\R)}\leq \norm{\W}_{\infty} \norm{f}_{L^\infty(\R)},$ for all $f\in {L^\infty(\R)}.$ In other words, $f\mapsto \W*f$ is a bounded linear operator from 
%	$L^\infty(\R)$ to itself.
%	\end{enumerate}
%

%%%%%%%%%%%%%%%%%%%%%%%%%%%%%%

\section{Some identities}
\label{sec:identities}
We start recalling that any finite energy solution to \eqref{TWc} is smooth and admits a lifting at infinity (without restriction on $c$). This result  corresponds to Corollary 2.4 in \cite{delaire2009}, where it was proved in dimension greater or equal than two, but the same proof applies in our one-dimensional setting.

\begin{lemma}
	\label{lemma:regularity} 
	Let $c\geq 0$ and let $u\in \boE(\R)$ be a solution to \eqref{TWc}. Then $u$ is bounded and of class $\boC^\infty(\R)$.
	Moreover, $\eta\coloneqq 1-\abs{u}^2$ and $u'$ belong to $W^{k,p}(\R)$,
	for all $k\in\N$ and $p\in[2, \infty]$. Furthermore, there exists a smooth lifting of $u$. More precisely, there exist $R>0$, $\delta\in(0,1)$ and $\theta\in\boC^\infty((-R,R)^c)$ such that
	$u=\rho e^{i\theta}\quad\text{on }(-R,R)^c$, with 
$	\rho\geq \delta$  on $(-R,R)^c$. In particular, 
$\theta',1-\rho\in W^{k,p}((-R,R)^c)$ for all $k\in\N, p\in[2, \infty]$, and 
\begin{equation}
\label{limits-infty}
\rho(\pm \infty)=1,\  D^j u(\pm \infty)=D^j\rho(\pm \infty)=D^j\theta(\pm \infty)=D^j\eta(\pm\infty)=0, \ \ \text{for all } j\geq 1.
\end{equation}
%		\\
%		&\rho\geq \delta\quad\text{and}\quad u=\rho e^{i\theta}\quad\text{on }(-R,R)^c.
%	\end{align*}
	Finally, if $u\in\boN\boE(\R)$, then the above conclusions still hold true in $\R$, i.e.\  for $R=0$.
\end{lemma}

\begin{proof}
	The proof is contained in \cite{delaire2009} except for the regularity of $1-\rho$. Notice that
	\begin{equation}
		\label{id:trivial}
	1-\rho=\frac{1-\rho^2}{1+\rho}=\frac{\eta}{1+\rho}\quad\text{ on }\R.
	\end{equation}
	Therefore, $|1-\rho|\leq|\eta|$, so $1-\rho\in L^p(\R)$ for every $p\in [2,\infty]$. Moreover, using that $|u|\geq\delta$ on $(-R,R)^c$ and also Cauchy--Schwarz inequality, we deduce that
	\[|\rho'|=\frac{|\langle u,u'\rangle|}{|u|}\leq|u'|\quad\text{ on }(-R,R)^c.\]
Thus $1-\rho\in W^{1,p}((-R,R)^c)$ for all $p\in [2,\infty]$. Similarly, using that 
	\[\rho''=\frac{|u'|^2+\langle u,u''\rangle}{|u|}-\frac{|\langle u,u'\rangle|^2}{|u|^3}\quad\text{ on }(-R,R)^c,\]
%	and then,
%	\[|\rho''|\leq\frac{2|u'|^2}{|u|}+|u''|\quad\text{ on }(-R,R)^c.\]
and that $u,u'\in L^\infty(\R)$, we conclude that $\rho''\in W^{2,p}((-R,R)^c)$ for all $p\in [2,\infty]$. Repeating the previous arguments inductively, we arrive up to $1-\rho\in W^{k,p}((-R,R)^c)$ for all $k\in\N$ and $p\in [2,\infty]$. Finally, if $u\in\boN\boE(\R)$, then $\rho\geq \delta$ on $\R$ for some $\delta\in (0,1)$, so that the previous arguments are valid on $\R$. 
\end{proof}

We now establish some key identities in terms of $\eta=1-\abs{u}^2$. 
In particular, we derive equation \eqref{eq:conv:intro}
and we deduce that, if  $c>0$, the finite energy solutions to \eqref{TWc} do not vanish.
Notice that $\eta \leq 1$ on $\R$, but  $\eta$ could be negative.

%vortexless 

\begin{proposition}\label{prop:etakidentities}
	Let $c\geq 0$ and let $u\in\mathcal{E}(\mathbb{R})$ be a solution to \eqref{TWc}. Setting $K=|u'|^2$ and $\eta=1-|u|^2$, the following identities are satisfied on $\R$:
	\begin{align}
		&\frac{c}{2}\eta=-\langle i u',u\rangle, 
		\label{eq:eta1}
		\\
		&-\eta''+2\mathcal{W}\ast\eta-c^2\eta=2K+2\eta(\W\ast\eta),
		\label{eq:ellipticequation}
		\\
		& K'=\eta'(\mathcal{W}\ast\eta),
		\label{eq:kprime}
		\\
		&c^2\eta^2+(\eta')^2=4K(1-\eta).
		\label{eq:localequation}
	\end{align}
	As a consequence, if $c>0$, then $\eta<1$ on $\R$,  $u\in\boN\boE(\R)$  and  
	\begin{equation}\label{eq:quadratic}
		2K=\frac{c^2\eta^2}{2(1-\eta)}+\frac{(\eta')^2}{2(1-\eta)}\quad\text{ on }\R.
	\end{equation}
%	In particular, $\eta$ satisfies 
%	\begin{equation}\label{eq:singlevariable}
%		-\eta''+2\mathcal{W}\ast\eta-c^2\eta=\frac{c^2\eta^2}{2(1-\eta)}+\frac{(\eta')^2}{2(1-\eta)}+2\eta(\mathcal{W}\ast\eta)\quad \text{ on }\mathbb{R}.
%	\end{equation}
\end{proposition}
Notice that \eqref{eq:ellipticequation} corresponds to equation \eqref{eq:conv:intro}, that we will use in Section~\ref{sec:decay}  to establish the decay and analyticity of solutions, as well as the nonexistence for the critical speed.

\begin{proof}[Proof of Proposition~\ref{prop:etakidentities}]
	Let $u=u_1+iu_2$. Taking real and imaginary parts in \eqref{TWc}, we obtain
	\begin{align}
		u_1''-cu_2'+u_1\left(\mathcal{W}\ast(1-u_1^2-u_2^2)\right)&=0,
		\label{ueq1}
		\\
		u_2''+cu_1'+u_2\left(\mathcal{W}\ast(1-u_1^2-u_2^2)\right)&=0.
		\label{ueq2}
	\end{align}
	Multiplying \eqref{ueq1} by $u_2$ and \eqref{ueq2} by $u_1$, we get
	\[\frac{c}{2}\eta'=(u_1 u_2'-u_2 u_1')'.\]
	Integrating over $\mathbb{R}$ and taking into account \eqref{limits-infty} yields
	\begin{equation*}
		\frac{c}{2}\eta=u_1 u_2'-u_2 u_1',
	\end{equation*}
	which is exactly \eqref{eq:eta1}.
	
	On the other hand, multiplying \eqref{ueq1} by $u_1$ and \eqref{ueq2} by $u_2$, and using \eqref{eq:eta1}, we deduce that
	\begin{equation}\label{etaeq2}
		u_1''u_1+u_2''u_2=\frac{c^2}{2}\eta-(1-\eta)(\mathcal{W}\ast\eta).
	\end{equation}
Hence, by differentiating,
	\begin{equation}\label{etaeq3}
		\eta''=-2|u'|^2-2\langle u,u''\rangle =-2K-2(u_1''u_1+u_2''u_2),
	\end{equation}
which allows us to obtain \eqref{eq:ellipticequation} by combining \eqref{etaeq2} and \eqref{etaeq3}.
	
	Let us now multiply \eqref{ueq1} by $u_1'$ and \eqref{ueq2} by $u_2'$. This gives
	\[u_1'u_1''+u_2'u_2''+(u_1u_1'+u_2u_2')(\mathcal{W}\ast\eta)=0.\]
	Therefore, \eqref{eq:kprime} follows directly by the definitions of $\eta$ and $K$.
	
	In order to show \eqref{eq:localequation}, let us multiply both sides of \eqref{eq:ellipticequation} by $2\eta'$. Taking \eqref{eq:kprime} into account, it is easy to see that
	\[-((\eta')^2)'+4K'-c^2(\eta^2)'=4(K\eta)'.\]
	Hence, \eqref{eq:localequation} follows by integrating this equality. 
	
	Let us now show that if $c>0$, then $\eta<1$ on $\R$. Since 
	$\eta\leq 1$ on $\R$, we assume by contradiction 
that there exists $x_0\in\R$ such that $\eta(x_0)=1$. Then, $x_0$ is a maximum of $\eta$, so that $\eta'(x_0)=0$,
and substituting  into \eqref{eq:localequation} we get $c^2=0$, a contradiction.
	
	Finally, equation \eqref{eq:quadratic} is an immediate consequence of  \eqref{eq:localequation}.
\end{proof}

\begin{remark}
	\label{rem:energymomentum}
	Recall that the energy functional is defined by
	\[E(u)=\frac{1}{2}\int_\R \abs{u'}^2+\frac{1}{4}\int_\R(\W*\eta)\eta\quad\text{ for all }u\in\boE(\R),\]
	where $\eta=1-\abs{u}^2.$
	Moreover, the momentum reads
	\[p(u)=-\frac{1}{2}\int\langle iu',u\rangle\frac{\eta}{1-\eta}.\]
	Taking \eqref{eq:eta1} and \eqref{eq:quadratic} into account, we see that the energy and the momentum of any \emph{solution} $u\in\boN\boE(\R)$ to \eqref{TWc}, with $c>0$, can be written only in terms of $\eta$ and $c$ as
	\begin{align*}
		E(u)=\frac{c^2}{8}\int_\R\frac{\eta^2}{1-\eta}+\frac{1}{8}\int_\R\frac{(\eta')^2}{1-\eta}+\frac{1}{4}\int_\R(\W*\eta)\eta
\quad 		\text{ and }\quad
		p(u)=\frac{c}{4}\int_\R\frac{\eta^2}{1-\eta}.
	\end{align*}
	Furthermore, since $\eta=1-\rho^2$ and $\eta'=-2\rho\rho'$, we also have expressions for the energy and the momentum of a solution $u\in\boN\boE(\R)$ in terms of $\rho$,
	\begin{align*}
		E(u)=\frac{c^2}{8}\int_\R\frac{(1-\rho^2)^2}{\rho^2}+\frac{1}{2}\int_\R(\rho')^2+\frac{1}{4}\int_\R(\W*(1-\rho^2))(1-\rho^2)
		\quad 		\text{ and }\quad
		p(u)=\frac{c}{4}\int_\R\frac{(1-\rho^2)^2}{\rho^2}.
	\end{align*}
	That is the key observation in order to establish the variational framework in Section~\ref{sec:aec}.
\end{remark}

The next result gives an essential reformulation of the complex-valued \eqref{TWc} for vortexless solutions.
In this case, we can reduce the problem to a single real-valued equation. 
	
\begin{proposition}\label{prop:hydrodynamic}
	Let $c\geq 0$. If $u=\rho e^{i\theta}\in \boN\boE(\R)$ is a solution to \eqref{TWc}, then 
	\begin{align}
		&\theta'=\frac{c}{2}\left(\frac{1}{\rho^2}-1\right)\quad\text{on }\mathbb{R},
		\label{eq:theta}
		\\
		-&\rho''+\frac{c^2}{4}\frac{1-\rho^4}{\rho^3}=\rho\left(\mathcal{W}\ast(1-\rho^2)\right)\quad\text{on }\mathbb{R}.
		\label{eq:rho}
	\end{align}
%	where $\rho,\theta\in \boC^\infty(\R)$ are given by Lemma~\ref{lemma:regularity}.
	
	Reciprocally, let $\rho\in \boC^2(\R)$ be such that $\rho>0$ on $\R$ and assume that it satisfies \eqref{eq:rho}. For any $a\in\R$, let us define 
	\begin{equation}\label{eq:deftheta}
		\theta(x)=\frac{c}{2}\int_a^x\left(\frac{1}{\rho(y)^2}-1\right)dy\quad\text{ for all } x\in\R.
	\end{equation}
	Then, the function $u=\rho e^{i\theta}$ belongs to $\boC^2(\R;\C)$ and is a solution to \eqref{TWc}. If in addition $1-\rho\in H^1(\R)$, then $u\in\boN\boE(\R)$. 
\end{proposition}

\begin{proof}
Let  $u=\rho e^{i\theta}$, with  $|u|>0$ on $\R$  and $\theta\in\boC^2(\R)$. By computing the derivatives of $u$
and taking  real and imaginary parts, we check that  $u$ satisfies \eqref{TWc} if and only if the couple $(\rho,\theta)$ satisfies
	\begin{equation}\label{eq:rhothetasystem}
		\begin{cases}
			-c\theta'\rho+\rho''-\rho(\theta')^2+\rho\left(\W\ast(1-\rho^2)\right)=0 &\text{on }\mathbb{R},
			\\
			c\rho'+2\theta'\rho'+\theta''\rho=0 &\text{on }\mathbb{R}.
		\end{cases}
	\end{equation}

Let $u=\rho e^{i\theta}\in \boN\boE(\R)$ be a solution to \eqref{TWc}. By Lemma~\ref{lemma:regularity}, we have $\rho,\theta\in \boC^\infty(\R)$, so that $(\rho,\theta)$ satisfies \eqref{eq:rhothetasystem}. By multiplying the second equation in \eqref{eq:rhothetasystem} by $\rho$,  we obtain 
	\begin{equation}\label{rhothetaderivative}
		(c\rho^2+2\theta'\rho^2)'=0.
	\end{equation}
Bearing in mind \eqref{limits-infty}, we can integrate \eqref{rhothetaderivative} and obtain \eqref{eq:theta}. Plugging \eqref{eq:theta} into the first equation of \eqref{eq:rhothetasystem}, we get \eqref{eq:rho}. 
	
	We turn now to the second part of the result. Indeed, let $\rho\in\boC^2(\R)$ be such that $\rho>0$ on $\R$. Assume that $\rho$ satisfies \eqref{eq:rho} and consider $\theta$ defined by \eqref{eq:deftheta}. Then, one may immediately check that the equations \eqref{eq:rhothetasystem} are satisfied. Hence, $u=\rho e^{i\theta}$ is a solution to \eqref{TWc}.
	It only remains to verify that  $u\in\boN\boE(\R)$ if  $1-\rho\in H^1(\R)$. 
	Indeed, by the Sobolev embedding theorem, we get  $\rho \in L^\infty(\R)$,
 with $\rho(\pm \infty)=1$ and 	 $\rho\geq\delta$ on $\R$ for some $\delta\in(0,1)$.
	Thus
	$$ 1-\abs{u}^2=1-\rho^2=(1-\rho)(1+\rho) \in L^2(\R).$$
	Moreover, by definition of $\theta$, 
$$\abs{u'}^2=(\rho')^2+\rho^2 (\theta')^2=(\rho')^2+\frac{c^2}{4\rho^2}(1-\rho^2)^2,$$
which also belongs to $L^2(\R)$, since $(1-\rho^2)^2=(1-\rho)^2(1+\rho)^2$ and $\rho \in L^\infty(\R).$
\end{proof}

In view of Proposition~\ref{prop:hydrodynamic}, the problem of existence of vortexless finite energy solution to \eqref{TWc} is reduced to the existence of positive solution to \eqref{eq:rho} with  $1-\rho\in H^1(\R)$.  Abusing of the concept of energy, we will say that a solution to \eqref{eq:rho} has finite if $\rho\in 1+H^1(\R)$.

%%%%%%%%%%%%%%%%%%%%%%%%%%%%%%%%%%%%%%%%

\section{The variational formulation }
\label{sec:aec}
In this section we introduce a variational formulation that will lead to the proof of Theorem~\ref{thm:existenceae}. Formally speaking, it is showed in \cite{delaire-mennuni} that critical points of the functional $E(u)-cp(u)$ are (complex-valued) solutions to \eqref{TWc}. Thanks to Proposition~\ref{prop:hydrodynamic}, we may simplify the setting and work in a space of \emph{real-valued} functions. More precisely, we will find solutions to \eqref{eq:rho} as critical points of the functional 
$J_c:H^1(\R)\to\R\cup\{-\infty\}$ formally defined by
$$
	J_c(1-\rho)=\boA(1-\rho)-c^2\boB(1-\rho),\quad \text{ for  }\rho\in 1+H^1(\R),$$ where
	$$\boA(1-\rho)=\frac 1 2 \int_\R (\rho')^2+\frac 1 4 \int_\R
	(\W\ast(1-\rho^2))(1-\rho^2)
\quad \text{ and }\quad
	\boB(1-\rho)=\frac 1 8\int_\R\frac{(1-\rho^2)^2}{\rho^2}.
$$
It is easy to see, thanks to Remark~\ref{rem:energymomentum}, that for every solution $u\in\boN\boE(\R)$ to \eqref{TWc}, the equality $J_c(1-\rho)=E(u)-cp(u)$ holds, where $\rho=\abs{u}$.

Notice that if $\rho\in 1+H^1(\R)$ with $\rho\geq 0$, then $1-\rho\in L^2(\R)$ iff $1-\rho^2\in L^2(\R)$ by \eqref{id:trivial}.
To avoid ambiguities in the definition, we will restrict $J_c$ to 
the nonvanishing set 
\[\NV=\{v\in H^1(\R):\,\, 1-v>0\ \text{ in }\R\},\]
which is an open in $H^1(\R)$ due to the continuous embedding $H^1(\R)\subset L^\infty(\R)$. Thus $J_c$ is defined in the variable $v=1-\rho\in \NV$ by
\[J_c(v)=\frac{1}{2}\int_\R (v')^2+\frac{1}{4}\int_\R (\W\ast(v(2-v)))v(2-v)-\frac{c^2}{8}\int_\R\frac{v^2(2-v)^2}{(1-v)^2}.\]

It is  not difficult to show that the functional satisfies a mountain pass geometry (see Lemmas~\ref{lemma:localminimum} and \ref{lemma:Jnegative}). However, it is not clear at all that the Palais--Smale sequences are bounded. In order to overcome this issue, we take advantage of  the ``monotonicity trick'' of Struwe \cite{struwe88}. More precisely, we are deeply inspired by the work of Jeanjean \cite{jeanjean99}. We adapt some of his results, since several nontrivial modifications are needed due to the singular behavior of $J_c$. This way, we are able to obtain \emph{bounded} Palais--Smale sequences for \emph{almost every} speed $c\in (0,\sqrt{2})$.

We start by showing that $J_c$ is smooth on $\NV$ and that the critical points provide solutions to \eqref{TWc}.
\begin{lemma}\label{lemma:Jsmooth}
Let $c>0$. The functional $J_c$ is of class $\boC^2(\NV)$. Moreover, for any $v\in\NV$,
its Fr\'echet derivatives are given by
	\begin{align}
		\label{eq:Jderivative1}
		J_c'(v)(\phi)&=\int_\R v'\phi'+\int_\R (\W\ast f(v) )(1-v)\phi
		-{c^2}\int_\R h(v)\phi, \\
%			J_c'(v)(\phi)&=\int_\R v'\phi'+\int_\R (\W\ast(v(2-v)))(1-v)\phi
%		-\frac{c^2}{4}\int_\R\frac{v(2-v)(v^2-2v+2)}{(1-v)^3}\phi, \\
						\label{eq:Jderivative2}
		J_c''(v)(\phi,\psi)&=\int_\R\phi'\psi' -c^2 \int_\R h'(v)\phi\psi
+ \int_\R \big(\W\ast(f'(v)\psi)\big)(1-v)\phi - \int_\R\big(\W\ast f(v)\big)\phi\psi, 
	\end{align}
for all $\phi,\psi\in H^1(\R)$,	where
$f(s)=s(2-s)$ and $h(s)=\frac{s(2-s)(s^2-2s+2)}{4(1-s)^3}$ for all $s<1$.

%Moreover, if there exists $ v \in\NV\setminus\{0\}$ such that $J_c'(v)=0$, then $\rho=1-v$
%is a solution to \eqref{eq:rho}. In particular, this provides a nontrivial solution 
%$u\in \boN\boE(\R)$ to \eqref{TWc}.  
\end{lemma}

\begin{remark}
	\label{rem:eq:v}
Observe that if $ v \in\NV\setminus\{0\}$ satisfies $J_c'(v)=0$, then 
\begin{equation}
	\label{eq:v}
	-v''+\big( \mathcal{W}*f(v) \big)(1-v)-{c^2}h(v)=0 \quad\text{on }\mathbb{R}.
\end{equation}
Hence, setting $\rho=1-v$ and noticing that $h(1-\rho)=(1-\rho^4)/(4\rho^3)$
and that $f(1-\rho)=1-\rho^2$, we conclude that $\rho$ is a nontrivial solution 
to \eqref{eq:rho}. Therefore,  by
Proposition~\ref{prop:hydrodynamic}, this provides a nontrivial finite energy solution $u$ to \eqref{TWc}.
\end{remark}

\begin{proof}[Proof of Lemma~\ref{lemma:Jsmooth}]
First, we recall that since $\W$ is even, we have 
\begin{equation}
	\label{eq:commutative}
	\int_\R (\W*g_1)g_2=\int_\R (\W*g_2)g_1,\quad  
\text{ for all }g_1,g_2\in L^2(\R).
\end{equation}
 In order to differentiate the nonlocal term of $J_c$, 	by using the dominated convergence theorem, we  conclude that
 the functional $v \in H^1(\R)\mapsto \int_\R\left(\W\ast f(v)\right)f(v)\in\R$ admits a Fr\'echet derivative, given by
	\[2\int_\R \left(\W\ast f(v)\right)f'(v)\phi, \quad\text {for all }\phi\in H^1(\R).\]
	Therefore, we easily deduce \eqref{eq:Jderivative1}. Computing \eqref{eq:Jderivative2} is also straightforward.
	
	In order to prove the continuity of $J''_c$ in $\NV$, let $v\in\NV$ and consider a sequence $\{v_n\}\subset H^1(\R)$ such that $v_n\to v$  in $H^1(\R)$. Then, $v_n\to v$ in $L^\infty(\R)$. In particular, since there exists $\delta\in (0,1)$ such that $v\leq 1-\delta$ on $\R$, it follows that $v_n\leq 1-{\delta}/{2}$ on $\R$, for  $n$ large enough. Thus, $h'(v_n)\to h'(v)$ in $L^2(\R)$ by using the dominated convergence theorem. The continuity of the other terms is standard.
\end{proof}

To apply a mountain pass argument, we need to invoke a deformation lemma.
Although there are many versions of this classical lemma, we did not find one that fits in our framework since our functional is well-defined only in the open set $\NV$, and not in the whole space. For this reason, we give here a modification of Lemma 2.3 in \cite{willem96}
 that can be applied to our purposes. Furthermore, such a version does not require any Palais--Smale condition. For the sake of completeness, we also include its proof in the appendix.

\begin{lemma}\label{lemma:deformation}
	Let $c>0$. For some $R>0$ and for every $\delta\in (0,1)$, let us consider the set
	\begin{equation}\label{Z}
		Z_\delta=\{v\in\NV: \|v\|_{H^1(\R)}\leq R+1-\delta,\ v\leq 1-\delta\text{ in }\R\}.
	\end{equation}
	Assume that there exist constants  $0<\delta_1<\delta_2<\delta_3<1$, $\varepsilon>0$ and $\gamma\in\R$ such that 
\begin{equation*}
			\|J'_c(v)\|_{H^{-1}(\R)}\geq \frac{2\varepsilon }{\delta_3-\delta_2},\quad\text{ for all } v\in J_c^{-1}([\gamma-2\varepsilon,\gamma+2\varepsilon])\cap Z_{\delta_1}.
	\end{equation*}
	Then there exists a continuous function $h:[0,1]\times \NV\to \NV$ such that
	\begin{enumerate}
		\item \label{def1} $h(0,v)=v,\quad\text{ for all } v\in \NV,$		
		\item \label{def2} $h(t,v)=v,\quad\text{ for all } v\in \NV\setminus\left( J_c^{-1}([\gamma-2\varepsilon,\gamma+2\varepsilon])\cap Z_{\delta_1}\right),
		\text{ for all }t\in [0,1],$	
		\item \label{def2bis} $h(t,Z_{\delta_3})\subset Z_{\delta_2},\quad\text{ for all } t\in [0,1]$,	
		\item \label{def3} $J_c(h(t,v))\leq J_c(v),\quad\text{ for all } v\in \NV,\quad\text{ for all } t\in [0,1],$
		\item \label{def4} $h(1, J_c^{\gamma+\varepsilon}\cap Z_{\delta_3})\subset J_c^{\gamma-\varepsilon}\cap Z_{\delta_2},$
	\end{enumerate}
	where $J_c^d=J_c^{-1}((-\infty,d])$ for every $d\in\R$.
\end{lemma}

\begin{remark}\label{remark:burenkov}
To simplify the statement and the proof of Lemma~\ref{lemma:deformation}, we used the sharp constant in the Sobolev embedding (see \cite[p.\ 138, Theorem 4]{burenkov}), so that
\begin{equation}
	\label{sobolev:sharp}
\|v\|_{L^\infty(\R)}\leq  \frac12 (\norm{v}_{L^2(\R)}+\norm{v'}_{L^2(\R)}) \leq 
\sqrt{  \norm{v}_{L^2(\R)}^2 + \norm{v'}_{L^2(\R)}^2}=
\|v\|_{H^1(\R)}, 
\end{equation}
for all $v\in H^1(\R)$.	
\end{remark}

The following two lemmas provide the mountain pass geometry of $J_c$. 
%In order to prove that $0$ is a local minimum of $J_c$, next condition will be essential:
%\begin{equation}\label{eq:Wsigmakappa}
%	\wh\W(\xi)\geq(\sigma-\kappa \xi^2)^+,\quad\text{ for all } \xi\in\R,\quad\text{ for some }\sigma\in (0,1],\,\,\kappa\in \left[0,\frac{1}{2}\right).
%\end{equation}

\begin{lemma}\label{lemma:localminimum}
	Assume that $\W$ satisfies  \ref{h:lowerboundWbis}, and let $c\in (0,\sqrt{2\sigma})$. Then, there is a constant $r_c>0$ such that, for every $r\in(0,r_c]$, there exists $l_r>0$,
	depending on $\kappa$, $\sigma$, $c$ and $r$, such that $J_c(1-\rho)\geq l_r$ for every $\rho\in 1+H^1(\R)$ with $\|1-\rho\|_{H^1(\R)}=r$.
\end{lemma}

\begin{proof}
	Let $\rho\in 1+H^1(\R)$ be such that $\|1-\rho\|_{H^1(\R)}\leq r_c$ for some $r_c\in (0,1)$ to be chosen later, and let $\eta=1-\rho^2.$ By applying  Plancherel's identity and \ref{h:lowerboundWbis}, we deduce that
		\begin{equation*}
		J_c(1-\rho)\geq\frac{1}{2}\int_\R (\rho')^2+\frac{1}{8\pi}\int_\R (\sigma-\kappa \xi^2)|\wh{\eta}(\xi)|^2 d\xi-\frac{c^2}{8}\int_\R\frac{\eta^2}{\rho^2}.
	\end{equation*} 
Using that 
	$$
\frac{1}{2\pi}\int_\R (\sigma-\kappa\xi^2)|\wh{\eta}(\xi)|^2 d\xi=\sigma \int_\R \eta^2-\kappa \int_\R(\eta')^2, 
$$
and that $\eta'=-2\rho\rho'$, we get the lower bound  
	\begin{align*}
		J_c(1-\rho)&\geq\frac{1}{2}\int_\R (\rho')^2(1-2\kappa\rho^2)+\frac{1}{4}\int_\R \left(\sigma-\frac{c^2}{2\rho^2}\right)\eta^2.
	\end{align*} 
By invoking \eqref{sobolev:sharp}, we see that $1-r_c\leq \rho \leq 1+r_c$ on $\R$,
Thus, recalling that $2\kappa<1$ and $c^2<\sigma$, we choose $r_c>0$ small enough so that $1-2\kappa(1+r_c)^2>0$ and $\sigma-\frac{c^2}{2(1-r_c)^2}>0$.
Consequently, using that $\eta^2=(1-\rho)^2(1+\rho)^2\geq (1-\rho)^2$, we finally obtain
	\begin{equation*}
		J_c(1-\rho)\geq \ell_r\|1-\rho\|^2_{H^1(\R)}, \quad 
		\text{with }
\ell_r=		 \min\left\{\frac{1-2\kappa(1+r_c)^2}{2},\frac{1}{4}\left(\sigma-\frac{c^2}{2(1-r_c)^2}\right)\right\}.	\end{equation*} 
 The result follows by choosing $l_r=\ell_r r^2$.
\end{proof}

\begin{lemma}\label{lemma:Jnegative}
	For every $\gc>0$, there exists $\phi_{\gc}\in 1+H^1(\R)$ such that  $J_c(1-\phi_{\gc})\in(-\infty,0)$ for every $c\geq \gc$.
\end{lemma}

\begin{proof}
	For any $\delta\in (0,1)$ and $r>0$ to be chosen later, let us consider a nonnegative even function $\phi$,  with $0\leq 1-\phi^2\leq 1-\delta \text{ in } \R,$  satisfying the following properties:
	\begin{equation*}
\phi^2=\delta \text{ in } [0,r], \qquad \phi=1 \text{ in } [r+1,\infty), \qquad 
		\phi(x+r)=\psi(x),\text{ for all } x\in [0,1],
\end{equation*} 
where $\psi:\R\to\R$ is a function  independent of $r$ that we choose such that 
$\phi\in\boC^\infty(\R)$.
In particular, $1-\phi\in H^1(\R)$. Thus, by Plancherel's identity, we 
get for all $c\geq \gc$, 
	\begin{align*}
		J_c(1-\phi)&\leq \frac{1}{2}\int_\R (\phi')^2+\frac{\|\wh\W\|_{L^\infty(\R)}}{4}\int_\R(1-\phi^2)^2-\frac{\gc^2}{8}\int_\R\frac{(1-\phi^2)^2}{\phi^2}
		\\
		&=\frac{(1-\delta)^2}{2}\left(\|\wh\W\|_{L^\infty(\R)}-\frac{\gc^2}{2\delta}\right)r+\int_r^{r+1}\left[(\phi')^2+\left(\frac{\|\wh\W\|_{L^\infty(\R)}}{2}-\frac{\gc^2}{4\phi^2}\right)(1-\phi^2)\right].
	\end{align*}
Let us choose $\delta\in (0,1)$ so that $\|\wh\W\|_{L^\infty(\R)}-{\gc^2}/{2\delta}<0$.
Notice that the last integral depends on $\delta$ and $\gc$, but not on $r$,   since $\phi(x+r)=\psi(x)$ for all $x\in [0,1]$. Therefore we may take $r>0$ large enough so that $J_c(1-\phi)<0$. In this way, $\delta$ and $r$ depend on only $\gc$ and $\|\wh\W\|_{L^\infty(\R)}$. The proof concludes by taking $\phi_{\gc}=\phi$. 
\end{proof}

In the rest of this section we assume that $\W$ satisfies \ref{h:lowerboundWbis}, we fix  $\gc\in (0,\sqrt{2\sigma})$ and we focus on speeds $c$ on the interval $(\gc,\sqrt{2\sigma})$.
We consider the paths connecting the origin
with $1-\phi_{\gc}$ given by Lemma~\ref{lemma:Jnegative}, as follows
\[\Gamma(\gc)=\{g\in \boC([0,1],\NV):\,\, g(0)=0,\,\, g(1)=1-\phi_{\gc}\}.\]
Thanks to the mountain pass geometry (Lemma~\ref{lemma:localminimum} and Lemma~\ref{lemma:Jnegative}) and to the continuity of $J_c$ in $\NV$, the so-called mountain pass level is well-defined and  is positive:
\[\gamma_\gc(c)\coloneqq\inf_{g\in\Gamma(\gc)}\max_{t\in [0,1]} J_c(g(t))>0,\quad\text{ for all } c\in [\gc,\sqrt{2\sigma}).\]
%{\color{blue} Indicar la dependencia de $\gamma$ en $c_0$. Quiz\'a cambiar la notaci\'on de $c_0$. Por $\mathfrak{c}$? C\'omo quedar\'a $\gamma_\mathfrak{c}(c)$?
	
%Pensar tambi\'en qu\'e hacer con la notaci\'on $I(c_0)=(c_0,\sqrt{2\sigma})$. Yo creo que de alg\'un modo es necesario considerar esos intervalos, pero en el teorema final nos los podemos ahorrar seguramente.}
	
Notice that the $\Gamma(\gc)$ and $\gamma_\gc(c)$ are not standard  since the paths take values on  the set $\NV$ (and not on a vector space). We recall that  the advantage of working only on $\NV$ is  that  $J_c$ is smooth. However, the drawback of  this setting if that one needs to control $J_c$ near $\partial\NV$ in some sense. In fact, in principle there might exist a sequence $\{\rho_n\}\subset 1-\NV$ such that $\inf_\R \rho_n$ tends to zero but $J_c(1-\rho_n)$ remains finite for all $n$. The next result provides some properties of the functional $\boB$ that prevent undesirable phenomena to happen, so we can deal with the singular behavior near $\partial\NV$.

\begin{lemma}\label{lemma:singularity}
	Given $\rho\in 1+H^1(\R)$, it holds that
	\[\boB(1-\rho)<+\infty \iff 1-\rho\in\NV.\]
	More precisely, if $\rho\in 1+H^1(\R)$ satisfies
	\begin{equation}\label{rhoestimate}
		\|1-\rho\|_{H^1(\R)}+ \boB(1-\rho)\leq R.
	\end{equation}	
for some $R>0$, then there exists $\delta\in(0,1)$ such that
$\rho\geq\delta$ on $\R.$
\end{lemma}

\begin{proof}
	We first prove the equivalence. Let $1-\rho\in \NV$, so that  $\rho$ is continuous, $\rho>0$ on $\R$ and $\rho(\pm\infty)=1$. Thus, taking $x_0\in\R$ such that $\rho(x_0)=\min_\R \rho>0$, we have
	\[8\boB(1-\rho)=\int_\R \frac{(1-\rho^2)^2}{\rho^2}=\int_\R \frac{(1-\rho)^2 (1+\rho)^2}{\rho^2}\leq\frac{\|1-\rho\|_{L^2(\R)}^2\|1+\rho\|_{L^\infty(\R)}^2}{\rho(x_0)^2}<+\infty.\]
For the converse implication, we consider 
$\rho\in 1+ H^1(\R)$ 
such that  $\min_{\R}\rho\leq 0$. Since $\rho(\pm\infty)=1$, we deduce that there is $x_0\in \R$ such that $\rho(x_0)=0$ and $\rho(x)>0$ for all $x>x_0$. 
Thus, the continuous embedding $H^1(\R)\subset\boC^{0,{1}/{2}}(\R)$ implies that
	\[\rho(x)^2=(\rho(x)-\rho(x_0))^2\leq \|1-\rho\|^2_{\boC^{0,\frac{1}{2}}(\R)}(x-x_0),
	\quad \text{for all } x>x_0,\]
so that 	 $\int_{x_0}^{x_0+\delta}{1}/{\rho^2}=+\infty$, for every $\delta>0$. Now we choose $\delta>0$ such that $\min_{[x_0,x_0+\delta]}(1-\rho^2)^2>0$. Hence,
	\[8\boB(1-\rho)\geq\int_{x_0}^{x_0+\delta}\frac{(1-\rho^2)^2}{\rho^2}\geq \min_{[x_0,x_0+\delta]}(1-\rho^2)^2\int_{x_0}^{x_0+\delta}\frac{1}{\rho^2}=+\infty.\]
	Thus the equivalence $\boB(1-\rho)<+\infty \iff 1-\rho\in\NV$ holds true.
	
	We turn now to the proof of the fact that $\rho\geq\delta$ provided that \eqref{rhoestimate} holds. First, we have already proved that any $\rho\in 1+H^1(\R)$ satisfying \eqref{rhoestimate} belongs to $1-\NV$. We argue now by contradiction and assume that there exist $R>0$ and a sequence $\{\rho_n\}\subset 1-\NV$ such that $\rho_n$ satisfies \eqref{rhoestimate} for all $n$ but $\min_\R\rho_n\to 0$ as $n\to\infty$. Since $\{\rho_n(\cdot+x_n)\}\subset 1+H^1(\R)$ still satisfies \eqref{rhoestimate} for any sequence $\{x_n\}\subset\R$, then we may assume without loss of generality that $\rho_n(0)=\min_\R\rho_n\to 0$ as $n\to\infty$. 
By the  embedding $H^1(\R)\subset\boC^{0,{1}/{2}}(\R)$, we deduce that there is a constant $C>0$ such that 
$\|1-\rho_n\|_{\boC^{0,\frac{1}{2}}(\R)}\leq C\|1-\rho_n\|_{H^1(\R)}\leq CR$ for all $n$, so that 
$$
\rho_n(x)\leq \rho_n(0)+CR\sqrt{x},\quad\text{ for all } x>0, \ \text{for all }n.
$$
We conclude as before that  $\lim_{n\to\infty}\int_0^\delta{1}/{\rho_n^2}=+\infty$, for every $\delta>0$. If there exists $\delta>0$ such that the sequence $\{\min_{[0,\delta]}(1-\rho_n^2)^2\}$ is bounded away from zero, then we may argue as above to conclude that $\lim_{n\to\infty}\boB(1-\rho_n)=+\infty$, a contradiction. Otherwise, for every $\delta>0$,  up a subsequence (that depends on $\delta$), $\lim_{n\to\infty}\min_{[0,\delta]}(1-\rho_n^2)^2=0$. For some fixed $\delta>0$ to be chosen below, we consider the mentioned subsequence, which we do not relabel, and we take $\{x_n\}\subset[0,\delta]$ such that $\rho_n(x_n)=\max_{[0,\delta]}\rho_n$. Observe that $\rho_n(x_n)\to 1$ as $n\to\infty$ and
	\[\rho_n(x_n)-\rho_n(0)\leq CR\sqrt{x_n}\leq CR\sqrt{\delta},\quad\text{ for all } n.\]
 As the left hand side of the previous inequality tends to one as $n\to\infty$, choosing $0<\delta<1/(CR)^2$ leads one more time to a contradiction. The proof is now concluded.
\end{proof}

Following the ideas of \cite{jeanjean99}, we observe that $\gamma_\gc( c)$ is a nonincreasing function of $c$. Therefore, its derivative $\gamma_\gc'( c)$ exists for almost every $c\in [\gc,\sqrt{2\sigma}]$. The points of differentiability of $\gamma_\gc$ will be crucial in our arguments. For this reason
we introduce the set 
\[\boD_{\mathfrak{c}}=\{c\in (\mathfrak{c},\sqrt{2\sigma}) :\  \gamma_{\mathfrak{c}}\text{ is differentiable at }c\}.\]
As we have pointed out, 
\begin{equation}
	\label{Dcmeasure}
	|\boD_\gc|=|(\gc,\sqrt{2\sigma})|=\sqrt{2\sigma}-\gc.
\end{equation}

Now we can state the following result due to Jeanjean \cite{jeanjean99} adapted to our setting.
\begin{lemma}\label{lemma:jeanjeanestimate}
	Assume that $\W$ satisfies  \ref{h:lowerboundWbis}. Let $c\in \boD_\gc$
and let $\{c_n\}$ be an increasing sequence such that $c_n\to c$. Then there exist a sequence $\{g_n\}\subset\Gamma(\gc)$ and a constant $R=R(\gamma'_\gc(c))>0$ such that the following holds:
	\begin{enumerate}
		\item For any $t\in [0,1]$ such that $J_c(g_n(t))\geq \gamma_\gc(c)-(c-c_n)$, we have the estimate \[\|g_n(t)\|_{H^1(\R)}+\boB(g_n(t))\leq R.\]
		\item $\max_{t\in [0,1]} J_c(g_n(t))\leq \gamma_\gc(c)+(-\gamma'_\gc(c)+2)(c-c_n).$
	\end{enumerate}
\end{lemma}
\begin{proof}
	The proof is exactly as the one of \cite[Proposition 2.1]{jeanjean99}. We only point out that, in our case, we conclude using the coerciveness of $\boA$. We also stress that the estimate $\boB(g_n(t))\leq R$ follows directly from the proof of \cite[Proposition 2.1]{jeanjean99}, as it is also observed in \cite[Lemma 4.5]{ruiz-bellazzini}.
\end{proof}

The next result is also mainly due to Jeanjean \cite{jeanjean99}, but some crucial modifications are needed since $J_c$ is not of class $\boC^2$ in the whole space $H^1(\R)$. Thus we adapt the proof thanks to Lemmas~\ref{lemma:deformation} and \ref{lemma:singularity}.

\begin{proposition}\label{prop:key}
	Assume that $\W$ satisfies  \ref{h:lowerboundWbis}. Let  $c\in \boD_\gc$, let $R_c\coloneqq R(\gamma'_\gc(c))>0$ be given by Lemma~\ref{lemma:jeanjeanestimate} and let $\delta_c\coloneqq \delta(R_c)\in (0,1)$ be given by Lemma~\ref{lemma:singularity}.  For any $\alpha>0$ and $\delta\in (0,1)$, let us consider the set
	\[Y_{\alpha,\delta}= J_c^{-1}([\gamma_\gc(c)-\alpha,\gamma_\gc(c)+\alpha])\cap Z_\delta,\]
	where $Z_\delta$ is defined by \eqref{Z}. Then, $Y_{\alpha,\delta_c}$ is nonempty for every $\alpha>0$. Moreover,
	\begin{equation}
	\label{positiveinf}
		\inf\{\|J'_c(v)\|_{H^{-1}(\R)}: v\in Y_{\alpha,\delta_c}\}=0,\quad\text{ for all } \alpha>0.
	\end{equation}
\end{proposition}

\begin{proof}
	We first show that $Y_{\alpha,\delta_c}\not=\emptyset$ for any $\alpha>0$. Indeed, let $\{c_n\}\subset\R$ be an increasing sequence such that $c_n\to c$, and let $\{g_n\}\subset \Gamma(\gc)$ be the sequence given by Lemma~\ref{lemma:jeanjeanestimate}. For some $n$ to be chosen later, let $t_n\in [0,1]$ be such that 
	\[J_c(g_n(t_n))=\max_{t\in[0,1]}J_c(g_n(t)).\]
	We check now that $g_n(t_n)\in Y_{\alpha,\delta_c}$. On the one hand, by definition of $\gamma_\gc(c)$, it follows that 
	\begin{equation}
	\label{lowerineq}
		J_c(g_n(t_n))\geq\gamma_\gc(c)-(c-c_n).
	\end{equation} 
	In consequence, Lemmas~\ref{lemma:jeanjeanestimate} and \ref{lemma:singularity} imply that $g_n(t_n)\in Z_{\delta_c}$. On the other hand, from Lemma~\ref{lemma:jeanjeanestimate} we deduce that 
	\begin{equation}
	\label{upperineq}
		J_c(g_n(t_n))\leq\gamma_\gc(c)+(-\gamma'_\gc(c)+2)(c-c_n).
	\end{equation}
	Hence, bearing \eqref{lowerineq} and \eqref{upperineq} in mind, $n$ can be chosen large enough so that $g_n(t_n)\in Y_{\alpha,\delta_c}$.
	
	We turn now to the proof of \eqref{positiveinf}. Arguing by contradiction, let us assume that $\inf\{\|J'_c(v)\|_{H^{-1}(\R)}: v\in Y_{2\alpha,\delta_c}\}>0$ for some $\alpha>0$. Let us denote \[X_{2\alpha,\delta_c}= J_c^{-1}((\gamma_\gc(c)-2\alpha,\gamma_\gc(c)+2\alpha))\cap Z_{\delta_c}.\]
	Of course, $\inf\{\|J'_c(v)\|_{H^{-1}(\R)}: v\in X_{2\alpha,\delta_c}\}>0$ too. By Lemma~\ref{lemma:technical} in the appendix, there exists $\tilde{\delta}\in (0,\delta_c/2)$ such that
	\[\inf\{\|J'_c(v)\|_{H^{-1}(\R)}: v\in  X_{2\alpha,\delta_c-2\tilde{\delta}}\}>0.\]
	Therefore, we immediately have that
	\[\inf\{\|J'_c(v)\|_{H^{-1}(\R)}: v\in  Y_{\alpha,\delta_c-2\tilde{\delta}}\}>0.\]
	Let us choose $\varepsilon>0$ small enough so that 
	\begin{equation*}
	\label{lowerestimateJprime}
		2\varepsilon<\alpha,\quad \|J'_c(v)\|_{H^{-1}(\R)}\geq\frac{2\varepsilon}{\tilde\delta},\quad\text{ for all } v\in  Y_{\alpha,\delta_c-2\tilde{\delta}}.
	\end{equation*}
	Thus, we may apply Lemma~\ref{lemma:deformation} with $\delta_1=\delta_c-2\tilde{\delta}$, $\delta_2=\delta_c-\tilde{\delta}$ and $\delta_3=\delta_c$, so that there exists a continuous function $h:[0,1]\times \NV \to \NV$ satisfying items \ref{def1}-\ref{def4}. In the rest of the proof we will abuse of the notation and write $h(v)=h(1,v)$ for simplicity.
	
	Recall that $\{c_n\}$ is an increasing sequence such that $c_n\to c$ and $\{g_n\}\subset \Gamma(\gc)$ is given by Lemma~\ref{lemma:jeanjeanestimate}. We show now that $h\circ g_n\in \Gamma(\gc)$ for all $n$. Notice that, since $\gamma_\gc(c)>0$, for every $v\in\NV$ with $J_c(v)\leq 0$ one may choose $\varepsilon>0$ small enough so that $J_c(v)\not\in [\gamma_\gc(c)-2\varepsilon,\gamma_\gc(c)+2\varepsilon]$. In consequence, recalling that $J_c(0)=0$ and $J_c(1-\phi_\gc)<0$, we deduce that
	\[0, 1-\phi_\gc\in\NV\setminus J_c^{-1}([\gamma_\gc(c)-2\varepsilon,\gamma_\gc(c)+2\varepsilon]).\]
	Therefore, item \ref{def2} of Lemma~\ref{lemma:deformation} implies that $h(g_n(0))=h(0)=0$ and $h(g_n(1))=h(1-\phi_{\gc})=1-\phi_{\gc}$ for all $n$. In sum, $h\circ g_n\in \Gamma(\gc)$ for all $n$.
	
	We now claim that 
	\begin{equation}
	\label{maxJ}
		\max_{t\in[0,1]} J_c(h(g_n(t)))\leq \gamma_\gc(c)-(c-c_n)
	\end{equation}
	for some $n$ large enough. In order to prove the claim, let us fix $t\in [0,1]$. On the one hand, if $J_c(g_n(t))\leq\gamma_\gc(c)-(c-c_n)$, then item \ref{def3} of Lemma~\ref{lemma:deformation} implies that $J_c(h(g_n(t)))\leq \gamma_\gc(c)-(c-c_n)$. 
	
	On the other hand, if $J_c(g_n(t))>\gamma_\gc(c)-(c-c_n)$, then Lemma~\ref{lemma:jeanjeanestimate} implies that $\|g_n(t)\|_{H^1(\R)}+\boB(g_n(t))\leq R_c$. In turn, Lemma~\ref{lemma:singularity} yields that $1-g_n(t)\geq \delta_c$ on $\R$. In particular, $g_n(t)\in Z_{\delta_c}$. Moreover, by Lemma~\ref{lemma:jeanjeanestimate}, we have  $J_c(g_n(t))\leq \gamma_\gc(c)+(-\gamma'_\gc(c)+2)(c-c_n)$. Thus, by taking $n$ large enough (independent of $t$) so that $(-\gamma'_\gc(c)+2)(c-c_n)\leq\varepsilon$, we derive that $g_n(t)\in J_c^{\gamma_\gc(c)+\varepsilon}$. Therefore, item \ref{def4} of Lemma~\ref{lemma:deformation} implies that $J_c(h(g_n(t)))\leq \gamma_\gc(c)-\varepsilon\leq \gamma_\gc(c)-(c-c_n)$.
	
	In any case, since $t$ was arbitrary, we have shown that 
	\eqref{maxJ} holds for some $n$ large enough. This contradicts the definition of $\gamma_\gc(c)$.
\end{proof}

Next result shows the existence of a bounded Palais--Smale sequence of $J_c$ in $\NV$, for  $c\in \boD_\gc$.

\begin{proposition}\label{prop:boundedPSsequence}
	Assume that $\W$ satisfies \ref{h:lowerboundWbis}. For any  $c\in\boD_\gc$, there exist $R_c>0$, $\delta_c\in (0,1)$ and a sequence $\{v_n\}\subset\NV$ such that 
	\begin{equation}\label{eq:boundedPSsequence}
	\{v_n\}\subset Z_{\delta_c}, \quad J_c(v_n)\to \gamma_\gc(c),\quad
	\text{and }\quad  \|J'_c(v_n)\|_{H^{-1}(\R)}\to 0,
	\end{equation}
	where $Z_{\delta_c}$ is defined by \eqref{Z}. Moreover, up a subsequence,
	 $v_n\rightharpoonup v$ weakly in $H^1(\R)$, for some $v\in Z_{\delta_c}$ with $v\not\equiv 0$.
\end{proposition}

Before proving this proposition, let us recall a  well-known result 
(see e.g.\ Lemma~3.3 in \cite{tintarev-fieseler}).
\begin{lemma}\label{lemma:wellknown}
	Let $\{v_n\}\subset H^1(\R)$ be a bounded sequence and let $q\in (2,\infty)$.
	Then  
 $v_n\to 0$ in $L^q(\R)$ if and only if 
 $v_n(\cdot +y_n)\wto 0$, for every sequence $\{y_n\}\subset \R$.
\end{lemma}

%\begin{lemma}\label{lemma:wellknown}
%	Let $\{v_n\}\subset H^1(\R)$ a bounded sequence. Assume that, for some $q\in [1,\infty)$, 
%	\[\sup_{j\in\Z}\int_j^{j+1}|v_n|^q\to 0.\]
%	Then $v_n\to 0$ strongly in $L^p(\R)$ for every $p\in (\max\{2,q\},\infty)$.
%\end{lemma}
We will also need the  following key lemma.
\begin{lemma}
	Let $v\in\boN\boE(\R)$ and $c>0$. Then the following identity holds
	\label{lem:J:iden}
\begin{equation}
	\label{J:iden}
		2 J_c(v) - J_c'(v)(v)= \frac{1}{2}\int_\R (\W\ast(v(2-v)) )v^2+\frac{c^2}{4}\int_\R\frac{v(2-v) v^2}{(1-v)^3}.
\end{equation}
		In particular, if $v\leq 1-\delta$ on $\R$, then 
\begin{equation}
	\label{identidad:L4}
\abs{2 J_c(v) - J_c'(v)(v)}\leq 
\max\left\{\frac{\|\wh\W\|_{L^\infty(\R)}}{2},\frac{c^2}{4\delta^3}\right\}
(2+\|v\|_{L^\infty(\R)} ) \|v\|_{L^2(\R)}\|v\|_{L^4(\R)}^2.
\end{equation}
\end{lemma}
\begin{proof}
Let $\rho=1-v$ and $\eta=1-\rho^2$, so that $\eta=v(2-v)$. From \eqref{eq:Jderivative1}, we have
\begin{align*}
2 J_c(v) - J_c'(v)(v)=&
	\frac{1}{2}\int_\R (\W\ast \eta )(\eta-2(1-v) v) -\frac{c^2}{4}\int_\R\left(\frac{\eta^2}{\rho^2}-\frac{(1-\rho^4)(1-\rho)}{\rho^3}\right)
	\\
	&= \frac{1}{2}\int_\R (\W\ast\eta)v^2+\frac{c^2}{4}\int_\R\frac{\eta v^2}{\rho^3},
\end{align*}
which is \eqref{J:iden}. Finally, the right-hand side
can be bounded from above by
\begin{align*}
	&\max\left\{\frac{\|\wh\W\|_{L^\infty(\R)}}{2},\frac{c^2}{4\delta^3}\right\}\|\eta\|_{L^2(\R)}\|v\|_{L^4(\R)}^2,
%	&\leq \max\left\{\frac{\|\wh\W\|_{L^\infty(\R)}}{2},\frac{c^2}{4\delta^3}\right\}\|1+\rho\|_{L^\infty(\R)}\|1-\rho\|_{L^2(\R)}\|v\|_{L^4(\R)}^2, 
\end{align*}
which, using that $\|\eta\|_{L^2(\R)}\leq \|2-v\|_{L^\infty(\R)}\|v\|_{L^2(\R)}$, yields \eqref{identidad:L4}.
\end{proof}

\begin{proof}[Proof of Proposition~\ref{prop:boundedPSsequence}]
	Applying Proposition~\ref{prop:key} with $\alpha=1/n$, we deduce the existence of $R_c$, $\delta_c$ and a sequence   $\{w_n\}\subset\NV$ satisfying 
	 \eqref{eq:boundedPSsequence}.
Since $\{w_n\}$  is  bounded  in $H^1(\R)$ and $\gamma_\gc(c)\neq 0$,	 we infer from \eqref{identidad:L4} that  
  $w_n\not\to 0$ strongly in $L^4(\R)$. Then we deduce from Lemma~\ref{lemma:wellknown}
the existence of  a  sequence $\{y_n\}\subset \R$ such that $ v_n:= {w_n(\cdot +y_n)}$ does not converge weakly to $0$ in $H^1(\R)$. Of course, $\{{v}_n\}\subset  Z_{\delta_c}$ still satisfies \eqref{eq:boundedPSsequence}. Thus, there exists a subsequence (that we do not relabel) such that $v_n\rightharpoonup v$ in $H^1(\R)$ for some $v\not\equiv 0$. In addition, since the $\|\cdot\|_{H^1(\R)}$-norm is weakly lower semicontinuous, then $\|v\|_{H^1(\R)}\leq R_c+1-\delta_c$. Moreover, since $v_n\to v$ pointwise on $\R$, it follows that $v\leq 1-\delta_c$ on $\R$, so that $v\in Z_{\delta_c}$. 
\end{proof}

Finally, we proceed to prove Theorem~\ref{thm:existenceae}, that establishes the existence of solitons 
for a.e.\ $c\in (0,\sqrt {2\sigma}).$ 

\begin{proof}[Proof of Theorem \ref{thm:existenceae}]
		Let us fix $c\in \boD_\gc$ and let $\{v_n\}\subset H^1(\R)$ be the sequence given by Proposition~\ref{prop:boundedPSsequence}. We set as usual  $\rho_n=1-v_n$ and $\eta_n=1-\rho_n^2$. Then, by Remark~\ref{rem:eq:v},
	\[J_c'(v_n)(\phi)=\int_\R v_n'\phi'+\int_\R (\W\ast \eta_n)\rho_n\phi -\frac{c^2}{4}\int_\R\frac{1-\rho_n^4}{\rho_n^3}\phi\to 0,\]
for all $\phi\in \boC_0^\infty(\R)$, and $v_n\rightharpoonup v$ weakly in $H^1(\R)$ for some $v\in Z_\delta$ with $v\not\equiv 0$.
Observe that, by the Sobolev embedding \eqref{sobolev:sharp}, 
${\eta_n}$ and $1-\rho_n^4$ are also bounded in $H^1(\R)$.
 Thus, by setting $\rho=1-v$ and 
$\eta=1-v^2$,
we deduce by the uniqueness of the limit that, up to subsequences, 
\begin{align}
		\label{conv:weak}
\eta_n \rightharpoonup \eta
\quad\text{and}\quad 
1-\rho_n^4\rightharpoonup 1-\rho^4 
 \quad &\text{ in }H^1(\R),\\
 	\label{conv:loc}
 v_n\to v
 \quad\text{and}\quad 
  \quad \rho_n\to \rho,\quad &\text{ in }L^\infty_{\text{loc}}(\R).
\end{align}

We will show now that  we can pass to the limit the three integral terms. The first one is trivial. For the second term, observe that by \eqref{W-22}, the convolution is continuous
in $L^2(\R)$, which combined with \eqref{conv:weak} and the Rellich theorem
 implies that, 
  \begin{equation*}
% 	\label{lemma:localconvergence}
 	\W\ast \eta_n\to \W\ast \eta \quad \text{ in }L^\infty_{\text{loc}}(\R).
 \end{equation*}
     Since $\phi\in \boC_0^\infty(\R)$, using also \eqref{conv:loc}, we conclude that   $(\W\ast\eta_n)\rho_n\phi\to (\W\ast\eta)\rho \phi$ in $L^1(\R)$.
     
For the last term, using that  $(1-v_n)^3\geq\delta^3$ on $\R$ for all $n$, we similarly deduce that 
$$
\int_\R\frac{1-\rho_n^4}{\rho_n^3}\phi\to 
\int_\R\frac{1-\rho^4}{\rho^3}\phi.
$$
	
	Gathering all together, we have proved that $J'(v)=0$, so that $\rho=1-v\in 1+H^1(\R)$ is a nontrivial positive solution to \eqref{eq:rho}. Moreover, by elliptic regularity, we infer that  $\rho\in\boC^\infty(\R)$. Therefore, by virtue of Proposition~\ref{prop:hydrodynamic}, it follows that $u=\rho e^{i\theta}$, with $\theta$ defined by \eqref{eq:deftheta}, belongs to $\boN\boE(\R)$ and is a nontrivial solution to \eqref{TWc}. 
	
%	{\color{blue} Quiz\'a esta puede ser una mejor notaci\'on, habr\'ia que introducirla antes: Let us consider the set
%		\[\boD_{\mathfrak{c}}=\{c\in (\mathfrak{c},\sqrt{2\sigma}):\,\,\exists \gamma'_{\mathfrak{c}}(c)\}.\]

In conclusion, we have shown that there exists a solution $u\in\boN\boE(\R)$ to \eqref{TWc} for any $c\in\boD_\mathfrak{c}$.  Thus, the same holds for every $c\in\boD$, where \[\boD:=\bigcup_{\mathfrak{c}\in (0,\sqrt{2\sigma})} \boD_\mathfrak{c}\subset (0,\sqrt{2\sigma}).\]  
%Since $\gamma_\mathfrak{c}$ is nonincreasing, then the measure of $\boD_c$ is 
% $|\boD_\mathfrak{c}|=\mathfrak{c}-\sqrt{2\sigma}$, so that 
By \eqref{Dcmeasure} we have
\[\sqrt{2\sigma}-\mathfrak{c}=|\boD_\mathfrak{c}|\leq |\boD|\leq \sqrt{2\sigma},\quad\text{for all } \mathfrak{c}\in (0,\sqrt{2\sigma}).\]
Taking limits as $\mathfrak{c}\to 0$, we conclude that   $|\boD|=\sqrt{2\sigma}$,
which proves the theorem.
\end{proof}

\begin{proof}[Proof of Corollary \ref{cor:condition12}]
Since  $\wh \W\geq 0$ a.e.\  and $\wh \W(\xi)	\geq \sigma- \xi^2/2$, for a.e.\ $\abs{\xi}\leq \sqrt{2\sigma}$,  it is immediate to check that for every $\tilde \sigma\in (0,\sigma)$, we have
$\wh \W(\xi)	\geq \tilde \sigma- \kappa_\sigma \xi^2/2$, for a.e.\ $\xi\in\R$, where $\kappa_\sigma=\tilde\sigma/(2\sigma)$. By Theorem~\ref{thm:existenceae}, we conclude the existence for a.e.\ speed in
every the interval $]0, \sqrt{2\tilde \sigma}[$, for any $\tilde \sigma\in (0,\sigma)$, which yields the existence 
for a.e. $c\in (0, \sqrt{2 \sigma})$.
\end{proof}

%%%%%%%%%%%%%%%%%%%%%%%%%%%%%%%%%%%%
\section{Existence in the whole subsonic regime}
\label{sec:allc}
Our next goal is to provide conditions on $\W$ in order to extend Theorem~\ref{thm:existenceae} and conclude the existence of solution for every subsonic speed. For this reason, we introduce the following assumptions.

\begin{enumerate}[label=(H\arabic*),ref=\textup{({H\arabic*})}]
	\setcounter{enumi}{1}
	\item \label{h:derivative}
	$\wh\W\in W^{1,\infty}(\R)$. In addition either $\wh\W\in W^{2,\infty}(\R)$, or the map $\xi\mapsto \xi\big(\wh\W\big)'(\xi)$ is bounded and continuous a.e.\ on $\R$. 
	\item 
	\label{h:restrictive} $\wh\W \in W_{\loc}^{1,\infty}(\R)$  and there exists $m\in [0,1)$ such that $\big(\wh\W\big)'(\xi)\geq-m\xi$ for a.e.\  $\xi>0$. Moreover,  $\wh\W(0)=1$ and $\widehat{ \mathcal  W}\geq 0$ on $\mathbb R$. 
	\item 
	\label{W:infty} 
	$\W$ is given by a (signed) finite Borel measure.  In particular, there is a constant $\norm{\W}$  such that 
$\norm{\W*f}_{L^p(\R)}\leq \norm{\W} \norm{f}_{L^p(\R)},$ for all $f\in {L^p(\R)},\ p\in [1,\infty]$. 
%	In other words, $f\mapsto \W*f$ is a bounded linear operator from 
%	$L^p(\R)$ to itself.
\item\label{W:apriori} 
There exists a continuous function $V_0: (0,\sqrt{2})\to (0,\infty)$ such that 
for any $u\in \boN\boE(R)$ solution to \eqref{TWc}, with  $c\in (0,\sqrt{2})$,  we have
$\norm{u}_{L^\infty(\R)}\leq V_0(c)$.
\end{enumerate}

			Recall that we are always assuming that \ref{H0} holds.
	Observe that if  $\wh\W \in W_{\loc}^{1,\infty}(\R)$, we can assume that $\wh  \W$ is continuous by the Sobolev embedding theorem, so that the condition $\wh \W(0)=1$ in \ref{h:restrictive} is meaningful. By integration, we also deduce that	if  \ref{h:restrictive} holds, then 
	\begin{equation}\label{desigualdad-W}
		\wh \W(\xi)\geq 1-m\xi^2/2, \quad \text{ for all }\xi\in\R.
	\end{equation}	
	In particular, \ref{h:lowerboundWbis} is satisfied with $\sigma=1$ and $\kappa=m/2$, and  Theorem~\ref{thm:existenceae} gives the existence of solitons for a.e.\ $c\in(0,\sqrt 2)$.
	Finally,  let us remark that the condition $\wh \W(0)=1$ is only a normalization of the potential, so that the speed of sound is from now on fixed and equal to $c_*(\W)=\sqrt 2$. Indeed, if $\wh\W(0)\not=0$, making a change of variable we can replace $\wh \W(\xi)$ with $\wh \W(\xi)/\wh \W(0)$ as in \cite{delaire-mennuni}, which gives the normalization.

%{\color{blue}\st{Indeed, if $\wh\W$ is increasing  on $\R^+$, then the limit $\lim_{\xi\to 0^+}\wh \W(\xi)$ exists.
%Since we always assume that $\wh \W$ is even, this yields that we can suppose that $\wh \W$
%is continuous at the origin. Therefore,} }

%\begin{equation}
%	\label{W:min}
%	\min_{\xi\in\R}\wh\W(\xi)=1,
%\end{equation}  
%so that \ref{h:lowerboundWbis} is satisfied with $\sigma=1$ and $\kappa=0$. In particular, if $\W$ satisfies \ref{H0} and \ref{h:restrictive}, Theorem~\ref{thm:existenceae} gives the existence of solitons for a.e. $c\in(0,\sqrt 2)$.

Concerning  \ref{W:infty},
invoking   the results in  \cite[Section 2]{grafakos} and \cite{folland},
we see that if $f\mapsto \W*f$ is a bounded linear operator from 
$L^1(\R)$ to itself, then $\W$ is given by a  finite Borel measure $\mu$, i.e.
\begin{equation}
	(\W*f)(x)=\int_\R f(x-y)d\mu(y), \quad \text{for all }f\in L^p(\R), \ p\in[1,\infty].
\end{equation}
Thus $\norm{\W}=\int_\R d\abs{\mu (y)}$ and \eqref{W:even} holds for $f\in L^p(\R)$, $g\in L^{p'}(\R)$. In addition, $\wh \mu$ is continuous, with 
$\wh \mu(\xi)=\int_\R e^{-ix\xi} d\mu(x)$. Consequently, if \ref{h:restrictive} also holds, 
then 
\begin{equation}
	\label{Wconv1}
	\wh \W(0)= \wh \mu(0)=\W*1=1.
\end{equation}

Now we can state our main theorem concerning the existence of solitons  for every subsonic speed.
\begin{theorem}
	\label{thm:existenceallc}
	Let $c\in (0,\sqrt{2})$ and assume that $\W$ satisfies   \ref{h:derivative}, \ref{h:restrictive}, \ref{W:infty} and \ref{W:apriori}. Assume in addition that $mV_0(c)^2<1$, where $m$ and $V_0(c)$ are given by \ref{h:restrictive} and \ref{W:apriori} respectively. Then there exists a nontrivial solution $u\in\boN\boE(\R)$ to \eqref{TWc}.
\end{theorem}

Notice that assumption \ref{W:apriori} can be seen as an alternative way of imposing 
that the  equation \eqref{TWc} satisfies some type of maximum principle. Clearly, 
given a potential $\W$, this is the only hypothesis difficult to verify. 
Remark that  if one we can show  the existence of a constant $C>0$ such that 
any solution $u\in \boN\boE(R)$  to \eqref{TWc}, with  $c\in (0,\sqrt{2})$, satisfies
$\norm{u}_{L^\infty(\R)}\leq C$, then \ref{W:apriori} holds true.

In Proposition~\ref{prop:universalestimate}, we will prove that \ref{W:apriori}
holds for a potential of the form \eqref{Wf}. Moreover, in this case 
$(\wh \W_\mu)'=A_\mu(\wh \mu)'$, so that Theorem~\ref{thm:existence:Wf} follows immediately from
Theorem~\ref{thm:existenceallc}  by taking $m=0$ in \ref{h:restrictive}.

As explained in the introduction, for the proof of Theorem~\ref{thm:existenceallc},
we can take $c\in(0,\sqrt2)$ and apply Theorem \ref{thm:existenceae} to get the existence of  a sequence of speeds $\{c_n\}$ and a sequence of associated solutions $\{v_n\} \in \boN\boE(\R)$ to $(S(\W,c_n))$ such that $c_n\to c$. To conclude that $\{v_n\}$ converges to a {\em finite energy}
solution $v$ to \eqref{TWc}, we need to obtain uniform  estimates for $\{v_n\}$, 
and to get a more precise information of $v$, using the fact that each $v_n$ is the limit of a Palais--Smale sequence for $J_c$. We deal with these problems in the following subsections.

% {\color{blue} Theorem que estar\'a en la introducci\'on. Algo parecido a lo siguiente tambi\'en puede ir en la intro: Essentially, the idea is to fix any $c\in (0,\sqrt{2})$ and consider a sequence $\{c_n\}\subset (0,\sqrt{2})$ such that $c_n\to c$ and for which there exists a finite energy solution to (TW$_{\W,c_n}$). Roughly speaking, once one has estimates \emph{on solutions} of the type \eqref{eq:universalestimate}, then it is a simple task to pass to the limit. However, the energy may blow up and has to be controlled.}

%The a priori estimates in the estimates in Proposition~\ref{prop:universalestimate}, 
%will be key in Theorem~\ref{thm:existence:Wf}. From the proof of 
%Proposition~\ref{prop:universalestimate},
%we see that the main tool is the use of the maximum principle to obtain the bound on $\norm{u}_{L^\infty(\R)}$. 

\subsection{Uniform estimates}

We start by recalling  a Pohozaev identity that was proved by the first author in \cite{delaire2009} in a more general framework.
\begin{proposition}\label{prop:pohozaev}
	Let $c\geq 0$ and assume that $\W$ satisfies \ref{H0} and \ref{h:derivative}.
		Let  $u\in\boE(\R)$ be a solution to $S({\W,c})$. Then 
	\begin{equation}\label{eq:pohozaev}
		\int_\R \abs{u'}^2=\frac{1}{4\pi}\int_\R\left(\wh\W(\xi)-\xi\big(\wh\W\big)'(\xi)\right)|\wh\eta(\xi)|^2 d\xi,	
	\end{equation}
	where $\eta=1-\abs{u}^2$.
\end{proposition}

\begin{proof}
Let us remark that since $\eta \in H^1(\R)$, we have 
$$
\int_{\R}\abs{\xi} \abs{\wh \eta(\xi)}^2d\xi \leq \int_{B(0,1)} \abs{\wh \eta(\xi)}^2d\xi 
+\int_{B(0,1)^c} \xi^2 \abs{\wh \eta(\xi)}^2d\xi \leq 2\pi\norm{f}_{H^1(\R)}^2,
$$
so that $\xi\abs{\wh\eta(\xi)}^2\in L^1(\R)$ and  the integral in \eqref{eq:pohozaev} is well-defined when $(\wh\W\big)'\in L^\infty(\R).$
		
In the case that  the map $\xi\mapsto \xi\big(\wh\W\big)'(\xi)$ is bounded and continuous a.e.\ on $\R$, identity \eqref{eq:pohozaev}  is given by Propositions 5.1 and 5.3 in \cite{delaire2009}, taking $N=1$. 
	
Let us suppose now that $\wh\W\in W^{2,\infty}(\R)$. By invoking Corollary~\ref{cor:algebraicdecay}, with $s=2$ and $\ell=1$, we deduce that 	$|\cdot| \eta, |\cdot| \eta' \in L^2(\R)$. Then we can 
multiply $S({\W,c})$ by a test function, integrate by parts and apply the dominated convergence theorem, as explained in \cite{delaire2009}, p.\ 1473, to obtain \eqref{eq:pohozaev}.
\end{proof}

%\begin{remark}
%\label{rem:W:neg}
Notice that Proposition \ref{prop:pohozaev} can be applied to the focusing case $\W=-\delta_0$, and it is immediate to deduce that then the only finite energy solution to $S({\W,c})$ are the constants, since $\wh \W=-1$.

More general examples can be constructed by 
defining an even function such that $$\wh \W(\xi;g)=-\xi \int_\xi^\infty \frac{g(y)}{y^2}dy,\quad \text{ for }\xi>0,$$
with $g$ a bounded nonnegative function, since $\wh\W(\xi;g)-\xi\big(\wh\W\big)'(\xi;g)\leq 0$   a.e.\ on $\R$. For instance, if $g=1$, then we obtain $\wh \W(\xi;g)=-1$. Consequently,
we can construct potentials with arbitrary $L^1$-norm, such that the only finite energy solution to \eqref{TWc} are the trivial ones.

\begin{corollary}
	For any  $\alpha>0$, there exists a  function   $\W_\alpha\in \boC^\infty(\R)$ satisfying \ref{H0} and   \ref{h:derivative}, with $\norm{\W_\alpha}_{L^1(\R)}=\alpha$ and $\wh \W_\alpha \leq  0$ on $\R$, such that if $u\in \boE(\R)$ is a solution to $(S(\W_\alpha,c))$ for some $c\geq 0$, then $u$ is constant.
\end{corollary}
\begin{proof}
	Let us take $g(y)= y^3\exp(-y^2)$, which gives
	$\wh \W(\xi)=-\abs{ \xi} e^{-\xi^2}$, for $\xi\in\R$. In this manner, 
	$\W$ is a smooth function with exponential decay, so that it suffices to consider 
	$\W_\alpha=\alpha \W/\norm{\W}_{L^1(\R)}.$
\end{proof}

%\end{remark} 

Identity \eqref{eq:pohozaev}, together with \eqref{eq:ellipticequation} and \eqref{eq:localequation}, allow us to prove the following nonvanishing property of   nontrivial finite energy  solutions to \eqref{TWc}.

\begin{proposition}\label{prop:nonvanishing}
	Let $c\geq 0$. Assume that $\W$ satisfies \ref{h:derivative} and that
	\begin{equation}
		\label{assumption:Linfty}
3\xi^2+2\wh \W(\xi)+2\xi (\wh \W)'(\xi)\geq 2 \quad \text{  a.e.\ on } \R.
	\end{equation}
	Then every nontrivial solution $u\in\boE(\R)$ to \eqref{TWc} satisfies the estimate
	\begin{equation}\label{eq:nonvanishing}
		\|\W\ast\eta\|_{L^\infty(\R)}\geq\frac{2-c^2}{4}.
	\end{equation}

			In particular, if  $\W$ satisfies \ref{h:restrictive}, then \eqref{eq:nonvanishing} holds.
		
	%	As a direct consequence,
	%	\begin{equation}\label{eq:nonvanishing2}
	%		\|\eta\|_{H^1(\R)}\geq\frac{2-c^2}{4\|\wh\W\|_{L^\infty(\R)}}.
	%	\end{equation}
\end{proposition}
%%%%%%%%%%%%%
\begin{proof}
	Multiplying \eqref{eq:ellipticequation} by $2\eta'$ and integrating by parts leads to
	\[2\int_\R (\eta')^2 + 4\int_\R \eta(\W\ast\eta) - 2c^2\int_\R \eta^2 = 4\int_\R \abs{u'}^2\eta + 4\int_\R \eta^2(\W\ast\eta).\]
	Using now \eqref{eq:localequation}, we deduce that
	\[3\int_\R (\eta')^2 + 4\int_\R \eta(\W\ast\eta) - c^2\int_\R \eta^2 = 4\int_\R \abs{u'}^2 + 4\int_\R \eta^2(\W\ast\eta).\]
	Combining with \eqref{eq:pohozaev} and applying Plancherel's identity, we derive
	\[3\int_\R (\eta')^2 + 2\int_\R \eta(\W\ast\eta) - c^2\int_\R \eta^2 + \frac{1}{\pi}\int_\R \xi\big(\wh\W\big)'(\xi)|\wh\eta(\xi)|^2 d\xi = 4\int_\R \eta^2(\W\ast\eta).\]
	Again by Plancherel's identity, this equality can be recast as
	$$
	\frac1{2\pi}\int_\R \big( 3\xi^2+2\wh \W(\xi)+2\xi (\wh \W)'(\xi)\big) |\wh\eta(\xi)|^2 d\xi  - c^2\int_\R \eta^2
	= 4\int_\R \eta^2(\W\ast\eta).
	$$
	Therefore, inequality \eqref{assumption:Linfty} implies that
	\[(2-c^2)\int_\R\eta^2\leq 4\int_\R\eta^2(\W\ast\eta)\leq 4\|\W\ast\eta\|_{L^\infty(\R)}\int_\R \eta^2.\]
	Thus  result  \eqref{eq:nonvanishing}  follows by taking into account that $\eta$ is nontrivial, i.e.\ $\norm{\eta}_{L^2(\R)}>0$.

		Finally, we remark that if \ref{h:restrictive} is satisfied, then  \eqref{desigualdad-W} holds true, which 
	implies that
		\[3\xi^2+2\wh\W(\xi)+2\xi\big(\wh\W\big)'(\xi)-2\geq 3\xi^2+2(1-m\xi^2/2)-2m\xi^2-2=3(1-m)\xi^2\geq 0.\]
	This completes the proof.
\end{proof}

The next proposition shows that the potential in  \eqref{Wf} satisfies \ref{W:apriori}.

{
\begin{proposition}\label{prop:universalestimate}
	Let $c>0$ and assume that $\W_\mu=A_\mu(\delta_0+\mu)$ is as in  \eqref{Wf}. Then, 
	%	$\wh f(0)>-1$,  $A=\frac{1}{1+\wh f(0)}$ and,
	for every solution $u\in\boE(\R)$ to \eqref{TWc}, the following estimates hold:
	\begin{align}
		\label{eq:universalestimate}
		\|u\|_{L^\infty(\R)}^2 &\leq B_0(\mu)\Big(1+\frac{c^2}{4}\Big),
		\\
		\label{eq:universalestimateprime}
		\|u'\|_{L^\infty(\R)} &\leq B_1(\mu)\Big(1+\frac{c^2}{4}\Big)^2,
	\end{align}
	where $B_0(\mu)=1+\frac{\norm{ \mu^+}}{1-\norm{ \mu^-}}$  and $B_1(\mu)$
	is a constant depending only on  $\norm{\mu^+}$ and $\norm{\mu^-}$.
	Moreover, for any $k\geq 2$ there is a constant $C_k(c)>0$, depending only on $c$ and $k$, 
	and $B_k(\mu)>0$, depending only on $\norm{\mu^+}$, $\norm{\mu^-}$ and $k$, 
	such that
	\begin{equation}
		\label{der:higher}
		\|D^k u\|_{L^\infty(\R)}\leq B_k(\mu)C_k(c).
	\end{equation}
\end{proposition}
}
\begin{proof}
	%	We first observe that 
	%	$1=\widehat{\W}(0)=A(1+\widehat{f}(0))$, so $\wh f(0)>-1$ and $A=\frac{1}{1+\widehat{f}(0)}$. 
	{ 
	Since $c>0$, by Propositions \ref{prop:etakidentities} and \ref{prop:hydrodynamic}, the function 	
	$\rho=|u|$ satisfies the equation \eqref{eq:rho}. 
	By using Young's inequality and  the fact that  $\mu\ast 1=\int_\R d\mu(x)=\widehat{\mu}(0)$, we estimate the term on the right-hand side of the equation
	as follows, where we drop the subscript $\mu$ for simplicity,
	\begin{align*}
		\W*(1-\rho^2)&=A(1-\rho^2)+A \mu\ast(1-\rho^2)=A(1-\rho^2)+A\widehat{\mu}(0)-A\mu\ast(\rho^2)
		\\
		&=1-A\rho^2-A \mu\ast(\rho^2)=1-A\rho^2-A \mu^+\ast(\rho^2)+A \mu^-\ast(\rho^2),
	\end{align*}		
	where we used that 	$A(1+\widehat{\mu}(0))=1$. Therefore 		
	\begin{equation}\label{est1}	\W*(1-\rho^2)	\leq 1-A\rho^2+A(\mu^-)\ast(\rho^2)\leq 1-A\rho^2+A\|\mu^-\|\|\rho\|^2_{L^\infty(\R)},
\end{equation}
	and 
	\begin{equation}\label{est2}
	\norm{\W*(1-\rho^2)}_{L^\infty(\R)}	\leq 1+A\norm{\rho}_{L^\infty(\R)}^2+A \|\mu^-\|\|\rho\|^2_{L^\infty(\R)}.
\end{equation}
	Plugging \eqref{est1} into \eqref{eq:rho} leads to
	\[-\rho''+\rho\Big(A\rho^2-1-\frac{c^2}{4}-A\|\mu^-\|\|\rho\|^2_{L^\infty(\R)}\Big)\leq 0\quad\text{on }\R.\]
	By applying the maximum principle or Proposition 2.1 in \cite{farina}, we conclude that 
	\begin{equation}\label{eq:cota}
	\rho(x)^2\leq\frac{1}{A}\big(1+\frac{c^2}{4}\big)+\|\mu^-\|\|\rho\|^2_{L^\infty(\R)},
	\quad \text{for all }x\in\R.
	\end{equation}
}
	Let us assume that there exists some $\bar{x}\in\R$ such that $\rho(\bar{x})>1$; otherwise, the result is trivial.
	Since $\rho(\pm\infty)=1$, there exists $\tilde{x}\in\R$ such that $\rho(\tilde{x})=\|\rho\|_{L^\infty(\R)}$. Thus, using \eqref{eq:cota} in $\tilde{x}$, we get
	\[(1-\|\mu^-\|)\|\rho\|^2_{L^\infty(\R)}\leq\frac{1}{A}\Big(1+\frac{c^2}{2}\Big),\]
	which proves \eqref{eq:universalestimate}.

	In order to establish \eqref{eq:universalestimateprime}, we follow \cite{farina,bethuel}  and define $v(x)= u(x) e^{\frac{ic}{2}x}$, for $x\in\R$. It is immediate to verify that $v\in\boE(\R)$ and that it solves the equation
	\begin{equation}
		\label{eq:simply}
		-v''=\Big(\frac{c^2}{4} +\W\ast(1-\abs{v}^2)\Big)v\quad\text{ on }\R.
	\end{equation}
	From \eqref{eq:universalestimate} and \eqref{est2}, it follows that
	\begin{align*}
		\|v''\|_{L^\infty(\R)} &\leq \norm{v}_{L^\infty(\R)}  \Big(\frac{c^2}{4}+1+A\norm{\rho}_{L^\infty(\R)}^2+A \|\mu^-\|\|\rho\|^2_{L^\infty(\R)} \Big)
		\\
		&\leq \Big(1+\frac{c^2}{4}\Big)^\frac{3}{2}B_0^{1/2}\big(
		1+AB_0+AB_0\norm{\mu^-}).
	\end{align*}
	Recalling that $A(1+\|\mu^+\|-\|\mu^-\|)=1$, it is clear that
	\[\|v''\|_{L^\infty(\R)}\leq 2\Big(1+\frac{c^2}{4}\Big)^\frac{3}{2}B_0^{1/2}.\]
	Thus, using the Landau--Kolmogorov interpolation inequality (see e.g.\ p.133 in\cite{burenkov})
	\begin{equation*}
		\|v'\|_{L^\infty(\R)}\leq \sqrt{2}	\|v\|_{L^\infty(\R)}	\|v''\|_{L^\infty(\R)},
	\end{equation*}
	we infer that 
	\begin{equation*}
		\|v'\|_{L^\infty(\R)}\leq 2\sqrt{2}B_0 \Big(1+\frac{c^2}{4}\Big)^2.
	\end{equation*}	
	Therefore, by definition of $v$ and using that $c/2\leq 1+c^2/4$, we deduce that
	\[\|u'\|_{L^\infty(\R)}\leq \frac{c}{2}\|u\|_{L^\infty(\R)}+\|v'\|_{L^\infty(\R)}\leq  \left(1+\frac{c^2}{4}\right)^2 B_0^{1/2}\big(1+ 2\sqrt{2} B_0^{1/2}\big).
	\]
	Hence, taking $B_1(\mu):=B_0^{1/2}\big(1+ 2\sqrt{2} B_0^{1/2}\big)$, we have
	\eqref{eq:universalestimateprime}. Differentiating \eqref{eq:simply} and using 
	the higher order  Landau--Kolmogorov inequalities, we finally conclude the proof of \eqref{der:higher}.
\end{proof}

Next two propositions show that, for general potentials satisfying the continuity property \ref{W:infty}, an $L^\infty$ estimate for the solutions (i.e. condition \ref{W:apriori}) implies  a priori estimates also for the derivatives as well as a uniform lower bound.

\begin{proposition}
	\label{est:derivative}
	Assume that $\W$ satisfies  \ref{W:infty} and \ref{W:apriori}. Then, for every $k\in \N$,
	there exist  continuous functions $V_k: (0,\sqrt{2})\to (0,\infty)$  such that
	for any $u\in \boN\boE(\R)$ solution to \eqref{TWc}, with  $c\in (0,\sqrt{2})$,  we have
	$\norm{D^k u}_{L^\infty(\R)}\leq V_k(c)$.  { In particular, 	if  $\W=\W_\mu$ is given by  \eqref{Wf},  then $V_1(c)=B_1(\mu)(1+{c^2}/{4})^2$, where $B_1(\mu)$ is the constant in Proposition~\ref{prop:universalestimate}.}
\end{proposition}
%%%%%%%%%%%%%
\begin{proof}
	By using \eqref{eq:simply}	and   \ref{W:infty}, the proof follows the same line
	as  Proposition~\ref{prop:universalestimate}. 
\end{proof}

\begin{proposition}\label{prop:lowerbound}
	Assume that $\W$ satisfies  \ref{W:infty} and \ref{W:apriori}.
	Let $c\in(0,\sqrt 2)$	and let $u\in\boN\boE(\R)$ be a solution to \eqref{TWc}.
	Then 
\begin{equation}
	\label{cota:minora}
|u(x)|\geq \frac{\sqrt{1+4c^2/V_1(c)}-1}{\sqrt{1+4c^2/V_1(c)}+1},\quad \text{for all }x\in\R,
\end{equation}	
	where $V_1$ is the function given by Proposition~\ref{est:derivative}.
\end{proposition}

\begin{proof}
	Since $u\in\boN\boE(\R)$, we have  that $\min_{\R} \abs{u}>0$. Let $x_0\in \R$ be such that 
	$u(x_0)=\min_{\R} \abs{u}$. From the identity \eqref{eq:localequation}, we deduce that
	the function $\eta=1-\abs{u}^2$ satisfies
	\begin{equation}\label{ineqn}
		c^2\eta(x_0)^2\leq \norm{u'}_{L^\infty(\R)}(1-\eta(x_0)).
	\end{equation}
	By using the estimate in Proposition \ref{est:derivative}, we get 
	$$c^2\eta(x_0)^2 +V_1(c) \eta(x_0)-V_1(c) \leq 0,$$
	which implies that 
	$$
	\eta(x_0)\leq \frac{-V_1(c)  +\sqrt{V_1(c) ^2+4V_1(c) c^2}}{2c^2}.
	$$
	In terms of $|u(x_0)|$ we get $$
	\abs{u(x_0)}^2\geq 1+\frac{V_1(c)  -\sqrt{V_1(c) ^2+4V_1(c) c^2}}{2c^2}=\frac{\sqrt{V_1(c) ^2+4V_1(c) c^2}-V_1(c) }{\sqrt{V_1(c) ^2+4V_1(c) c^2}+V_1(c) },
	$$
	which completes the proof.
\end{proof}

The following nonvanishing property of the functional $\boA$ will be useful.

\begin{lemma}
	\label{lem:cota:A}
	Assume that $\boW$ satisfies  \ref{h:derivative}, \ref{W:infty}, \ref{W:apriori}
	and \eqref{assumption:Linfty}.
	 Then there exists $C>0$ such that for any  nonzero solution $v\in  H^1(\R)$ to 
	  	\eqref{eq:v}, we have
	\begin{equation}
		\label{cota:A}
		\boA(v)\geq \frac{C(2-c^2)^2}{16}.
	\end{equation}
\end{lemma}
\begin{proof}
	Let $\eta=1-\abs{v}^2$. Then 
	\begin{align*}
		\boA(v)&=\frac{1}{2}\int_\R(v')^2+\frac{1}{4}\int_\R(\W\ast\eta)\eta=\frac{1}{8}\int_\R\frac{(\eta')^2}{1-\eta}+\frac{1}{4}\int_\R(\W\ast\eta)\eta
		\\
		&\geq \frac{1}{8\|1-\eta\|_{L^\infty(\R)}\|\wh\W\|_{L^\infty(\R)}^2}\int_\R (\W\ast\eta')^2+\frac{1}{4\|\wh\W\|_{L^\infty(\R)}}\int_\R (\W\ast\eta)^2.
	\end{align*}
	 By using the Sobolev's embedding, \eqref{eq:nonvanishing} and \eqref{cota:minora}, we conclude that  there exists $C>0$ such that
	\[\boA(v)\geq C\|\W\ast\eta\|^2_{H^1(\R)}\geq C\|\W\ast\eta\|^2_{L^\infty(\R)}\geq\frac{C(2-c^2)^2}{16}.\]
\end{proof}
%%%%%%%%%%%%%%%%%%%%%%%%%%%%%%%%%%%%%%%%%%
%%%%%%%%%%%%%%%%%%%%%%%%%%%%%%%%%%%%%%%%%%%%%%%%

\subsection{Refined study of Palais--Smale sequences}
We start by recalling a classical result of profile decomposition of a bounded sequence, 
that is a refinement of the Banach--Alaoglu theorem. We use here the version given in Theorem~4.6.5  in  \cite{tintarev} (see also \cite{taoufik-keraani,banica-duyckaerts}).

\begin{theorem}
	\label{thm:decomp}
	Let $\{v_n\}\subset H^1(\R)$ be a bounded sequence. Then 
	there exist a family of concentration profiles $\{w_j\}\subset H^1(\R)$
	and points $\{y_{n,j}\}\subset \R$ such that, on a renumbered subsequence, 
	\begin{equation*}
%		\label{}
		y_{n,1}=0, \quad \lim_{n\to\infty}\abs{y_{n,i}-y_{n,j}}\to\infty, \text{ if }i\neq j,
	\end{equation*}
	$v_n(\cdot + y_{n,j})\wto w_{j}$ in $H^1(\R)$ and 	$v_n(\cdot + y_{n,j})\to w_{j}$ in $L^\infty_{\textup{loc}}(\R)$, 
	\begin{equation}
		\label{Sn:conv}
		v_n-S_n\to 0 \text{ in }L^q(\R), \quad \text{ where } S_n=\sum_{j=1}^\infty w_{j}(\cdot -y_{n,j}),
	\end{equation}
	for all $q\in (2,\infty)$. Moreover, the series $S_n$ converges in $H^1(\R)$ unconditionally and uniformly in $n$, and for all $\varphi \in L^2(\R)$, $\{\alpha_n\}\subset \R$  and $q\in (2,\infty)$, we have
	\begin{align}
		\label{lim:resto}
		v_n=\sum_{j=1}^k w_{n,j}+r_{n,k}, 	 \text{ with }	\lim_{k\to\infty }\limsup_{ n\to\infty}\left| \int_\R r_{n,k}(\cdot -\alpha_n)\varphi \right|= \lim_{k\to\infty }\limsup_{n\to\infty}\norm{r_{n,k}}_{L^q(\R)}=0, 
	\end{align}
	where  $ w_{n,j}=w_{j}(\cdot -y_{n,j})$. In addition,
	\begin{equation}
		\label{asym:decomp}
		\norm{D^m v_{n}}^2_{L^2(\R)}=
		\sum_{j=1}^k \norm{D^m w_{j}}_{L^2(\R)}^2+
		\norm{D^m r_{n,k} }_{L^2(\R)}^2+o_n(1), \quad \text{for }m\in\{0,1\}.
	\end{equation}
\end{theorem}

In Theorem~\ref{thm:decomp} and for the rest of the article, the notation $o_n(1)$ stands for a sequence in $\R$ such that $o_n(1)\to 0$, as $n\to\infty$. Besides, from now on $o_n(1;L^1)$ will denote a function such that $\norm{o_n(1;L^1)}_{L^1(\R)}\to 0$, 
as $n\to\infty$. 

Notice also that we added to the statement in \cite{tintarev} that $v_n(\cdot + y_{n,j})$ converges  to  $w_{j}$ in $L^\infty_{\textup{loc}}(\R)$, by invoking the Rellich theorem.

%\begin{lemma}
%\label{lemma:wjvortexless}
%	Let $\{v_n\}\subset H^1(\R)$ be a bounded sequence for which  there exists $\delta\in (0,1)$ such that $v_n\leq 1-\delta$ on $\R$, for all $n$. With the notations of  Theorem~\ref{thm:decomp}, suppose in addition that the profile decomposition 
%	is finite, i.e.
%	\begin{align}
%		\label{finite:decomp}
%		v_n=\sum_{j=1}^k w_{n,j}+r_{n}, 	\quad  \text{ for some }	k\geq1,
%	\end{align}
%	with $r_n\wto 0$ in $H^1(\R)$ and $r_n\to 0$ in $L^q(\R)$ for all $q\in (2,\infty)$. Then
%	\begin{align}
%	\label{eq:wjvortexless}
%		w_j&\leq 1-\delta,\quad\text{on }\R,\quad\text{for all }j=1,\dots,k,
%		\\
%	\label{eq:rnvortexless}
%		r_n&\leq 1-\frac{\delta}{2},\quad\text{on }\R,\quad\text{for all }n.
%	\end{align}
%\end{lemma}

We  recall now a version of the Brezis--Lieb lemma given in 
\cite{vanschaftingen-jiankang}.
\begin{lemma}
	\label{lema:decomG}
	Assume that $G\in C^1(\R;\R)$, with $G(0)=0$, and that there exist $a>0$ and $q>1$ such that 
	\begin{equation}
		\label{hyp:g}
		\abs{G'(t)}\leq a(\abs{t}+\abs{t}^q), \quad \text{ for all }t\in \R.
	\end{equation}	
	If the sequence $\{v_n\}$ is bounded in $H^1(\R)$ and converges a.e.\ to $v$, then 
	\begin{equation}
		\label{decomp:G}
		G(v_n)=G(v_n-v)+G(v)+o_n(1;L^1).
	\end{equation}
	Moreover, using the notations in Theorem~\ref{thm:decomp}, if the profile decomposition 
	is finite, i.e.\ there exists $k\geq 1$ such that
	\begin{align}
		\label{finite:decomp}
		v_n=\sum_{j=1}^k w_{n,j}+r_{n},
		\text{ with }
		 r_n\wto 0 \text{ in }H^1(\R) \text{ and }r_n\to 0 \text{ in }
		 L^q(\R) \text{ for all }q\in (2,\infty),
	\end{align}
 then
	\begin{align}
		\label{G:decomp2}
		G(v_n)=\sum_{j=1}^k G(w_{n,j})+
		G(r_{n} )+o_n(1;L^1).
	\end{align}
\end{lemma}

\begin{proof}
	The decomposition in \eqref{decomp:G} corresponds to Lemma~2.3 in \cite{vanschaftingen-jiankang}.
	To show \eqref{G:decomp2}, 
	let us denote by $\tau_{n,j}$
	the translation by $y_{n,j}$, i.e.\ $\tau_{n,j}v=v(\cdot + y_{n,j})$.
	Using that  $y_{n,1}=0$, we  have by Theorem~\ref{thm:decomp} that  $v_n\rightharpoonup w_1$ in $H^1(\R)$. Then, by \eqref{decomp:G} we obtain
	\begin{equation}
		\label{dem:Gnt}
	G(v_n )= G(w_1)+G(v_n-w_1 )+o_n(1;L^1).
	\end{equation}
	Now, again by Theorem~\ref{thm:decomp} and using that $|y_{n,2}|\to\infty$ as $n\to\infty$, we derive that  $\tau_{n,2}v_n-\tau_{n,2}w_1\to w_2$ a.e.\ on $\R$. Thus \eqref{decomp:G} implies that 
	\begin{equation*}
%		\label{dem:Gnt}
		G(\tau_{n,2}v_n-\tau_{n,2}w_1)-G(w_2)-G(\tau_{n,2}v_n-\tau_{n,2}w_1- w_2)=o_n(1;L^1).
	\end{equation*}
	Therefore, by a change of variables, 
	$$
	G(v_n-w_1)-G(w_{n,2})-G(v_n-w_1- w_{n,2})=o_n(1;L^1).
	$$
	Combining with \eqref{dem:Gnt}, we conclude that 
	$$G(v_n )= G(w_1)+G(w_{n,2})+G(v_n-w_1- w_{n,2})+o_n(1;L^1).
	$$
	By repeating the same argument $k$ times, we get \eqref{G:decomp2}.
\end{proof}

In the following lemma we deal with the splitting of the singular term  $\boB$.

\begin{lemma}
	\label{lemma:splitting:2}
	Let $\{v_n\}\subset H^1(\R)$ be a bounded sequence such that $v_n\wto v$ in $H^1(\R)$ for some $v\in H^1(\R)$. Assume that there exists $\delta\in (0,1)$ such that $v_n\leq 1-\delta$ on $\R$. Then, there is $N\in\N$ such that,
	\begin{equation}
		\label{vcontroled}
		v\leq 1-\delta,\quad v_n-v\leq 1-{\delta}/{2}\quad\text{on }\R,\,\,\text{for all }n\geq N,
	\end{equation}
	and
	\begin{equation}
		\label{B:split}
		\boB(v_n)=\boB(v_n-v)+ \boB(v)+o_n(1).
	\end{equation}
	Moreover, if the profile decomposition 
	is finite
		as in \eqref{finite:decomp}, then 
	\begin{align}
		\label{eq:wjvortexless}
		w_j&\leq 1-\delta,\quad\text{on }\R,\quad\text{for all }j=1,\dots,k,
		\\
		\label{eq:rnvortexless}
		r_n&\leq 1-\frac{\delta}{2},\quad\text{on }\R,\quad\text{for all }n\geq N,
	\end{align}
	and
	\begin{equation}
		\label{B:decomp}
		\boB(v_n)=\sum_{j=1}^k \boB(w_{j})+
		\boB(r_{n} )+o_n(1).
	\end{equation}
\end{lemma}

\begin{proof}
	We first prove \eqref{vcontroled}. Since $v_n\to v$ a.e.\ on $\R$ and  $v_n\leq 1-\delta$, it follows that $v\leq 1-\delta$. Now, since $v\in H^1(\R)$, we can fix $R>0$ such that $|v|\leq\delta/2$ a.e.\ on $\R\setminus B_R(0)$. Then, for all $n$,
	\[v_n-v\leq 1-\delta+{\delta}/{2}=1-{\delta}/{2}\quad\text{on }\R\setminus B_R(0).\]  Moreover, since $v_n\to v$ in $L^\infty(B_R(0))$, then, for any $n$ large enough, 
	\[v_n-v\leq \|v_n-v\|_{L^\infty(B_R(0))}\leq1-{\delta}/{2}\quad\text{on }B_R(0).\]
	In any case, \eqref{vcontroled} holds.
	
	We turn now to proving \eqref{B:split}. Using the notation in Lemma~\ref{lemma:Jsmooth}, we see that
	$$\boB(v)=\int_\R H(v(x))dx, \quad \text{ where } H(t)=\int_0^t h(s)ds,\quad h(s)=\frac{s(2-s)(s^2-2s+2)}{4(1-s)^3}.$$
%	We remark that there is a constant $A_\delta>0$ such that 
%	\begin{equation}
%		\label{BLcond1}
%		\abs{h(s)}\leq A_{\delta}(\abs{s}+s^4),\quad \text{for all }s\leq 1-{\delta}/{2}.
%	\end{equation}
	We remark that we can easily construct a bounded function $\chi_\delta\in \boC^1(\R)$ such that 
	\begin{equation}
 \label{xi:delta}
		\chi_\delta(s)=\frac{1}{4(1-s)^3}\quad\text{for all }s\leq 1-\frac{\delta}{2},\quad \|\chi_\delta\|_{L^\infty(\R)}\leq B_\delta,
		\end{equation}
	for some constant $B_\delta>0$ depending only on $\delta$. Then the function $\tilde{h}(s)=s(2-s)(s^2-2s+2)\chi_\delta(s)$ clearly satisfies
	\[\abs{\tilde{h}(s)}\leq C_{\delta}(\abs{s}+s^4),\quad \text{for all }s\in\R,\]
	for some $C_\delta>0$ depending only on $\delta$. Therefore, condition \eqref{hyp:g} holds for $G=\tilde{H}$, being $\tilde{H}(t)=\int_0^t \tilde{h}(s)ds$. Thus we obtain
	\[\tilde{H}(v_n)=\tilde{H}(v_n-v)+\tilde{H}(v)+o_n(1;L^1).\]
	Using now \eqref{vcontroled}, we conclude that  
	\begin{equation}
		\label{decomp:H}
		H(v_n)=H(v_n-v)+H(v)+o_n(1;L^1),	
	\end{equation}
	which gives \eqref{B:split}.
	
	Next, we prove \eqref{eq:wjvortexless} and \eqref{eq:rnvortexless}. In order to do so, let us fix $\varepsilon>0$ to be chosen later. The density of $\boC^\infty_0(\R)$ in $H^1(\R)$ implies that, for every $j=1,\dots,k$, there exist $g_j\in\boC_0^\infty(\R)$ and $\varphi_j\in H^1(\R)$ such that
	\begin{equation}
		\label{dense}
		w_j=g_j+\varphi_j,\quad \|\varphi_j\|_{L^\infty(\R)}<{\varepsilon}/{k}.
	\end{equation}
	Hence, we can take $R>0$ such that 
	$\cup_{j=1}^k\supp(g_j)\subset B_R(0).$
	Let us denote $g_{n,j}=g_j(\cdot-y_{n,j})$. It is clear that 
	\[\supp(g_{n,j})\subset B_R(y_{n,j}),\quad\text{for all }j=1,\dots,k.\]
	In particular, since  $|y_{n,i}-y_{n,j}|\to\infty$ for all $i\not=j$, there is $N\in\N$ such that, for all $n\geq N$, 
	\begin{equation}
		\label{disjoint}
		\supp(g_{n,i})\cap\supp(g_{n,j})=\emptyset\quad\text{for all }i\not=j.
	\end{equation}
	
	On the other hand, by  Theorem~\ref{thm:decomp},  $v_n(\cdot+y_{n,j})\to w_j$ a.e.\ on $\R$, so \eqref{eq:wjvortexless} follows directly from the fact that $v_n\leq 1-\delta$. Moreover,
	$v_n(\cdot+y_{n,j})\to w_j$ in $L^\infty(B_R(0))$, for all $j=1,\dots,k.$
	Thus, we may take $N$ larger if necessary in order to get, for all $n\geq N$ and for all $j=1,\dots,k$,
	\begin{equation}
		\label{sumcontrolled}
		\|v_n(\cdot+y_{n,j})-w_j\|_{L^\infty(B_R(0))}<{\varepsilon}/{k}.
	\end{equation}
	
	To show \eqref{eq:rnvortexless}, fix $x\in\R$ and $n\geq N$. Observe that
	\[r_n(x)=v_n(x)-\sum_{j=1}^k g_{n,j}(x)-\sum_{j=1}^k \varphi_j(x-y_{n,j}).\]
	Now we have two possibilities. On the one hand, if $x\not\in\supp(g_{n,j})$ for any $j=1,\dots,k$, then, using \eqref{dense}, we obtain		\[r_n(x)=v_n(x)-\sum_{j=1}^k \varphi_j(x-y_{n,j})< 1-\delta+\varepsilon.\]
	On the other hand, if $x\in \supp(g_{n,i})$ for some $i=1,\dots,k$, then $i$ is unique by virtue of \eqref{disjoint}. We may assume without loss of generality that $i=1$. Moreover, $x\in B_R(y_{n,1})$, so $z_n\coloneqq x-y_{n,1}\in B_R(0)$. Therefore, using  \eqref{dense} and \eqref{sumcontrolled}, we deduce that
	\begin{align*}
		r_n(x)&=v_n(x)-g_{n,1}(x)-\sum_{j=1}^k \varphi_j(x-y_{n,j})
		=v_n(z_n+y_{n,1})-w_1(z_n)-\sum_{j=2}^k \varphi_j(x-y_{n,j})\\
		&<\|v_n(\cdot+y_{n,1})-w_1\|_{L^\infty(B_R(0))}+\frac{(k-1)\varepsilon}{k}<\varepsilon.
	\end{align*}
	In any case, we can choose $\varepsilon=\min\{\delta/2,1-\delta/2\}=\delta/2$ so that \eqref{eq:rnvortexless} holds.
	
	Once \eqref{eq:wjvortexless} and \eqref{eq:rnvortexless} are proved, \eqref{B:decomp} follows from Lemma~\ref{lema:decomG} by applying the same procedure by truncation described above.
	%	
	%	It only remains to prove \eqref{B:decomp}, we use again the 
	%	the translations  $\tau^{\pm}_{n,y}$. 
	%	By \eqref{B:split}, we  have
	%	$$
	%	\boB(v_n )= \boB(w_1)+\boB(v_n-w_1 )+o_n(1).
	%	$$
	%	Since   $\tau^+_{n,2}v_n-\tau^+_{n,2}w_1\to w_2$ a.e.\,  \eqref{B:split} implies that 
	%	\begin{align}
	%		\boB(v_n-w_1 )&=\boB(w_2)+\boB(\tau^+_{n,2}v_n-\tau^+_{n,2}w_1- w_2)+o_n(1)\\
	%		&=\boB(w_2)+\boB(v_n-w_1- \tau^-_{n,2}w_2)+o_n(1),
	%	\end{align}
	%	Repeating the same argument $k$-times, we get \eqref{B:decomp}.
	%	\begin{align}
	%	J_c(v_n-w_1- \tau^-_{n,2}w_2)&=J_c(w_3)+J_c(\tau^+_{n,3}v_n-\tau^+_{n,3} w_1- \tau^+_{n,3}\tau^-_{n,2}w_2- w_3)+o(1)\\
	%&	J_c(w_3)+J_c(v_n- w_1- \tau^-_{n,2}w_2-\tau^-_{n,3} w_3)+o(1)
	%		\label{key}
	%		\boB(v_n)=\sum_{j=1}^k J_c(w_j)+J_c(r_n),
	%	\end{align}
\end{proof}

In order to deal with the splitting of the nonlocal term, we introduce the notation 
$$\inner{u}{v}=\int_\R (\W*u)v, \quad  \normW{u}^2=\inner{u}{u}, \quad  \boQ(u)=\inner{u(2-u)}{u(2-u)}, \quad \text{ for all }u,v\in H^1(\R).$$
Notice that $\inner{\cdot }{\cdot }$ is symmetric and  bilinear, so that 
$ \normW{\cdot}$ defines a norm provided that $\wh\W\gneq 0$.

\begin{lemma}
	\label{lem:A}
	Let $\{v_n\}\subset H^1(\R)$ be a bounded sequence. Using the notation in  Theorem~\ref{thm:decomp}, we  have, up to a subsequence,
	%	\begin{equation}
	%		\label{asym:decompA}
	%		\liminf_{n\to\infty}  \boQ(v_n)=	\sum_{j=1}^\infty \boQ(w_{k}) +\lim_{k \to\infty}
	%		\liminf_{n\to\infty}	\boQ(r_{n,k}).
	%	\end{equation}
	\begin{equation}
		\label{asym:decompQ}
		\boQ(v_n)=	\sum_{j=1}^k \boQ(w_{j}) + \boQ(r_{n,k})+\varepsilon_{n,k},
	\end{equation}
	where $\{\varepsilon_{n,k}\}\subset\R$ satisfies
	\begin{equation}
		\label{epskn}
		\lim_{k\to\infty}\limsup_{n\to\infty}|\varepsilon_{n,k}|=0.
	\end{equation}
\end{lemma}

\begin{remark}
	\label{remark:decompA}
	It is clear from \eqref{asym:decomp} and \eqref{asym:decompQ} that 
	\begin{equation*}
		\boA(v_n)=	\sum_{j=1}^k \boA(w_{j}) + \boA(r_{n,k})+\varepsilon_{n,k},
	\end{equation*}
	for some $\{\varepsilon_{n,k}\}\subset\R$ satisfying \eqref{epskn}. Moreover, if $\wh\W\geq 0$ a.e.\ on $\R$, then $\boA(r_{n,k})\geq 0$, so that
	\begin{equation*}
		\sum_{j=1}^k\boA(w_{j})\leq \limsup_{n\to \infty}(\boA(v_n) + |\varepsilon_{n,k}|)\leq 
	\limsup_{n\to \infty}\boA(v_n) +
		\limsup_{n\to \infty} |\varepsilon_{n,k}|.
	\end{equation*}
	Then, \eqref{epskn} implies that
	\begin{equation}
		\label{asym:decompA}
		\sum_{k=1}^\infty \boA(w_{k})\leq 
		\limsup_{n\to\infty}  \boA(v_n).
	\end{equation}
	The inequality \eqref{asym:decompA} will be used below in order to show that, if $\{v_n\}$ is a Palais--Smale sequence of $J_c$ at level $\gamma_\gc(c)\not=0$, then $v_n$ is decomposed only in a finite number of profiles $w_1,\dots, w_k$. 
\end{remark}

\begin{proof}[Proof of Lemma~\ref{lem:A}]
	To prove \eqref{asym:decompQ}, we first remark that
	$$\boQ(u)=4\normW{u}^2+\normW{u^2}^2-4\inner{u}{u^2}.$$
	Observe also that, for  any  $\pmb f=(f_1,f_2,\dots, f_m) \in H^1(\R)^m$, we have
	\begin{align}
		\label{sums1}		
		&	\Big\vert\kern-0.25ex\Big\vert\kern-0.25ex\Big\vert
			\sum_{i=1}^m f_i
		\Big\vert\kern-0.25ex\Big\vert\kern-0.25ex\Big\vert ^2=
		\sum_{i=1}^m \normW{f_i}^2+T_1(\pmb f), \quad T_1(\pmb f)=\sum_{i\not=j}^m \inner{f_i}{f_j},\\
		\label{sums2}
	&	\Big\vert\kern-0.25ex\Big\vert\kern-0.25ex\Big\vert
		\Big(\sum_{i=1}^m f_i\Big)^2
		\Big\vert\kern-0.25ex\Big\vert\kern-0.25ex\Big\vert ^2=
		\sum_{i=1}^m \normW{f_i^2}^2
		+T_2(\pmb f), \ T_2(\pmb f)=
			\sum_{ i\not=j}^m \inner{f_i^2}{f_j^2}+
		2\sum_{k; i\not=j}^m \inner{f_k^2}{f_if_j} + \sum_{i\not=j; k\not=\ell}^m \inner{f_i f_j}{f_kf_\ell},\\
		\label{sums3}
		&\inner{ \sum_{i=1}^m f_i}{ \Big(\sum_{i=1}^m f_i\Big)^2}=
		\sum_{i=1}^m\inner{f_i}{f_i^2}+
		T_3(\pmb f), \quad  T_3(\pmb f)=	\sum_{i\not= j}^m\inner{f_i}{f_j^2}+
		\sum_{k;i\not= j}^m\inner{f_k}{f_if_j}.
	\end{align}
	In sum,
	\begin{align}
		\label{sumsQ}
		\boQ\Big(\sum_{i=1}^m f_i\Big)=
		\sum_{i=1}^m  \boQ(f_i)+ 
		T(\pmb f),\quad 
		\text{ with } T=4T_1+T_2-4T_3.
	\end{align}
	%	Notice that $T_1,T_2$ and $T_3$ are a continuous operators on $(H^1(\R))^m$, which follows 
	%	from the Sobolev embedding $H^1(\R)\subset L^q(\R)$, $q\in[2,\infty]$ and the estimates
	%	\begin{align}
	%		\label{T-cont}
	%		\norm{T_1(f)}&\lesssim \sum_{ i<j}^m \norm{f_i}_{L^2(\R)}\norm{f_j}_{L^2(\R)},\\
	%		\norm{T_2(f)}&\lesssim \sum_{k, i<j}^m \norm{f_k}_{L^4(\R)}^2 \norm{f_i}_{L^4(\R)}\norm{f_j}_{L^4(\R)},\\
	%		\norm{T_3(f)}&\lesssim 	\sum_{i\neq j}^m  \norm{f_i}_{L^2(\R)}\norm{f_j}_{L^4(\R)}^2
	%		+	 \sum_{k, i<j}^m \norm{f_k}_{L^2(\R)} \norm{f_i}_{L^4(\R)}\norm{f_j}_{L^4(\R)},\
	%	\end{align}
	
	%{\color{blue}	where $X\lesssim Y$ is a shorthand for the inequality $X\leq C Y$
	%	for some constant $C$ independent of $m$ and $\pmb f$. }
	
%	Recall that  Theorem~\ref{thm:decomp} implies that, for every $k$,
%	$$v_n=\sum_{i=1}^k w_{n,i} +r_{n,k},$$
%	where $ w_{n,i}=w_{i}(\cdot -y_{n,i})$, 
%	with $r_{n,k}$ satisfying \eqref{lim:resto}. 
	
	From now on, the 
	notation $X\lesssim Y$ means that there exists a constant $C$ independent of $n$ and $k$ such that $X\leq C Y$. Since we are assuming that $\{v_n\}$ is bounded in $H^1(\R)$, we can write $\norm{v_n}_{H^1(\R)}\lesssim 1$. Thus,
	it follows from \eqref{asym:decomp} and the Sobolev's embedding that 
	\begin{equation}
		\label{suc:acotadas}
		\sum_{i=1}^k\norm{w_{i}}_{L^p(\R)}\lesssim 1, \quad 	\norm{r_{n,k}}_{L^p(\R)}\lesssim 1,\quad \text{for all }p\in[2,\infty].
	\end{equation}

	Now we apply \eqref{sumsQ} with $m=k+1$,
	$f_i= w_{n,i}$, for $1\leq i\leq k$,
	and  $f_{m}=r_{n,k}$. Hence, one obtains \eqref{asym:decompQ} with
	\[\varepsilon_{n,k}=T(w_{n,1},\dots, w_{n,k},r_{n,k}).\]
	We aim to show that $\varepsilon_{n,k}$ satisfies \eqref{epskn}.
	
	Let us start with $T_1(w_{n,1},\dots, w_{n,k},r_{n,k})$, where there are two types of terms.
	The first type is of the form $\inner{ w_{n,i}}{ w_{n,j}}$ with $i\not=j$. This case is simple to handle by using that
	$\abs{y_{n,j}-y_{n,i}}\to \infty$, which leads to
	$$
	\inner{ w_{n,i}}{ w_{n,j}}=\inner{w_{i}}{w_j{(\cdot -y_{n,j}+y_{n,i})}}\to 0,
	\quad \text{ as } n\to\infty.
	$$ 
	The other  terms in the summation $T_1(w_{n,1},\dots, w_{n,k},r_{n,k})$ are of the form $\inner{ w_{n,i}}{r_{n,k}}$. In this case we apply \eqref{lim:resto} with $\varphi=\W\ast w_i$ and $\alpha_n=-y_{n,i}$, so we get
	$$
	\lim_{k\to\infty}\limsup_{n\to\infty}|\inner{ w_{n,i}}{r_{n,k}}|=\lim_{k\to\infty}\limsup_{n\to\infty}|\inner{w_{i}}{r_{n,k}{(\cdot +y_{n,i})}}|= 0.
	$$
	
	Let us study now  $T_2( w_{n,1},\dots, w_{n,k},r_{n,k})$ and 
	$T_3( w_{n,1},\dots, w_{n,k},r_{n,k})$. Here we find several types of terms. We first remark that the terms of the form $\inner{w_{n,i}}{w_{n,j}^2}$
	and $\inner{w_{n,i}^2}{w_{n,j}^2}$,  with $i\not=j$, and $\inner{r_{n,k}}{w_{n,i}^2}$ can be dealt with as we did above for the terms in $T_1( w_{n,1},\dots, w_{n,k},r_{n,k})$. Next we show how to treat the rest of the terms.
	
	First we consider the terms of the form, for $i\neq j$,
	\begin{equation}
		\label{termF}
		F_{n}=\inner{g_{n}}{ w_{n,i}w_{n,j}},
	\end{equation}
	for some $g_{n}$ with $\|g_{n}\|_{L^p(\R)}\lesssim 1$ for every $p\in [2,\infty]$. In view of \eqref{suc:acotadas}, this is the case when considering, for $1\leq\ell\leq m\leq k$, \begin{equation}
		\label{choices}
		g_{n}\in\{r_{n,k}, r_{n,k}^2, r_{n,k}w_{n,\ell}, w_{n,\ell}w_{n,m}, w_{n,\ell}\}.
	\end{equation}
%	These are in fact all the possible choices for $g_{n}$. {\color{blue} Check.} 
	
	In order to deal with \eqref{termF}, by the density of $\boC_0^\infty(\R)$ in $H^1(\R)$, we may consider two sequences $\{a_m\},\{b_m\}\subset\boC_0^\infty(\R)$ such that 
	$a_m\to w_i$ and  $b_m\to w_j,$  in $H^1(\R).$
	Of course, $\{a_m\},\{b_m\}$ depend on $i,j$ respectively, we do not denote explicitly this dependence for clarity. Notice that
	\begin{equation}
		\label{termFdecomposed}
		F_n=A_{n,k,m}+B_{n,k,m},
	\end{equation}
	where
	\[A_{n,k,m}=\inner{g_{n}}{w_{n,i}[w_{n,j}-b_m(\cdot-y_{n,j})]} 
	+ \inner{g_{n}}{b_m(\cdot-y_{n,j})[w_{n,i}-a_m(\cdot-y_{n,i})]}\]and
	\[B_{n,k,m}=\inner{g_{n}}{b_m(\cdot-y_{n,j})a_m(\cdot-y_{n,i})}.\]
	On the one hand, we have by \eqref{W-22} and H\"older's inequality, 
	\begin{equation*}
		|A_{n,k,m}|\lesssim \|w_j-b_m\|_{L^2(\R)}+\|b_m\|_{L^\infty(\R)}\|w_i-a_m\|_{L^2(\R)}.
	\end{equation*}
	Thus, given $\varepsilon>0$, we may fix $m$, independent of $n$, such that $|A_{n,k,m}|<\varepsilon$ for every $n$. On the other hand, since $a_m$ and $b_m$ have compact support and $|y_{n,i}-y_{n,j}|\to\infty$ as $n\to\infty$, it follows that $B_{n,k,m}=0$ for every $n$ large enough. In sum,
	\[\lim_{n\to\infty}F_n=0.\]
	
	We focus now on the terms of the form
	$G_n=\inner{g_{n}}{ w_{n,i}r_{n,k}},
	$
	with $g_{n}$ satisfying \eqref{choices}. Again by \eqref{W-22}, \eqref{suc:acotadas} and H\"older's inequality, we deduce the estimate 
	$$\abs{G_n}\lesssim \norm{w_{n,i} r_{n,k}}_{L^2(\R)} \lesssim  \norm{r_{n,k}}_{L^4(\R)}.$$
	Using \eqref{lim:resto} with $q=4$, we conclude that
	\[\lim_{k\to\infty}\limsup_{n\to\infty}\abs{G_n}=0.\]
	
	Finally, it remains to consider the terms of the form $\inner{w_{n,i}}{r_{n,k}^2}$ and $\inner{w_{n,i}^2}{r_{n,k}^2}$, which can be  handled as $G_{n}$. Consequently, the proof of is complete.
\end{proof}

Applying the splitting properties that we have proved to bounded Palais--Smale sequences, we obtain the following general theorem. 

\begin{theorem}\label{prop:splitting}
	Assume that $\W$ satisfies \ref{h:derivative}, \ref{W:infty}, \ref{W:apriori}, 
	 \eqref{assumption:Linfty} and $\wh \W\geq 0$ a.e.\ on $\R$. Let $c>0$ and let $\{v_n\}\subset \NV$ be a Palais--Smale sequence of $J_c$ at level $\gamma$, for some $\gamma\not=0$, i.e.\
	\begin{equation}
		\label{hyp:PS:thm}
		J_c(v_n)\to \gamma,\quad \|J_c'(v_n)\|_{H^{-1}(\R)}\to 0.
	\end{equation}
	Assume in addition that there exist $R>0$ and $\delta\in (0,1)$ such that, for all $n$, \begin{equation}
		\label{hyp:suc:vn}
		\|v_n\|_{H^1(\R)}\leq R\quad \text{and}\quad v_n\leq 1-\delta\,\,\text{ on }\R.
	\end{equation}
	Then there exist $k\in \N$ and $w_1,w_2,\dots,w_k \in \NV$ such that 
	\begin{equation}\label{eq:sumJc}
		\sum_{j=1}^{k} J_c(w_j)=\gamma.
	\end{equation}

	In addition, for any $1\leq j\leq k$, the function  $\rho_j=1-w_j$ 
	is a nontrivial finite energy solution to \eqref{eq:rho} and 
	$\rho_j\geq{\delta}$ on $\R.$
	%	\begin{equation}\label{eq:rhojvortexless}
	%		\rho_j\geq\frac{\delta}{2}\,\,\text{ on }\R.
	%	\end{equation}
\end{theorem}
\begin{proof}
	Since $\{v_n\}$ is bounded in $H^1(\R)$, 
 by Theorem~\ref{thm:decomp}, there are profiles $\{w_j\}_{j\geq 1}\subset H^1(\R)$
 and points $\{y_{n,j}\}\subset \R$ such that \eqref{Sn:conv}, \eqref{lim:resto}
 and \eqref{asym:decomp} hold.
  In addition, as in Lemma~\ref{lemma:splitting:2}, 
 we infer that $w_j\leq 1-\delta$ on $\R$, for all $j\geq 1$. 
 Moreover, as in the proof of Theorem~\ref{thm:existenceae}, we conclude that 
$J_c'(w_j)=0$, so that  $w_j$ is a solution to \eqref{eq:v} and $\rho_j=1-w_j$ is  solution to \eqref{eq:rho}.
Furthermore, we see that $\{\boA(v_n)\}$ is bounded. Since  $\wh \W\geq 0$ a.e.\ on $\R$, we deduce from \eqref{asym:decompA} in Remark~\ref{remark:decompA} that
	\begin{equation}
		\label{dem:cotaA}
		\sum_{j=1}^\infty \boA(w_{j}) \leq C,
	\end{equation}
	for some constant $C>0$ depending only  on  $\sup_{n}\norm{v_n}_{H^1(\R)}$ and $\|\wh\W\|_{L^\infty(\R)}$. 
	
	Let us show that there is $j_0\geq 1$ such that 
	$w_{j_0}\neq 0.$ Indeed, assuming otherwise, i.e.\ $w_{j}=0$, for all $j\geq 1$, we deduce from \eqref{Sn:conv} that $S_n=0$, so that 
	$v_n\to 0\text{ in } L^4(\R)$.
	On the other hand, \eqref{hyp:PS:thm} implies that 
	${2 J_c(v_n) - J_c'(v_n)(v_n)}\to 2 \gamma$. This leads to 
	a contradiction with the estimate in 
	\eqref{identidad:L4}, since $\gamma\neq 0$.
	
	In addition, there can only be a finite number of nonzero profiles.
	Indeed,  if $w_j$ is nonzero, then Lemma~\ref{lem:cota:A} provides a positive lower bound for $\boA(w_{j})$, which is independent of $j$. Therefore, \eqref{dem:cotaA} implies that the number of nonzero profiles if finite.
	Consequently, without loss of generality, we can assume that there is $k\geq 1$ such that 
	$w_{j}\not\equiv 0$, for all $j\leq k$, and $w_{j}\equiv 0$, for all $j>k$. In this manner, the profile decomposition is finite, and we can write 
\begin{equation}
	\label{ultima:dem}
	 v_n=\sum_{j=1}^k w_{n,j}+r_{n},
	 \end{equation}
	with $r_n\leq 1-\delta/2$, for $n$ large enough, by Lemma~\ref{lemma:splitting:2}. Also, by \eqref{asym:decomp},  
	\eqref{B:decomp} and \eqref{asym:decompQ}, 
	\begin{align}
		%	J_c(v_n-w_1- \tau^-_{n,2}w_2)&=J_c(w_3)+J_c(\tau^+_{n,3}v_n-\tau^+_{n,3} w_1- \tau^+_{n,3}\tau^-_{n,2}w_2- w_3)+o(1)\\
		%&	J_c(w_3)+J_c(v_n- w_1- \tau^-_{n,2}w_2-\tau^-_{n,3} w_3)+o(1)
		\label{key}
		J_c(v_n)=\sum_{j=1}^k J_c(w_j)+J_c(r_n)+o_n(1).
	\end{align}
	
Therefore, to prove \eqref{eq:sumJc}, it is  enough to show that
	\begin{equation}
		\label{dem:claim}
		J'_c(r_n)(r_n)\to 0.
	\end{equation}
Indeed,	assuming this claim and using that  $\norm{r_n}_{L^4(\R)}\to 0$, we can invoke the estimate in \eqref{identidad:L4} to conclude that 
	$J_c(r_n)$ converges to $0$. Thus, taking the limit in \eqref{key},  we obtain 
	\eqref{eq:sumJc}, which concludes the proof the  theorem.
	
	To establish \eqref{dem:claim},  recall that by \eqref{eq:Jderivative1},  
	\[J'_c(r_n)(r_n)=\int_\R (r_n')^2+\inner{f(r_n)}{(1-r_n)r_n}-c^2\int_\R h(r_n)r_n.\]
	 Remark that, 
	if $\{z_n\}$ is bounded in $H^1(\R)$ and $\{a_n\}$ is bounded in $L^4(\R)$, then 
	\begin{equation}\label{est:o1}
	\inner{f(z_n)}{a_n r_n}=o_n(1).
	\end{equation} Indeed, this follows from the
fact that $\norm{r_n}_{L^4(\R)}\to 0$ and the estimate
	$$\abs{\inner{f(z_n)}{a_n r_n}}\leq \|\wh\W\|_{L^\infty(\R)}\norm{f(z_n)}_{L^2(\R)}\|a_n\|_{L^4(\R)} \norm{r_n}_{L^4(\R)}.$$
	  Therefore,
	 using also \eqref{ultima:dem}, we obtain
	\begin{equation}
		\label{derJ:r_n}
		J'_c(r_n)(r_n)=\int_\R v_n'r_n' -\sum_{j=1}^k \int_\R  w_{n,j}'r_n'
		+\inner{f(r_n)}{r_n}-c^2\int_\R{h(r_n)r_n}+o_n(1).
	\end{equation}
On the other hand, since  $J_c'(w_j)=0$, we deduce that  $J_c'(w_{n,j})=0$, so that,
using also \eqref{est:o1}, we get 
	\begin{equation}
		\label{derJ:w_n}
		0=J'_c(w_{n,j})(r_n)=\int_\R w_{n,j}'r_n' +\inner{f(w_{n,j})}{r_n}-c^2\int_\R{h(w_{n,j})r_n}+o_n(1).
	\end{equation}
	Similarly, using that $\|J_c'(v_n)\|_{H^{-1}(\R)}\to 0$ and \eqref{est:o1}, we obtain
	\begin{equation}
		\label{derJ:v_n}
		o_n(1)=J_c'(v_n)(r_n)=\int_\R  v_n'r_n' +\inner{f(v_{n})}{r_n}-c^2\int_\R{h(v_{n})r_n}+o_n(1).
	\end{equation}
	By putting together \eqref{derJ:r_n}, \eqref{derJ:w_n} and  \eqref{derJ:v_n}, we conclude that
	\begin{equation}
		\label{final}
		J_c'(r_n)(r_n)=\inner{r_n}{f(r_n)-f(v_n)+\sum_{j=1}^k f(w_{n,j})}
		-c^2 \int_\R (h(r_{n})-h(v_{n})+\sum_{j=1}^k h(w_{n,j}) )r_n+o_n(1).
	\end{equation}
	Notice now that, using \eqref{ultima:dem} and that $f(s)=2s-s^2$, 
	\[f(r_n)-f(v_n)+\sum_{j=1}^k f(w_{n,j})=-r_n^2+v_n^2-\sum_{j=1}^k w_{n,j}^2.\]
	Similarly, 
	\[h(r_{n})-h(v_{n})+\sum_{j=1}^k  h(w_{n,j})=g(r_{n})-g(v_{n})+\sum_{j=1}^k g(w_{n,j}), \text{  with }g(s):= h(s)-s=\frac{3s^2-8s+6}{4(1-s)^3}s^2.\]
		In sum,  \eqref{final} can be simplified as
	\begin{equation}
		\label{final2}
		J_c'(r_n)(r_n)=-\inner{r_n}{r_n^2-v_n^2+\sum_{j=1}^k w_{n,j}^2}
		-c^2 \int_\R (g(r_{n})-g(v_{n})+\sum_{j=1}^k g(w_{n,j}) )r_n+o_n(1).
	\end{equation}
Moreover, since  $r_n\leq 1-\delta/2$, $v_n\leq 1-\delta$ and 
$w_{n,j}\leq 1-\delta$, we can replace $g$ with $\tilde{g} \chi_\delta$,
where $\tilde{g}(s)=(3s^2-8s+6)s^2$ and $\chi_\delta$ is defined in \eqref{xi:delta}. Applying Lemma~\ref{lema:decomG} to \eqref{final2} with $G(s)=s^2$ 
	and with $G=\tilde{g} \chi_\delta$, and using \ref{W:infty}, we finally conclude  
	that there is a function $o_n(1;L^1)$ such that 
	$$
	\abs{J_c'(r_n)(r_n)}\lesssim \norm{\W}_{\infty}\norm{r_n}_{L^\infty(\R)}\norm{ o_n(1;L^1) }_{L^1(\R)}.
	$$
	This completes the proof of \eqref{dem:claim}.
\end{proof}

%%%%%%%%%%%%%%%%%%%%%%%%%%%%%%%%%%%%%%%%%%%%%%%

%%%%%%%%%%%%%%%%%%%%%%%%%%%%%%%%%%%%%%%%%%%%%%%

The property \eqref{eq:sumJc} given by Theorem~\ref{prop:splitting} can be seen as an a priori estimate for solutions obtained via splitting of Palais--Smale sequences. However, it is not clear how to use this property if $J_c$ changes sign. The following lemma guarantees that, actually, $J_c(1-\rho)$ is nonnegative if $\rho$ is a finite energy solution with sufficiently small maximum.

\begin{lemma}\label{lemma:positiveJc}

	Let $c>0$ and let $u\in\boN\boE(\R)$ be a solution to \eqref{TWc}. Then we have the following estimates
in terms of $\rho=|u|$ and $\eta=1-|u|^2$.
	\begin{enumerate}
		\item If  \ref{h:derivative} is satisfied, then 
			\begin{equation}\label{JcAfterPohozaev}
			J_c(1-\rho)=\int_\R(\rho')^2+\frac{1}{8\pi}\int_\R\xi\big(\wh\W\big)'(\xi)|\wh\eta|^2.
		\end{equation}
		\item If \ref{h:restrictive} is satisfied, then 
			\begin{equation}\label{J-elliptica}
			J_c(1-\rho)\geq \frac12 \int_\R(1-m\rho^2)(\rho')^2+
			\frac{1}{4} \int_\R \Big(1-\frac{c^2}{2\rho^2}\Big)\eta^2.
		\end{equation}
		\item If \ref{h:derivative} and \ref{h:restrictive} are satisfied, then 
	\begin{equation}\label{new-J}
	J_c(1-\rho)\geq \int_\R(1-m\rho^2)(\rho')^2.
\end{equation}		
	\end{enumerate}

\end{lemma}

%{\color{blue}
%\begin{remark}
%	From Lemma~\ref{lemma:positiveJc} one directly deduces that, if $m\|\rho\|_{L^\infty(\R)}^2<1$, then $J_c(1-\rho)\geq(1-m\|\rho\|_{L^\infty(\R)}^2)\int_\R(\rho')^2\geq 0.$
%\end{remark}
%}

\begin{proof}
	Combining \eqref{eq:pohozaev} and \eqref{eq:quadratic} yields
	\[\frac{c^2}{4}\int_\R\frac{\eta^2}{1-\eta}+\frac{1}{4}\int_\R\frac{(\eta')^2}{1-\eta}=\frac{1}{4\pi}\int_\R \left(\wh\W(\xi)-\xi\big(\wh\W\big)'(\xi)\right)|\wh\eta(\xi)|^2 d\xi.\]
	Writing the left-hand side in terms of $\rho$ and multiplying by $1/2$, we arrive at
	\[\frac{c^2}{8}\int_\R\frac{(1-\rho^2)^2}{\rho^2}+\frac{1}{2}\int_\R(\rho')^2=\frac{1}{8\pi}\int_\R \left(\wh\W(\xi)-\xi\big(\wh\W\big)'(\xi)\right)|\wh\eta(\xi)|^2 d\xi.\]
	Observe now that, by Plancherel's identity,
\begin{equation}\label{proof-J-elip}
	J_c(1-\rho)=\frac{1}{2}\int_\R(\rho')^2 + \frac{1}{8\pi}\int_\R \wh\W(\xi)|\wh\eta(\xi)|^2 d\xi-\frac{c^2}{8}\int_\R\frac{\eta^2}{\rho^2}.
\end{equation}

From both previous identities, we conclude the proof of \eqref{JcAfterPohozaev}. Now, using \ref{h:restrictive} and that $\eta'=-2\rho\rho'$, we derive
	\begin{align*}
		J_c(1-\rho)&\geq\int_\R(\rho')^2-\frac{m}{8\pi}\int_\R\xi^2|\wh\eta|^2=\int_\R(\rho')^2-\frac{m}{8\pi}\int_\R|\wh{\eta'}|^2
		\\
		&=\int_\R(\rho')^2-\frac{m}{4}\int_\R(\eta')^2=\int_\R(1-m\rho^2)(\rho')^2,
	\end{align*}	
	which gives \eqref{new-J}.
	
	The proof of \eqref{J-elliptica} follows directly from \eqref{proof-J-elip}, using \eqref{desigualdad-W} and 
	that  $\eta'=-2\rho\rho'$.
\end{proof}

%We can now conclude uniform estimates for a sequence of solutions $(u_n)$, associated with speed $c_n$, with $c_n\to c.$

We are now in position to prove a  uniform estimate for solutions obtained from bounded Palais--Smale sequences.

\begin{corollary}\label{cor:solutionwithboundedgrad}
Assume that $\W$ satisfies  \ref{h:derivative}, \ref{h:restrictive}, \ref{W:infty} and \ref{W:apriori}. Assume in addition that $mV_0(c)^2<1$, where $m$ and $V_0(c)$ are given by \ref{h:restrictive} and \ref{W:apriori} respectively. Then, for any  $c\in (0,\sqrt{2})$, there exist  sequences $\{c_n\}\subset (c,\sqrt{2})$ and $\{u_n\}\subset\boN\boE(\R)$ such that $c_n\to c$ and $u_n$ is a nontrivial solution to $(S(\W,c_n))$ for all $n$. In addition, there exist $C_1, C_2>0$ and $\delta_c \in (0,1)$, independent of $n$, such that, denoting $\rho_n=|u_n|$,

	\begin{equation}\label{eq:gradbound2}
	\norm{\rho_n'}_{L^2(\R)}^2\leq C_1J_{c_n}(1-\rho_n)\leq C_2,\quad\text{ for all } n,
	\end{equation}

	and 
	\begin{equation}\label{eq:lowerbound2}
		\rho_n\geq\delta_c\ \text{ on }\R, \ \text{ for all }n.
	\end{equation} 
\end{corollary}	
\begin{proof}
Notice that $\W$ satisfies \ref{h:restrictive}, so that \eqref{desigualdad-W} and \eqref{assumption:Linfty} hold, and 
\ref{h:lowerboundWbis} is fulfilled with $\sigma=1$ and $\kappa=m/2$.
	
	Let $\gc\in(0,c)$. Consider the set
	$\boD_\gc=\{s\in (\gc,\sqrt{2}) :\  \gamma_\gc\text{ is differentiable at }s\}.$	
	Let $\{c_n\}\in\boD_\gc$ be a nondecreasing sequence such that $c_n\to c$. Recall that such a sequence exists thanks to \eqref{Dcmeasure}. Proposition~\ref{prop:boundedPSsequence} implies that, for every fixed $n$, there exists a sequence $\{v_{n,m}\}\subset\NV$ such that 
	\[\|v_{n,m}\|_{H^1(\R)}\leq R_n,\,\,v_{n,m}\leq 1-\delta_n,\,\, \lim_{m\to\infty} J_{c_n}(v_{n,m})=\gamma_\gc(c_n),\,\,\lim_{m\to\infty}\|J'_{c_n}(v_{n,m})\|_{H^{-1}(\R)}=0,\]
	where $R_n>0$ and $\delta_n\in (0,1)$ do not depend on $m$. Therefore, Theorem~\ref{prop:splitting} yields
	\[\sum_{j=1}^{k_n} J_{c_n}(w_{j,n})=\gamma_\gc(c_n),\]
	being $1-w_{j,n}$ nontrivial finite energy solutions to \eqref{eq:rho} for every $j=1,\dots,k_n$. Let us denote $w_{1,n}=w_n$ and $\rho_n=1-w_n$. The continuity of $V_0$ implies that $mV_0(c_n)^2<1$ for every $n$ large enough. Then, by Lemma~\ref{lemma:positiveJc}, we deduce that
	\[(1-mV_0(c_n)^2)\int_\R (\rho_n')^2\leq J_{c_n}(w_n)\leq \gamma_\gc(c_n).\] 
	Since $\gamma_\gc$ is a nonincreasing function and $\{c_n\}$ is nondecreasing, we have that $\gamma_\gc(c_n)\leq\gamma_\gc(c_1)$. Moreover, again by the continuity of $V_0$, it is clear that the sequence $\{1-mV_0(c_n)^2\}$ is bounded away from zero. In conclusion, we obtain \eqref{eq:gradbound2}.
%We can take $\gc>0$	such that $c \in \boD_{\mathfrak{c}}$. Since $\gamma_c$ is nonincreasing, it is differentiable a.e., so that there are  $\ve>0$ and  a sequence 
% $\{c_\ell\}\subset \boD_{\mathfrak{c}} $  such that $\sqrt 2-\ve \geq c_\ell\geq c$ and $c_\ell\to c$, as $\ell\to\infty.$
%		By virtue of Proposition~\ref{prop:boundedPSsequence}, there exist $R_\ell>0$, $\delta_\ell \in (0,1)$ and $\{v_{n,\ell}\}\subset \NV$ such that \eqref{eq:boundedPSsequence} holds.
%		Therefore, Theorem~\ref{prop:splitting} can be applied to conclude that  there is at least one finite  nontrivial  energy solution to \eqref{eq:rho}, denoted by $\rho_{\ell}$ in $1-\NV$.
%		In addition, by   \eqref{eq:sumJc} and 
%		Lemma~\ref{lemma:positiveJc}, we conclude that
%		\[	\norm{\rho_\ell'}_{L^2(\R)}^2\leq J_{c_\ell}(1-\rho)\leq \gamma_\gc(c_\ell)\leq \gamma_\gc(c),\]
%		where we used again that  $\gamma_c$ is nonincreasing for the last inequality.
%		

By Lemma~\ref{prop:hydrodynamic}, there is $\theta_n$ such that $u_n=\rho_n e^{i\theta_n}$ is a  finite energy solution to $(S(\W,c_n))$. Finally,
		we deduce  from \eqref{cota:minora} in Proposition~\ref{prop:lowerbound}, that for all $x\in\R$,
	$$	|u_n(x)|\geq \delta_c:=\inf_{s\in[\gc,c]} \frac{\sqrt{1+4s^2/V_1(s)}-1}{\sqrt{1+4c^2/V_1(s)}+1},$$
	which proves \eqref{eq:lowerbound2}.
\end{proof}

%%%%%%%%%%%%%%%%%%%%%%%%%%%%%%%%%%%%%%%%%

\subsection{Passing to the limit}

For any $c\in(0,\sqrt{2})$, Corollary~\ref{cor:solutionwithboundedgrad} provides a sequence of nontrivial finite energy solutions $u_n$ to $(S(\W,c_n))$ with $c_n\to c$. The last step for completing the proof of Theorem~\ref{thm:existenceallc} consists of passing to the limit in $(S(\W,c_n))$, controlling that the limit is nontrivial and has finite energy. To this aim, the estimates proved in the previous subsections will be essential. In this subsection we adapt a technique from  \cite{ruiz-bellazzini} in a similar context.

We start with a lemma that provides a sufficient condition for the boundedness of the sequence of energies. 

\begin{lemma}\label{lemma:Snr}
	Assume that the hypotheses in Corollary~\ref{cor:solutionwithboundedgrad} hold,
and let  $\{\rho_n \}$ be given by Corollary~\ref{cor:solutionwithboundedgrad}. 
Set	$$S_n^r=\{x\in\R: \rho_n(x)<r\}.$$
If for some $r\in ({c}/{\sqrt{2}},1 )$ the sequence $\{|S_n^r|\}$ is bounded, then $\{E(u_n)\}$ is also bounded.	
\end{lemma}

\begin{proof}
	Let $\{c_n\}$ and $\delta_c\in (0,1)$ be given by Corollary~\ref{cor:solutionwithboundedgrad}, and let us take some $r\in (0,1)$. Since  $mV_0(c_n)^2<1$ for all $n$ large, we have $1-m\rho_n^2$, so by invoking  \eqref{J-elliptica},  we obtain
	\begin{align*}
		C\geq J_{c_n}(1-\rho_n)&\geq \frac{1}{4}\int_{S_n^r}\left(1-\frac{c_n^2}{2\rho_n^2}\right)(1-\rho_n^2)^2+
		\frac{1}{4}\int_{(S_n^r)^c}\left(1-\frac{c_n^2}{2\rho_n^2}\right)(1-\rho_n^2)^2
		\\
		&\geq \frac{1}{4}\left(1-\frac{c_n^2}{2r^2}\right)\int_\R (1-\rho_n^2)^2-\frac{c_n^2(1-\delta_c^2)^2}{8\delta_c^2}|S^r_n|.
	\end{align*}   
Since  $c_n\to c$, choosing $r\in ({c}/{\sqrt{2}},1)$, we infer that 
	\[\int_\R (1-\rho_n^2)^2\leq C_1(1+|S_n^r|),\quad\text{ for all } n,\]
	for some constant $C_1>0$ independent of $n$. Thus, if $\{|S_n^r|\}$ is bounded, we get 
	\[\int_\R(1-\rho_n^2)^2\leq C_2,\]
	for some $C_2>0$.
	
	Recall that $J_{c_n}(1-\rho_n)=E(u_n)-c_n p(u_n)$ (see Remark~\ref{rem:energymomentum}). Then, using \eqref{eq:gradbound2} and \eqref{eq:lowerbound2} again, we conclude
	\[E(u_n)=J_{c_n}(1-\rho_n)+\frac{c_n^2}{4}\int_\R\frac{(1-\rho_n^2)^2}{\rho_n^2}\leq C+\frac{c_n^2}{4\delta_c^2}\int_\R (1-\rho_n^2)^2\leq C_3,\]
for some $C_3>0$.
\end{proof}

Next result is proved as Lemma~{6.6} in \cite{ruiz-bellazzini}.

\begin{lemma}\label{lemma:L1}
	Let $\{f_n\}\subset L^1(\R)$ be a bounded sequence, and consider a sequence $\{S_n\}$ of measurable subsets of $\R$ such that $|S_n|\to\infty$. Then, for every $n$, there exist $x_n\in S_n$ and $R_n>0$ with $R_n\to\infty$ such that
	\[\int_{B(x_n,R_n)}|f_n|\to 0.\]
\end{lemma}

So far, we have only used the condition \ref{W:infty} for $p=2$ or $p=\infty$.
In the  next lemma we will use it also for $p=1$ in order to handle the weak star convergence, denoted by $\wstar$,  in $L^\infty$.
%{\color{blue} El siguiente lema nos permitir\'a pasar al l\'imite ``localmente'' en el t\'ermino ``no local''.}

\begin{lemma}\label{lemma:localconvergence}
	Assume that $\W$ satisfies  \ref{W:infty} and let $\{\eta_n\}$ be a bounded sequence in  $W^{1,\infty}(\R)$. Then there exists $\eta\in L^\infty(\R)$ such that, up to a subsequence,
	  $\eta_n\to\eta$ in $L^\infty_{\textup{loc}}(\R)$ and  $\W\ast\eta_n\wstar \W\ast\eta$   in $L^\infty(\R).$
	  In addition, for any sequence $\{f_n\}\subset L^\infty(\R)$ such that 
	   $f_n\to f$ in $L^\infty_{\textup{loc}}(\R)$,
	   we have the following convergence in the sense of distributions,
	  \begin{equation}
	  	\label{W:etan}
	  	f_n(\W\ast\eta_n) \to f(\W\ast\eta) \quad  \text{ in } \boD'(\R).
	  	\end{equation}	 
\end{lemma}

\begin{proof}
	First, a standard diagonal argument together with Ascoli--Arzela's theorem imply that there exists $\eta\in L^\infty(\R)$ such that, up to a subsequence, 
\begin{equation}
	\label{dem:eta1}
	\eta_n\to\eta\quad\text{ in }L^\infty_{\text{loc}}(\R).
\end{equation}
	
On the other hand,  we also deduce that there is $g\in L^\infty(\R)$ such that, up to a subsequence,  
$\eta_n\wstar g$  in  $L^\infty(\R)$. By \eqref{dem:eta1}, we get $g=\eta$.
Let us now fix a function $\varphi \in L^1(\R)$. By	\ref{W:infty}, we have $\W*\varphi \in L^1(\R)$ and, using that $\W$ is even, we deduce that
$$
\int_\R (\W*\eta_n)\varphi=\int_\R \eta_n (\W*\varphi)\to  \int_\R\eta (\W*\varphi)=\int_\R(\W*\eta) \varphi.
$$
Therefore,  
\begin{equation}
	\label{dem:W:etan}
	\W\ast\eta_n\wstar \W\ast\eta \quad \text{ in }L^\infty(\R).
\end{equation}
To prove \eqref{W:etan}, we consider $\phi \in \boC_0^\infty(\R)$,
with $\supp \phi \subset K$, for some compact set $K$, 
and  notice that 
 \begin{align}
 	\label{dem:weaks}
\int_\R 	(\W\ast\eta_n)f_n\phi =
\int_\R 	(\W\ast\eta_n)f\phi 
+ \int_K	(\W\ast\eta_n)(f_n-f)\eta\phi. 
	\end{align}
The second term in the right-hand side can be bounded by $\norm{\W}_{\infty}\norm{\eta_n}_{L^\infty(\R)}
\norm{f_n-f}_{L^\infty(K)}\norm{\phi}_{L^1(K)},$
which goes to zero by hypothesis. Since   $f\phi \in L^1(\R)$, 
using \eqref{dem:W:etan}
we can thus pass to the limit in \eqref{dem:weaks} and obtain \eqref{W:etan}.
\end{proof}

%	We aim to prove that, for every compact set $\boK\subset\R$,  \[\W\ast\eta_n\to\W\ast\eta\quad\text{ strongly in }L^\infty(\boK).\] 
%	
%	Indeed, let $\boK\subset\R$ be compact. 
%	By \ref{W:infty} we deduce that $\{\W\ast \eta_n\}$ is bounded in $W^{1,\infty}(\R)$, so that by Ascoli--Arzela's theorem, 
%	 there exist a new subsequence $\{\eta_{n_k}\}$ and $w\in L^\infty(\boK)$ such that
%	\[\W\ast\eta_{n_k}\to w\quad\text{ strongly in }L^\infty(\boK).\]
%    In particular, $\W\ast\eta_{n_k}\to w$ pointwise on $\boK$. On the other hand, we have already shown that $\eta_{n_k}\to\eta$ pointwise on $\R$. Moreover, for any fixed $x\in\R$, observe that
%	\[\big|f(x-y)\eta_{n_k}(y)\big|\leq C|f(x-y)|\quad\text{a.e. }y\in\R,\]
%	for some constant $C>0$, where $f(x-\cdot)\in L^1(\R)$. In sum, by the Lebesgue's theorem, $f\ast\eta_{n_k}\to f\ast \eta$ pointwise on $\R$ and, in turn, $\W\ast\eta_{n_k}\to \W\ast\eta$ pointwise on $\R$. Thus, $w=\left.\eta\right|_\boK$. Since $\eta$ was already fixed, then the original sequence $\{\W\ast\eta_n\}$ converges itself to $\W\ast\eta$ strongly in $L^\infty(\boK)$, so the proof is finished.

We are now in a position to prove Theorem~\ref{thm:existenceallc}.
\begin{proof}[Proof of Theorem~\ref{thm:existenceallc}]
	Let $\{c_n\}$ and $\{u_n\}$ be the sequences given by Corollary~\ref{cor:solutionwithboundedgrad}. Thus, $\{c_n\}\subset (0,\sqrt{2})$ with $c_n\to c$ and  $u_n\in\boN\boE(\R)$ is a nontivial solution to $(S(\W,c_n))$ for all $n$. We start by proving that $\{E(u_n)\}$ remains bounded. Indeed, arguing by contradiction, assume that $\{|S_n^r|\}$ is unbounded for every $r\in(c/\sqrt{2},1)$, where $S_n^r$ is defined in Lemma~\ref{lemma:Snr}. Let us thus fix $r\in({c}/{\sqrt{2}},1)$ to be chosen later. Observe that, by \eqref{eq:gradbound2}, the sequence $\{(\rho_n')^2\}$ is bounded in $L^1(\R)$. Then, applying Lemma~\ref{lemma:L1}, for every $n$ there exist $x_n\in S_n^r$ and $R_n>0$ such that $R_n\to\infty$ and
	\begin{equation}\label{eq:zerograd}
		\int_{B(0,R_n)}\rho_n'(x+x_n)^2 dx\to 0.
	\end{equation}
	Now we define $\tilde{u}_n=u_n(\cdot+x_n)$ and $\tilde{\rho}_n=|\tilde{u}_n|$. We know from \ref{W:apriori} and Proposition~\ref{prop:universalestimate} that $\{\tilde{u}_n\}$ is bounded in $W^{k,\infty}(\R;\C)$ for every $k\in\N$. As a consequence, arguing as in the proof of Lemma~\ref{lemma:regularity} and taking \eqref{eq:lowerbound2} into account, we deduce that $\{\tilde{\rho}_n\}$ is bounded in $W^{k,\infty}(\R)$ for every $k\in\N$. In particular, there exists $\rho\in W^{2,\infty}(\R)$ such that, up to a subsequence,
	\[\tilde{\rho}_n\to\rho\quad \text{ in }W^{2,\infty}_{\text{loc}}(\R).\]Moreover, thanks to \eqref{eq:zerograd}, we deduce that $\rho'\equiv 0$, so $\rho$ is a constant. Furthermore, the pointwise convergence $\tilde{\rho}_n\to \rho$ leads to
	\[\rho_n(x_n)=\tilde{\rho}_n(0)\to \rho.\]
	Therefore, since $\rho_n(x_n)<r$ for all $n$, it follows that $\rho\leq r$.  
	
	Notice that  $\tilde{\rho}_n\in 1+H^1(\R)$ satisfies the equation
	\begin{equation}\label{rhoneq}
		-\tilde{\rho}_n''+\frac{c_n^2(1-\tilde{\rho}_n^4)}{4\tilde{\rho}_n^3}=\tilde{\rho}_n(\W\ast(1-\tilde{\rho}_n^2))\quad\text{ on }\R.
	\end{equation}
	We aim to pass to the limit in \eqref{rhoneq}. In order to do so we notice that, since $\rho$ is a constant,  
	\[\tilde{\rho}_n''\to 0\quad\text{  in } L^\infty_{\text{loc}}(\R).\]
	Besides, by virtue of Lemma~\ref{lemma:localconvergence} and \eqref{Wconv1}, we have  
$$	\tilde{\rho}_n (\W\ast(1-\tilde{\rho}_n^2))\to 	\rho (\W\ast(1-{\rho}^2))=\rho(1-\rho^2)\wh\W(0)\quad\text{  in }\boD'(\R).$$
	Thus, using that $\wh\W(0)=1$, we can  to pass to the limit in \eqref{rhoneq} in $\boD'(\R)$ to obtain
	\[c^2(1-\rho^4)=4\rho^4(1-\rho^2).\]	
	Using that $1-\rho^4=(1-\rho^2)(1+\rho^2)$, it follows that
	\[c^2(1+\rho^2)=4\rho^4.\]
	From this equation, it is immediate to check that
	\[\rho^2=\frac{c^2+c\sqrt{c^2+16}}{8}.\]
	Observe that $\frac{c^2+c\sqrt{c^2+16}}{8}>\frac{c^2}{2}$ since $c<\sqrt{2}$. Therefore, we can choose $\varepsilon>0$ small enough (independent of $r$) so that
	\[\varepsilon+\frac{c^2}{2}<\frac{c^2+c\sqrt{c^2+16}}{8}=\rho^2\leq r^2.\]
	Finally, if we choose $r=\sqrt{\varepsilon+\frac{c^2}{2}}$ and take $\varepsilon>0$ possibly smaller so that $r\in (c/\sqrt{2},1)$, then we arrive at a contradiction. Therefore, $\{|S_n^r|\}$ must be bounded for some $r\in (c/\sqrt{2},1)$. Consequently, Lemma~\ref{lemma:Snr} implies that $\{E(u_n)\}$ is bounded too.
	
	Arguing as before, there exists $u\in\boC^2(\R;\C)$ such that, up to a subsequence,  
$u_n\to u$ in $W^{2,\infty}_{\text{loc}}(\R).$
	Moreover, Lemma~\ref{lemma:localconvergence} implies that 
$\W\ast(1-|u_n|^2)\wstar\W\ast(1-|u|^2)$ in $L^\infty(\R).$
	Thus, we can pass to the limit in $(S(\W,c_n))$ so that $u$ is a solution to \eqref{TWc}. 
	
Let us now check that $u\in\boE(\R)$. Indeed, as in Lemma~\ref{lemma:positiveJc}, using \eqref{desigualdad-W}, 
 we have
	\begin{align*}
		E(u_n)
		&\geq \frac{1}{2}\int_\R|u_n'|^2+\frac{1}{4}\int_\R(1-|u_n|^2)^2-\frac{m}{16\pi}\int_\R|\xi|^2|\wh\eta_n|^2
		\\
		&=\frac{1}{2}\int_\R|u_n'|^2+\frac{1}{4}\int_\R(1-|u_n|^2)^2-\frac{m}{2}\int_\R \rho_n^2(\rho_n')^2.
	\end{align*}
	Hence, \ref{W:apriori}, \eqref{eq:gradbound2} and the fact that $\{E(u_n)\}$ is bounded, imply that
	\[\frac{1}{2}\int_\R |u'_n|^2+\frac{1}{4}\int_\R (1-|u_n|^2)^2\leq C,\]
	for all $n$ and for some $C>0$ independent of $n$. By virtue of Fatou's lemma,
	\[\frac{1}{2}\int_\R |u'|^2+\frac{1}{4}\int_\R (1-|u|^2)^2\leq C.\]
	That is, $u\in\boE(\R)$.

	Finally, the estimate \eqref{eq:lowerbound2} ensures that $u\in\boN\boE(\R)$, while Proposition~\ref{prop:nonvanishing} implies that $u$ is nontrivial. The proof is concluded.
\end{proof}

%%%%%%%%%%%%%%%%%%%%%%%%%%%%%%%%%%%%%%%%%%%%%%%%%

\section{Nonexistence and properties of solitons}
\label{sec:decay}
This section is devoted to the study of the Fourier transform of equation \eqref{eq:ellipticequation}, that is 
\begin{equation}
	\label{contradiction}
M_c(\xi)\wh \eta(\xi)= \wh F(\xi),\quad \text{with } \ 
M_c(\xi)=\xi^2+2\widehat{\mathcal{W}}(\xi)-c^2, 
\end{equation}
where 
\[F=2K+2\eta(\W\ast\eta),\quad \eta=1-|u|^2\quad \text{and}\quad K=|u'|^2.\]
 We keep this notation for $F$ and $\eta$  for the rest of the Section, 
 and we assume, as always, that $\W$ satisfies \ref{H0}.
If $M_c>0$ a.e., we can recast \eqref{contradiction} as 
\begin{equation}
	\label{eq:conv}
	\wh \eta(\xi)=L_c(\xi) \wh F(\xi),\quad \text{with } \ 
	L_c(\xi)=\wh	\boL_c(\xi)=\frac{1}{M_c(\xi)}.
\end{equation}
We will see that the operator $\boL_c$ plays an essential role in order to study the regularity and asymptotic behavior at infinity of the solitons given by  Theorems~\ref{thm:existenceae} and \ref{thm:existenceallc}.

%For $c\geq 0$, let us define the measurable function 
%On any measurable set with nonzero measure where $M_c$ does not vanish a.e., one may also define the function
%\[L_c=\frac{1}{M_c}.\]
%Moreover, if $L_c\in \mathcal{S}'(\mathbb{R})$, then we may define $\mathcal{L}_c\in \mathcal{S}'(\mathbb{R})$ by
%\[\widehat{\mathcal{L}}_c=L_c.\]

 We also stress that \ref{h:lowerboundWbis} is sufficient condition for $\boL_c$ to be well defined, for $c\in [0,\sqrt{2\sigma})$. Indeed,  we have $M_c(\xi)\geq (1-2\kappa)\xi^2+2\sigma-c^2>0$ for a.e.\ $\xi\in\R$, so that
\begin{equation}
	\label{Lc:int}
\int_\R|L_c(\xi)|d\xi\leq\int_\R\frac{d\xi}{(1-2\kappa)\xi^2+2\sigma-c^2}<\infty.
\end{equation}
Thus  $L_c\in L^1(\R)$ and $\boL_c$ is a bounded continuous function on $\R$.
%Recall that condition \ref{h:lowerboundWbis} guarantees the existence of nontrivial finite energy traveling waves for a.e.\ $c\in (0,\sqrt{2\sigma})$, see Theorem~\ref{thm:existenceae}.

We can now establish our  nonexistence result for solitons with critical speed. 

\begin{theorem}\label{thm:nonexistence}
	Assume that $\wh\W\geq 0$ a.e.\ on $\R$ and that there exists $\delta>0$ such that one of the following holds:
	\begin{enumerate}
		\item $\wh\W(\xi)=1-\xi^2/2$, for all $\xi\in(-\delta,\delta)$.
		\label{nonex1}
		\item $\wh\W$ is differentiable on $(-\delta,\delta)$, $\wh\W(0)=1$ and $\wh\W(\xi)\not=1-\xi^2/2$ for a.e.\ $\xi\in (-\delta,\delta)$.
		\label{nonex2}
	\end{enumerate} 
Then $(S(\W,\sqrt{2}))$ admits no nontrivial solution in $\mathcal{E}(\mathbb{R})$.
\end{theorem}

\begin{proof}
%	[Proof of Theorem~\ref{thm:nonexistence}]
	Arguing by contradiction, assume that there exists a nontrivial solution $u\in\mathcal{E}(\mathbb{R})$ to  $(S(\W,\sqrt{2}))$. Then,  Proposition~\ref{prop:etakidentities} implies that \eqref{contradiction} holds, i.e.  
\begin{equation}
	\label{dem:nonexis}
M(\xi)\wh \eta(\xi)= \wh F(\xi),\quad \text{with } \ 
M(\xi)=\xi^2+2\widehat{\mathcal{W}}(\xi)-2.
	\end{equation}
On the other hand, by virtue of Lemma~\ref{lemma:regularity}, $\eta$ and $K$ belong to $W^{k,p}(\mathbb{R})$ for all $k\in\mathbb{N}$ and all $p\in [2,\infty]$.
%	 Furthermore, by  Proposition~\ref{prop:etakidentities}, $\eta$ and $K$ satisfy \eqref{eq:ellipticequation:section}. Taking Fourier transform in \eqref{eq:ellipticequation:section}, we deduce that
%	\begin{equation}
%	\label{contradiction}
%		M_c(\xi)\widehat{\eta}(\xi)=\widehat{F}(\xi)\quad\text{ a.e. }\xi\in\mathbb{R}.
%	\end{equation}
%	

	Let us show that $\wh F$ is continuous and $\wh F(0)>0$. Indeed, from \eqref{eq:localequation} and from the fact that $\eta(\pm \infty)=0$, we deduce the existence of  constants $R,C>0$ such that
	\[ \abs{u'(x)}^2\leq C(\eta(x)^2+\eta'(x)^2)\quad\text{for all } x\in \R\setminus[-R,R].\]
	Hence  $K\in L^1(\R)$ and, in turn, $F\in L^1(\R)$ and $\widehat{F}$ is continuous.
	Also, it follows from \eqref{dem:nonexis} that we can assume that $M\wh\eta$ is also continuous.
	 Moreover, since $u$ is not trivial and $\wh\W\geq 0$ a.e., Plancherel's identity yields
	\begin{equation}
		\label{intFpositive}
		\widehat{F}(0)=\int_\mathbb{R} F(x)dx=2\int_\R \abs{u'}^2+\frac{1}{\pi}\int_\R\wh\W(\xi)|\wh\eta(\xi)|^2d\xi>0.
	\end{equation}

If assumption \ref{nonex1} holds, we deduce that  $M\wh\eta=0$ on $(-\delta,\delta)$, so that, by \eqref{dem:nonexis},  $\wh F(0)=0$, which contradicts \eqref{intFpositive}.
	
If assumption \ref{nonex2} holds, then $M$ is differentiable on $(-\delta,\delta)$, $M(0)=0$ and $M(\xi)\not=0$ for a.e.\ $\xi\in (-\delta,\delta)$. Therefore,
	\begin{equation}
	\label{contradiction2}
		\widehat{\eta}(\xi)=L(\xi)\widehat{F}(\xi)\quad\text{a.e. }\xi\in (-\delta,\delta),
	\end{equation}
where $L=1/M$. Let us show now that $L\not\in L^1((-\tilde\delta,\tilde\delta))$ for every $\tilde\delta\in(0,\delta)$ small enough. Expanding $M$ around zero leads to
	\[M(\xi)=M'(0)\xi+o(\xi^2),\quad\text{ for all } \xi\in (-\tilde{\delta},\tilde{\delta}),\]
	for some $\tilde{\delta}\in (0,\delta)$. Thus, taking $\tilde{\delta}$ even smaller if necessary,
	\[|M(\xi)|\leq (|M'(0)|+1)|\xi|,\quad\text{ for all } \xi\in (-\tilde{\delta},\tilde{\delta}).\]
	In consequence,
	\[\int_{-\tilde{\delta}}^{\tilde{\delta}}L(\xi)d\xi\geq \frac{1}{|M'(0)|+1}\int _{-\tilde{\delta}}^{\tilde{\delta}}\frac{d\xi}{|\xi|}=\infty.\]
	In particular, $L\not\in L^2((-\tilde\delta,\tilde\delta))$. 	
	Hence, taking \eqref{contradiction2} into account, in order $\widehat{\eta}$ to belong to $L^2(\mathbb{R})$, it is necessary that $\widehat{F}(0)=0$. This is again a contradiction with \eqref{intFpositive}.
\end{proof}

We can prove now the nonexistence result stated in the Introduction.
\begin{proof}[Proof of Theorem~\ref{thm:nonexistence:intro}]
	Let $\delta > 0$ be such that $\wh \W\in\boC^2((-\delta,\delta))$.
	Recall that $\wh \W(0)=1$ and $(\wh \W)'(0)=0$.
	If  $(\wh \W)''(\xi)= -1$ for all $(-\delta,\delta)$, then $\wh \W(\xi)=1-\xi^2/2$,
for $\xi \in  (-\delta,\delta)$, so that we are in the case (ii) of Theorem~\ref{thm:nonexistence}.
	
	Assume now that $(\wh \W)''(0)\neq -1$. Then, by decreasing $\delta$ if necessary, we deduce by continuity that 
	$(\wh \W)''>-1$ on $(-\delta,\delta)$, or $(\wh \W)''<-1$ on $(-\delta,\delta)$.
	On the other hand, 
by	Taylor's theorem, and using that $\wh \W$ is even, we deduce that
	for any $\xi\in (-\delta,\delta)$, there exists $\tilde \xi\in (-\delta,\delta)$ such that 
	\begin{equation*}
	\wh \W(\xi)=1+(\wh \W)''(\tilde \xi)\frac{\xi^2}{2}.
	\end{equation*}
	Thus we are in the case (ii) of Theorem~\ref{thm:nonexistence}, and the conclusion follows.
\end{proof}

\subsection{Decay at infinity}
\label{sub:sec:decay}
We assume now that $M_c>0$ a.e.\ on $\R$ so that \eqref{eq:conv} holds a.e.
Notice also that if $L_c\in\mathcal{S}'(\mathbb{R})$, then 
\begin{equation}\label{eq:convolution}
	D^k\eta=\mathcal{L}_c\ast D^k F,\quad  \text{ for all }   k\in\N.
\end{equation}
This equation will be the key for analyzing the decay of the  solutions $u\in\boE(\R)$ to \eqref{TWc} as $x\to\pm\infty$. More precisely, it  will be provided the decay of $\eta=1-|u|^2$ at infinity. We start by showing that we  can also  recover  limits of $u$ at $\pm\infty$.
First, we need to establish the integrability of $\eta$, which means that $u$
has finite mass.

%We start with a lemma that tells that, under suitable conditions on $\W$, the finite energy solutions to \eqref{TWc} have finite mass.

\begin{lemma}\label{lemma:finitemass}
	Let $c\geq 0$ and let $u\in\boE(\R)$ be a solution to \eqref{TWc}. If $\boL_c\in L^1(\R)$, then $\eta\in W^{k,1}(\R)$ for every $k\in\N$. 
\end{lemma}

%\begin{remark}
%	For $k=0$ we obtain in particular that $\eta\in L^1(\R)$. That is, $u$ has finite mass. {\color{blue} The mass is formally conserved by the flow of the Gross-Pitaevskii equation. The physically natural fact that the solutions to the Gross-Pitaevskii equation with nonvanishing conditions at infinity have finite mass is in general not guaranteed (see references?).} 
%\end{remark}

\begin{proof}
%	[Proof of Lemma~\ref{lemma:finitemass}]
	In the case $k=0$, we argue as in the proof of Theorem~\ref{thm:nonexistence} to prove that $F\in L^1(\R)$. Then, if $\boL_c\in L^1(\R)$, Young's inequality applied to \eqref{eq:convolution} with $k=0$ implies that $\eta\in L^1(\R)$. The case $k\geq 1$ may be tackled similarly  by taking into account that $F'=4\eta'\W\ast\eta+2\eta\W\ast\eta'$, so the successive derivatives of $F$ have the form
	$D^k F=\sum_{j=1}^{2k}a_j b_j,$
	where $a_j,b_j\in L^2(\R)$. 
\end{proof}

For any $u\in\boE(\R)$, the limits $\lim_{x\to\pm}u(x)$ do not exist in general
(see \cite{gerard3}). The following result shows that, if $u$ solves \eqref{TWc}, then they do exist whenever $u$ presents no vortices and $\eta=1-|u|^2$ has finite mass.

\begin{proposition}\label{prop:limitsatinfinity}
	Let $c> 0$ and let $u=\rho e^{i\theta}\in \boN\boE(\R)$ be a solution to \eqref{TWc}. Assume that $\eta\in L^1(\R)$. Then the following limits exist and are finite:
	\begin{equation}\label{eq:limitstheta}
		\theta(\pm\infty)=\theta(0)+\frac{c}{2}\int_0^{\pm\infty}\frac{\eta}{1-\eta}.
	\end{equation}
	In particular, 
	\begin{equation}\label{eq:limitsu}
		u(+\infty)=e^{i\theta(+\infty)},\quad u(-\infty)=e^{i\theta(-\infty)}.
	\end{equation}
\end{proposition}

\begin{remark}
	On the one hand, we recall that Lemma~\ref{lemma:finitemass} provides sufficient conditions that assure that $\eta\in L^1(\R)$. On the other hand, we stress the fact that the limits $u(\pm\infty)$, if they exist, may be different from each other. In fact, from Proposition~\ref{prop:limitsatinfinity} it is easy to see that $u(+\infty)=u(-\infty)$ if and only if $\int_\R \frac{\eta}{1-\eta}=0$.
%	 This phenomenon of presenting different limits when approaching infinity from different directions is typical of the dimension one {\color{blue} (recall the explicit formula in the local case)}. The possible nonexistence of $\lim_{|x|\to\infty} u(x)$ represents an important difficulty in the proof of the existence of solution to \eqref{TWc}. In fact, one is not allowed to look for solutions in the \emph{affine space} $1+H^1(\R;\C)$. This is a remarkable difference with respect to the approach of \cite{ruiz-bellazzini}, for instance.
%	%	On the contrary, in the local case $\W=\delta_0$ and in dimension $N\geq 3$, {\color{blue}Gerad} assures that $\lim_{|x|\to\infty} u(x)$ exists for any $u\in\boE(\R)$. {\color{blue} Even for $N=2$, every \emph{solution} $u\in\boE(\R)$ to \eqref{TWc} has a limit as $|x|\to\infty$ (see reference).} 
\end{remark}

\begin{proof}[Proof of Proposition~\ref{prop:limitsatinfinity}]
By  Proposition~\ref{prop:hydrodynamic}, $\theta$ satisfies \eqref{eq:theta}. By integrating,  we have
	\[\theta(x)-\theta(0)=\frac{c}{2}\int_0^x\left(\frac{1}{\rho(y)^2}-1\right)dy=\frac{c}{2}\int_0^x\frac{\eta(y)}{1-\eta(y)}dy,\quad\text{ for all } x\in\R.\]
	Therefore, 
	\[\theta(+\infty)\coloneqq\lim_{x\to\infty}\theta(x)=\theta(0)+\int_0^\infty\frac{\eta}{1-\eta},\quad \theta(-\infty)\coloneqq\lim_{x\to -\infty}\theta(x)=\theta(0)-\int_{-\infty}^0\frac{\eta}{1-\eta}.\]
	Since $\inf_\R (1-\eta)=\inf_\R \rho^2>0$ and $\eta\in L^1(\R)$, it follows that both limits $\theta(+\infty)$ and $\theta(-\infty)$ are finite. Hence, we deduce directly \eqref{eq:limitsu} from the fact that $\rho(\pm \infty)=1$.
\end{proof}

In the rest of the subsection we will adapt the Bona--Li theory in \cite{BonaLi1,BonaLi2}
to our equation. First, we recall the following technical result proved in \cite{Pei}.
\begin{lemma}[\hspace{1sp}\cite{Pei}]
\label{lemma:Pei}
	For any $0<\ell<m$ and $\varepsilon>0$, the following inequality holds,
	\begin{equation}\label{convolutionineq}
		\int_{\mathbb{R}}\frac{e^{\ell|x|}}{(1+\varepsilon e^{|x|})^m e^{m|x-y|}}dx\leq B\frac{e^{\ell|y|}}{(1+\varepsilon e^{|y|})^m},\quad\text{ for all } y\in\mathbb{R},
	\end{equation}
	where $B=(\min\{\ell,m-\ell\})^{-1}$.
\end{lemma}

%\begin{theorem}\label{thm:exponentialdecay}
%	Let $c\geq 0$, and let $u\in\mathcal{E}(\mathbb{R})$ be a solution to \eqref{TWc}. Assume that $L_c\in\mathcal{S}'(\mathbb{R})$ and 
%	\begin{equation}\label{expLccondition}
%		e^{m|\cdot|}\mathcal{L}_c\in L^p(\mathbb{R})\,\,\text{ for some } m>0,\,\, p\in (1,\infty].
%	\end{equation}
%	Then, 
%	\[e^{\ell|\cdot|}D^k\eta\in (L^q\cap L^\infty)(\mathbb{R}),\quad\lim_{x\to\pm\infty}e^{\ell|x|}D^k \eta(x)=0,\quad\text{ for all } k\in\N,\,\, \text{ for all } \ell\in [0,m),\]
%	where $q=p'$ and $\eta=1-|u|^2$.
%\end{theorem}

\begin{proof}[Proof of Theorem~\ref{thm:exponentialdecay}]
	First, we point out that the result holds true for $\ell=0$. Indeed, \eqref{expLccondition} and H\"older's inequality yield  $\boL_c\in L^1(\R)$. Thus, by virtue of Lemmas~\ref{lemma:regularity} and \ref{lemma:finitemass}, $D^k\eta\in L^1(\R)\cap L^\infty(\R)$ for every $k\in\N$. We will then focus only on the case $\ell\in (0,m)$.
	
	From \eqref{eq:convolution}, H\"older's inequality and \eqref{expLccondition}, we deduce the following estimate, 
	\begin{equation}\label{convolutionestimate}
		|\eta(x)|\leq\int_\mathbb{R}|\mathcal{L}_c(x-y)|e^{m|x-y|}\frac{|F(y)|}{e^{m|x-y|}}dy\leq C_1^\frac{1}{q}\left(\int_\mathbb{R}\frac{|F(y)|^q}{e^{qm|x-y|}}dy\right)^\frac{1}{q},
	\end{equation}
	where $C_1=\|e^{m|\cdot|}\boL_c\|_{L^p(\R)}^q$.
			We will prove next that $e^{\ell|\cdot|}\eta\in L^q(\mathbb{R})$ and $e^{\ell|\cdot|}\eta'\in L^q(\mathbb{R})$ for all $\ell\in (0,m)$. In order to do so, let $\ell\in (0,m)$ and, for all $\varepsilon\in (0,1]$, let us consider the functions
	\[h_\varepsilon(x)=\frac{e^{\ell|x|}}{(1+\varepsilon e^{|x|})^m}|\eta(x)|,\quad \tilde{h}_\varepsilon(x)=\frac{e^{\ell|x|}}{(1+\varepsilon e^{|x|})^m}|\eta'(x)|.\]
	Since $\eta, \eta'\in L^\infty(\mathbb{R})$ and $\ell<m$, it is clear that $h_\varepsilon, \tilde{h}_\varepsilon\in L^q(\mathbb{R})$. Let us take now $r\in (0,q)$ and $R>1$. Using \eqref{convolutionestimate} and H\"{o}lder's inequality with exponents $\frac{q}{q-r}$ and $\frac{q}{r}$, we deduce that
	\begin{align*}
		\int_{\{|x|>R\}} &|h_\varepsilon(x)|^q dx =\int_{\{|x|>R\}}|h_\varepsilon(x)|^{q-r}\frac{e^{\ell r|x|}}{(1+\varepsilon e^{|x|})^{rm}}|\eta(x)|^r dx
		\\
		&\leq C_1^\frac{r}{q}\int_{\{|x|>R\}}|h_\varepsilon(x)|^{q-r}\frac{e^{\ell r|x|}}{(1+\varepsilon e^{|x|})^{rm}}\left(\int_\mathbb{R}\frac{|F(y)|^q}{e^{qm|x-y|}}dy\right)^{\frac{r}{q}}dx
		\\
		&\leq C_1^\frac{r}{q}\left(\int_{\{|x|>R\}} |h_\varepsilon(x)|^q dx\right)^{\frac{q-r}{q}}\left(\int_{\{|x|>R\}} \frac{e^{\ell q|x|}}{(1+\varepsilon e^{|x|})^{qm}}\left(\int_\mathbb{R}\frac{|F(y)|^q}{e^{qm|x-y|}}dy\right)dx\right)^\frac{r}{q}.
	\end{align*}
	From the previous inequality, one gets directly
	\[\int_{\{|x|>R\}}|h_\varepsilon(x)|^q dx\leq C_1\int_{\{|x|>R\}}\frac{e^{\ell q|x|}}{(1+\varepsilon e^{|x|})^{qm}}\left(\int_\mathbb{R}\frac{|F(y)|^q}{e^{qm|x-y|}}dy\right)dx.\]
	Now, by Fubini's theorem and Lemma~\ref{lemma:Pei}, we derive
	\begin{align*}
		\int_{\{|x|>R\}}&|h_\varepsilon(x)|^q dx\leq C_1 \int_\mathbb{R} |F(y)|^q \left(\int_{\{|x|>R\}}\frac{e^{\ell q|x|}}{(1+\varepsilon e^{|x|})^{qm} e^{qm|x-y|}}dx\right)dy
		\\
		\leq C_1 &\int_{\{|y|>R\}}|F(y)|^q\frac{B e^{\ell q|y|}}{(1+\varepsilon e^{|y|})^{qm}} dy 
		\\
		+ C_1 &\int_{\{|y|\leq R\}}|F(y)|^q\int_{\{|x|>R\}}\frac{e^{\ell q|x|}}{(1+\varepsilon e^{|x|})^{qm} e^{qm|x-y|}} dxdy.
	\end{align*}
	We will now estimate the last two integrals. On the one hand, using the inequality $||x|-|y||\leq |x-y|$, we obtain
	\begin{align*}
		\int_{\{|y|\leq R\}}&|F(y)|^q\int_{\{|x|>R\}}\frac{e^{\ell q|x|}}{(1+\varepsilon e^{|x|})^{qm} e^{qm|x-y|}} dxdy\leq \|F\|_{L^\infty(\mathbb{R})}^q\int_{\{|y|\leq R\}}\int_{\{|x|>R\}}\frac{e^{\ell q|x|}}{e^{qm|x-y|}}dxdy
		\\
		&\leq \|F\|_{L^\infty(\mathbb{R})}^q\left(\int_{\{|y|\leq R\}}e^{qm|y|}dy\right)\left(\int_{\{|x|>R\}}e^{-(m-\ell)q|x|}dx\right)\coloneqq C_2/C_1.
	\end{align*}
	On the other hand, recall that, by Lemma~\ref{lemma:regularity}, 
	$\eta(\pm\infty)=\eta'(\pm\infty)=(\W\ast\eta)(\pm\infty)=0.$
	Hence, equation \eqref{eq:localequation} in Proposition~\ref{prop:etakidentities} implies that, for any fixed $\delta>0$ we may choose $R>1$ large enough so that
	\begin{equation}
		\label{est:F}
		|F(y)|^q\leq \delta|\eta(y)|^q+\delta|\eta'(y)|^q,\quad\text{ for all } |y|>R.
	\end{equation}
	Therefore,
	\[\int_{\{|y|>R\}}|F(y)|^q\frac{B e^{\ell q|y|}}{(1+\varepsilon e^{|y|})^{qm}} dy\leq \delta B\int_{\{|x|>R\}}|h_\varepsilon(x)|^q dx+\delta B\int_{\{|x|>R\}}|\tilde{h}_\varepsilon(x)|^q dx.\]
	In sum,
	\begin{equation}\label{hvarepsiloneq}
		\int_{\{|x|>R\}}|h_\varepsilon(x)|^q dx\leq \delta BC_1\int_{\{|x|>R\}}|h_\varepsilon(x)|^q dx+\delta BC_1\int_{\{|x|>R\}}|\tilde{h}_\varepsilon(x)|^q dx+ C_2.
	\end{equation}
	
	We will now derive a similar estimate for $\tilde{h}_\varepsilon$. Indeed, from \eqref{eq:convolution} with $k=1$ and using \eqref{eq:kprime}, it follows that
	\begin{equation*}\label{etaprimeeq}
		\eta'=\mathcal{L}_c\ast\left(4\eta'\mathcal{W}\ast\eta+2\eta\mathcal{W}\ast\eta'\right).
	\end{equation*}
	Notice that $(\mathcal{W}\ast\eta')(\pm\infty)=0$ too. Hence, taking $R>1$ larger if necessary, we deduce that 
	\[\left|4\eta'(y)\mathcal{W}\ast\eta(y)+2\eta(y)\mathcal{W}\ast\eta'(y)\right|^q\leq\delta|\eta(y)|^q+\delta|\eta'(y)|^q,\quad\text{ for all } |y|>R.\]
	This estimate allows us to follow the same arguments as we did for $h_\varepsilon$ in order to deduce
	\begin{equation}\label{htildeeq}
		\int_{\{|x|>R\}}|\tilde{h}_\varepsilon(x)|^q dx\leq \delta B C_1\int_{\{|x|>R\}}|\tilde{h}_\varepsilon(x)|^q dx + \delta B C_1\int_{\{|x|>R\}}|h_\varepsilon(x)|^q dx + C_3,
	\end{equation}
	where $C_3=C_2\|F'\|^q_{L^\infty(\R)}/\|F\|^q_{L^\infty(\R)}$. Taking now $\delta\in(0,1/2BC_1)$, it follows directly from \eqref{hvarepsiloneq} and \eqref{htildeeq} that
	\[\int_{\{|x|>R\}}|h_\varepsilon(x)|^q dx+\int_{\{|x|>R\}}|\tilde{h}_\varepsilon(x)|^q dx\leq C_4,\]
	where $C_4=(C_2+C_3)/(1-2\delta BC_1)$.
	By virtue of Fatou's lemma, we take limits as $\varepsilon$ tends to zero and obtain
	\[\int_{\{|x|>R\}}e^{\ell q|x|}|\eta(x)|^q dx + \int_{\{|x|>R\}}e^{\ell q|x|}|\eta'(x)|^q dx\leq C_4.\]
	In conclusion, $e^{\ell|\cdot|}\eta\in L^q(\R)$, $e^{\ell|\cdot|}\eta'\in L^q(\R)$.
	
	Notice that both $e^{\ell x}\eta(x)$ and $e^{-\ell x}\eta(x)$ belong to $W^{1,q}(\mathbb{R})$. Indeed, from what we have already proved, it is clear that
	\[e^{\pm \ell x}\eta(x)\in L^q(\mathbb{R}),\quad\left(e^{\pm \ell x}\eta(x)\right)'=\left(\pm \ell\eta(x)+\eta'(x)\right)e^{\pm \ell x}\in L^q(\mathbb{R}).\]
	Hence, the Sobolev's embedding theorem implies that $e^{\ell|\cdot|}\eta\in L^\infty(\mathbb{R})$ and $\lim_{x\to\pm\infty}e^{\ell|x|}\eta(x)=0$. 
	
	We have just proved the result for $k=0$. Taking $k=2$ in \eqref{eq:convolution}, we deduce analogously as before that $\eta''=\mathcal{L}_c\ast F''$, where $F''$ satisfies
	\[|F''(y)|^q\leq\delta|\eta(y)|^q +\delta|\eta'(y)|^q+\delta|\eta''(y)|^q,\quad\text{ for all } |y|>R.\]
	Following the same process as above, and using the estimates we already have for $e^{\ell|\cdot|}\eta$ and $e^{\ell|\cdot|}\eta'$, we prove the result for $k=1$. The complete proof follows easily by induction.
\end{proof}

Conditions \eqref{expLccondition} in Theorem~\ref{thm:exponentialdecay} is not easy to check, since the operator $\boL_c$ is not simple to compute in general. For this reason, we recall the following Paley--Wiener theorem that provides sufficient conditions on $L_c$ that we will use when applying  Theorem~\ref{thm:exponentialdecay} to our examples in Section~\ref{sec:examples}.
We refer to Theorem~5.4.2 in \cite{krantz-parks} or Theorem IX.13 in~\cite{reed-simonII} for details.
 \begin{theorem}
 	\label{thm:reed}
 Let $T\in L^2(\R)$. Then 	$e^{b \abs{x}}T \in L^2(\R)$ for all $b<a$, 
 if and only if $\wh T$ has an analytical continuation to the strip $\{z\in\C : \abs{z}<a\}$
 with the property that for each $\zeta \in \R$
 with $\abs{\zeta}<a$, $\wh T(\cdot +i\zeta)\in L^2(\R)$
 and for any $b< a$,
 \begin{equation}
 	\label{cond:reed}
 	\sup_{\abs{\zeta}\leq b}\norm{\wh  T(\cdot +i\zeta)}_{L^2(\R)}<\infty.
 \end{equation}
 \end{theorem}

We now tackle the  algebraic decay,  whose proof will follow similar lines to that of Theorem~\ref{thm:exponentialdecay}. We will employ the following lemma proved in \cite{BonaLi2}.
\begin{lemma}
	\label{lemma:BonaLi}
	For every $m,\ell\in\R$ such that $m>1$ and $0<\ell<m-1$, there exists $B>0$ such that the following inequality holds:
	\begin{equation}\label{eq:BonaLi}
		\int_\R \frac{|x|^\ell}{(1+\varepsilon |x|)^m(1+|x-y|)^m} dx\leq \frac{B|y|^\ell}{(1+\varepsilon|y|)^m},\quad\text{ for all } y\in\R \text{ and  for all } \varepsilon\in (0,1].
	\end{equation}
\end{lemma}

\begin{theorem}\label{thm:algebraicdecay}
	Let $c\geq 0$ and let $u\in\mathcal{E}(\mathbb{R})$ be a solution to \eqref{TWc}. Assume that $L_c\in\mathcal{S}'(\mathbb{R})$ and
	\begin{equation}
		\label{eq:algdecaycond}
		(1+|\cdot|)^s\mathcal{L}_c\in L^p(\mathbb{R}) \text{ for some }p\in (1,\infty],\,\, s>1-\frac{1}{p}.
	\end{equation}
Setting 	$q=p'$ and $\eta=1-|u|^2$, we have
	\[|\cdot|^\ell D^k\eta\in L^q(\R)\cap L^\infty(\mathbb{R}),\quad \lim_{x\to\pm\infty}|x|^\ell D^k \eta(x)=0,\quad\text{ for all } \ell\in\Big(0,s-1+\frac{1}{p}\Big),\,\,k\in\N.\]
\end{theorem}
\begin{proof}
%	[Proof of Theorem~\ref{thm:algebraicdecay}]
	First, from \eqref{eq:convolution} with $k=0$, H\"older's inequality and \eqref{eq:algdecaycond}, we deduce that
	\[|\eta(x)|\leq C_1^\frac{1}{q}\left(\int_\R\frac{|F(y)|^q}{(1+|x-y|)^{sq}}dy\right)^\frac{1}{q},\]
	where $C_1=\|(1+|\cdot|)^s\boL_c\|_{L^p(\R)}^q$. Now, for $\ell\in\left(0,s-1+1/p\right)$ and $\varepsilon\in (0,1]$, we consider the functions
	\[h_\varepsilon(x)=\frac{|x|^\ell}{(1+\varepsilon|x|)^s}|\eta(x)|,\quad \tilde{h}_\varepsilon(x)=\frac{|x|^\ell}{(1+\varepsilon|x|)^s}|\eta'(x)|.\]
	Let us take $R>1$. Arguing as in the proof of Theorem~\ref{thm:exponentialdecay} and using Lemma~\ref{lemma:BonaLi}, we obtain the estimate
	\begin{align}
		\label{alg1}
		\int_{\{|x|>R\}}|h_\varepsilon(x)|^q dx &\leq C_1\int_{\{|y|>R\}}|F(y)|^q\frac{B|y|^{\ell q}}{(1+\varepsilon|y|)^{sq}}dy
		\\
		\nonumber &+C_1\int_{\{|y|\leq R\}}|F(y)|^q\int_{\{|x|>R\}}\frac{|x|^{\ell q}}{(1+\varepsilon|x|)^{sq}(1+|x-y|)^{sq}}dxdy.
	\end{align}
	On the one hand, since $\ell q<sq-1$, then the function $x\mapsto\frac{|x|^{\ell q}}{(|x|+1-R)^{sq}}$ belongs to $L^1((R,\infty))$. Hence, the inequality $||x|-|y||\leq|x-y|$ leads to
	\begin{align*}
		\int_{\{|y|\leq R\}}&|F(y)|^q\int_{\{|x|>R\}}\frac{|x|^{\ell q}}{(1+\varepsilon|x|)^{sq}(1+|x-y|)^{sq}}dxdy
		\\
		&\leq 2R\|F\|_{L^\infty(\R)}^q\int_{\{|x|>R\}}\frac{|x|^{\ell q}}{(|x|+1-R)^{sq}}dx\coloneqq C_2/C_1.
	\end{align*}
	On the other hand, a shown in the proof of Theorem~\ref{thm:exponentialdecay},  for any fixed $\delta>0$, we may choose $R>1$ large enough so that \eqref{est:F} holds.
	Therefore,
	\[\int_{\{|y|>R\}}|F(y)|^q\frac{B |y|^{\ell q}}{(1+\varepsilon |y|)^{sq}} dy\leq \delta B\int_{\{|x|>R\}}|h_\varepsilon(x)|^q dx+\delta B\int_{\{|x|>R\}}|\tilde{h}_\varepsilon(x)|^q dx.\]
	In sum, by \eqref{alg1},
	\begin{equation*}
		\int_{\{|x|>R\}}|h_\varepsilon(x)|^q dx\leq \delta BC_1\int_{\{|x|>R\}}|h_\varepsilon(x)|^q dx+\delta BC_1\int_{\{|x|>R\}}|\tilde{h}_\varepsilon(x)|^q dx+ C_2.
	\end{equation*}
	Reasoning as in the proof of Theorem~\ref{thm:exponentialdecay}, we derive the analogous estimate for $\tilde{h}_\varepsilon$:
	\begin{equation*}
		\int_{\{|x|>R\}}|\tilde{h}_\varepsilon(x)|^q dx\leq \delta BC_1\int_{\{|x|>R\}}|h_\varepsilon(x)|^q dx+\delta BC_1\int_{\{|x|>R\}}|\tilde{h}_\varepsilon(x)|^q dx+ C_3,
	\end{equation*}
	where $C_3=C_2\|F'\|_{L^\infty(\R)}^q/\|F\|_{L^\infty(\R)}^q$. Combining the last two inequalities and taking $\delta\in(0,1/2BC_1)$ yields
	\[\int_{\{|x|>R\}}|h_\varepsilon(x)|^q dx+\int_{\{|x|>R\}}|\tilde{h}_\varepsilon(x)|^q dx\leq C_4,\]
	where $C_4=(C_2+C_3)/(1-2\delta BC_1)$. By virtue of Fatou's lemma, we take limits as $\varepsilon$ tends to zero, and obtain
	\[\int_{\{|x|>R\}}|x|^{\ell q}|\eta(x)|^q dx + \int_{\{|x|>R\}}|x|^{\ell q}|\eta'(x)|^q dx\leq C_4.\]
	Equivalently, $|\cdot|^\ell\eta\in L^q(\R)$ and $|\cdot|^\ell\eta'\in L^q(\R)$.
	
	Let us now consider a function $\varphi:\R\to (0,\infty)$ that is of class $\boC^1$ and satisfies that $\varphi(x)=|x|^\ell$ for every $|x|>1$. At this point, it is clear that $\varphi\eta\in W^{1,q}(\R)$. Hence, the Sobolev's embedding implies that $\varphi\eta\in L^\infty(\mathbb{R})$ and $\lim_{x\to\pm\infty}\varphi(x)\eta(x)=0$. This proves the result for $k=0$. As in the proof of Theorem~\ref{thm:exponentialdecay}, the rest of the proof follows  by induction.
\end{proof}

Conditions \eqref{eq:algdecaycond} in Theorem~\ref{thm:algebraicdecay} can be difficult  to verify. The next corollary provides sufficient (and easy to check) conditions on $\W$ that guarantee \eqref{eq:algdecaycond} and, in turn, algebraic decay of finite energy traveling waves. 

\begin{corollary}
	\label{cor:algebraicdecay}
Assume  that $\W$ satisfies \ref{h:lowerboundWbis} and that weakly differentiable up to order $s\in\N\setminus\{0\}$, with
	\begin{equation}
		\label{eq:corollarycond}
		D^s\wh\W\in L^\infty(\R).
	\end{equation} 
	Let $c\in [0,\sqrt{2\sigma})$ and let $u\in\boE(\R)$ be a solution to \eqref{TWc}.
	Then,
	\[|\cdot|^\ell D^k\eta\in L^2(\R)\cap L^\infty(\mathbb{R}),\quad \lim_{x\to\pm\infty}|x|^\ell D^k \eta(x)=0,\quad\text{ for all } \ell\in(0,s-{1}/{2}),\,\,k\in\N.\]
\end{corollary}

\begin{proof}
By \eqref{Lc:int}, since \ref{h:lowerboundWbis} holds, we see $L_c$ is bounded with 
$L_c \in L^1(\R)$, so that  $L_c \in L^2(\R)$.	Using  also \eqref{eq:corollarycond}, one may verify  that $D^s L_c\in L^1(\R)\cap L^\infty(\R)$ too. In particular, $D^s L_c\in L^2(\R)$. Then, applying Fourier transform, $|\cdot|^s\boL_c\in L^2(\R)$. Hence, since $\boL_c\in L^\infty(\R)$, it follows that \eqref{eq:algdecaycond} holds for $p=2$ and we can  apply Theorem~\ref{thm:algebraicdecay}.
\end{proof}

\subsection{Analyticity}
\label{sec:analyticity}

Let us recall that for  $H\in\boS'(\R)$, the associated multiplier operator $\boH$ is  defined by
\[\wh{\boH(\varphi)}(\xi)=H(\xi)\wh\varphi(\xi), \quad\text{ for all } \varphi \in \boS(\R).\]
We say that $H$ is an $L^p$-\emph{multiplier}, with $p\in [1,\infty]$, if there exists $\alpha>0$ such that
\[\|\boH(\varphi)\|_{L^p(\R)}\leq \alpha\|\varphi\|_{L^p(\R)}, \quad\text{ for all }\varphi\in L^p(\R).\]
The smallest $\alpha>0$ for which the previous inequality holds is the norm of the multiplier, and it is denoted by $\|\boH\|_p$. 
For instance, by assumption \ref{H0}, $\W$ is an $L^2$-multiplier and, by
\eqref{W-22}, 
$\|\W\|_2=\|\wh\W\|_{L^\infty(\R)}$.

We recall the so-called H\"ormander--Mikhlin multiplier theorem in \cite{hormander,mikhlin} (see also \cite{littman}) adapted to our one-dimensional setting, as follows.

\begin{theorem}[\hspace{1sp}\cite{hormander,mikhlin}]\label{thm:hormandermikhlin}
	Let $H:\R\to\R$ be a weakly differentiable function and  suppose that there exists $M>0$ such that 
	\begin{equation}\label{thm-M}
		\sup\{ \abs{ \xi^k	D^k H(\xi) } \ :\ \xi\in \R\setminus\{0\},\ k\in \{0,1\}  \}\leq M.
	\end{equation}
	Then $H$ is an $L^p$-multiplier for every $p\in(1,\infty)$. Moreover, 
	there exists a constant $C_p>0$, depending only on $p$, such that
	$$\norm{\mathcal H}_p\leq C_p M.$$
\end{theorem}

Assume that $\W$ satisfies  \ref{h:lowerboundWbis}, and including also the limit case  $\kappa=1/2,$ and that  $\wh\W$ is (weakly) differentiable.
We will apply this theorem to 
the function 
\[H_c(\xi)=\frac{-\xi^2}{\xi^2+2\wh\W(\xi)-c^2}, \quad \text{ for }  c\in [0,\sqrt{2\sigma}) .\]
Observe that then $H_c\in L^\infty(\R)$ and that 
\[\xi H'_c(\xi)=\frac{2\xi^3\big(\wh\W\big)'(\xi)-4\xi^2\wh\W(\xi)+2c^2\xi^2}{(\xi^2+2\wh\W(\xi)-c^2)^2}.\]
Therefore, $\xi\mapsto\xi H'_c(\xi)$ is a bounded function if 	
\begin{equation}
	\label{eq:analyticcond}
	\big|\big(\wh\W\big)'(\xi)\big|\leq C(|\xi|+1)\quad\text{a.e. }\xi\in\R,
\end{equation}
for some $C>0$. In this case, using Theorem~\ref{thm:hormandermikhlin}  we conclude that  $H_c$ is an $L^p$-multiplier for every $p\in (1,\infty)$. More precisely, for every $p\in(1,\infty)$, there exists a constant $\alpha_p>0$ such that 
\begin{equation}
	\label{eq:multiplier}
	\|\boH_c(\varphi)\|_{L^p(\R)}\leq\alpha_p\|\varphi\|_{L^p(\R)},\quad\text{ for all } \varphi\in L^p(\R),
\end{equation}
where $\wh \boH_c=H_c$.

Let $u\in\boE(\R)$ be a solution to \eqref{TWc}. 
We will exploit \eqref{eq:convolution} in order to prove that $\eta$ is a real analytic function.
First, we prove a technical lemma.

% {\color{blue} If $u$ admits a lifting (in particular, if $c>0$), then the analyticity is transferred to $u$ itself.} 

\begin{lemma}
	\label{lemma:recursivederivatives}
	Assume that there exist $\sigma\in (0,1]$ and  $\kappa\in [0,1/2]$ such that  $\wh\W(\xi)\geq \sigma-\kappa \xi^2$ a.e.\ on $\R$.
	 Assume in addition that $\wh\W$ is weakly differentiable and that there exists $C>0$ such that \eqref{eq:analyticcond} holds. Let $c\in[0,\sqrt{2\sigma})$ and let $u\in\boE(\R)$ be a solution to \eqref{TWc}. Let us denote
	\[\mu_k\coloneqq\max\{\|D^j F'\|_{L^q(\R)}:j=0,\dots,k\},\quad\text{ for all } k\in\N.\]
	Then there exist $\beta,\gamma>0$, depending on $\eta$ only through $\|D^j\eta\|_{L^2(\R)}$ and $\|D^j\eta\|_{L^\infty(\R)}$ for $j=0,1,2$, such that 
	\begin{align}
		\|D^j\eta\|_{L^\infty(\R)}\leq\beta \mu_k \quad &\text{ for all } k\in\N,\,\,\text{ for all } j=0,\dots,k+2,
		\label{eq:Linftyesestimate}
		\\
		\|D^j\eta\|_{L^2(\R)}\leq\beta \mu_k \quad &\text{ for all } k\in\N,\,\,\text{ for all } j=0,\dots,k+2,
		\label{eq:Lqesestimate}
		\\
		\mu_k\leq\gamma^{k}k^{k-1} \quad &\text{ for all } k\in\N\setminus\{0\},\qquad\mu_0\leq \frac{\gamma}{24\omega\beta^2}-1,
		\label{eq:muestimate}
	\end{align}
	where $\omega=\|\wh\W\|_{L^\infty(\R)}$.
\end{lemma}

\begin{proof}
	We start by proving \eqref{eq:Linftyesestimate} and \eqref{eq:Lqesestimate}. These estimates hold true for $k=0$ by simply choosing 
	\[\beta\geq\frac{\max\{\|D^j\eta\|_{L^\infty(\R)},\|D^j\eta\|_{L^2(\R)}:j=0,1,2\}}{\|F'\|_{L^2(\R)}}.\]
	Let us take $k\geq 1$ and $j\in\{3,\dots,k+2\}$. 
By \eqref{eq:convolution}, we have 
\begin{equation}\label{eq:analytic}
	\eta'''=\boH_c(F')\quad\text{ on }\R, \text{ with } 		F'=4\eta'(\W\ast\eta)+2\eta(\W\ast\eta').
\end{equation}
By using also  \eqref{eq:multiplier} with $p=2$, it follows that
	\[\|D^j\eta\|_{L^2(\R)}\leq \alpha_2\|D^{j-3}F'\|_{L^2(\R)}\leq\alpha_2\mu_k.\]
	Moreover, by invoking  the Sobolev's embedding (see Remark~\ref{remark:burenkov}), we obtain 
	\begin{align*}
		\|D^j\eta\|_{L^\infty(\R)} &\leq \frac{1}{2}\left(\|D^j\eta\|_{L^2(\R)}+\|D^{j+1}\eta\|_{L^2(\R)}\right)
		\\
		&\leq \frac{\alpha_2}{2}\left(\|D^{j-3}F'\|_{L^2(\R)}+\|D^{j-2}F'\|_{L^2(\R)}\right)\leq \alpha_2\mu_k.
	\end{align*}
	Therefore, we take $\beta\geq\alpha_2$ so that \eqref{eq:Linftyesestimate} and \eqref{eq:Lqesestimate} follow.
	
	As far as \eqref{eq:muestimate} is concerned, we will prove it by induction. Indeed, it holds true for $k=1$ if one chooses $\gamma\geq\mu_1$. Let us assume as induction hypothesis that there exists $\tilde{k}\in\N\setminus\{0,1\}$ such that \eqref{eq:muestimate} holds for every $k\leq\tilde{k}$. Next we compute for $k=\tilde{k}$, taking \eqref{eq:convolution} into account,	
	\begin{align*}
		\|D^{k+1}&F'\|_{L^2(\R)}=\|D^{k+1}\left(2\eta(\W\ast\eta')+4\eta'(\W\ast\eta)\right)\|_{L^2(\R)}
		\\
		&= \|2 D^k\left(\eta'(\W\ast\eta')+\eta(\W\ast\eta'')\right)+4 D^k\left(\eta'(\W\ast\eta')+\eta'(\W\ast\eta'')\right)\|_{L^2(\R)}
		\\
		&\leq 2\sum_{j=0}^k {k\choose j}(\|D^{j+1}\eta (\W\ast D^{k-j+1}\eta)\|_{L^2(\R)}+\|D^j\eta (\W\ast D^{k-j+2}\eta)\|_{L^2(\R)})
		\\
		&+4\sum_{j=0}^k{k\choose j}(\|(\W\ast D^{j+1}\eta) D^{k-j+1}\eta\|_{L^2(\R)}+\|(\W\ast D^j\eta) D^{k-j+2}\eta\|_{L^2(\R)}).
	\end{align*}
	Using \eqref{W-22}, \eqref{eq:Linftyesestimate} and  \eqref{eq:Lqesestimate}, we deduce that
	\begin{align*}
		\|D^{k+1}&F'\|_{L^2(\R)}\leq 2\omega\sum_{j=0}^k{k\choose j}(\|D^{j+1}\eta\|_{L^\infty(\R)} \|D^{k-j+1}\eta\|_{L^2(\R)}+\|D^j\eta\|_{L^\infty(\R)} \|D^{k-j+2}\eta\|_{L^2(\R)})
		\\
		&+4\omega\sum_{j=0}^k{k\choose j}(\|D^{j+1}\eta\|_{L^2(\R)} \|D^{k-j+1}\eta\|_{L^\infty(\R)}+\|D^j\eta\|_{L^2(\R)} \|D^{k-j+2}\eta\|_{L^\infty(\R)})
		\\
		&\leq 12\omega\beta^2\sum_{j=0}^k{k\choose j}\mu_j\mu_{k-j}.
	\end{align*}
	The induction hypothesis leads to
	\begin{align*}
		\|D^{k+1}&F'\|_{L^2(\R)}\leq 12\omega\beta^2\Bigg(2\mu_0\mu_k +\gamma^k \sum_{j=1}^{k-1}{k\choose j}j^{j-1}(k-j)^{k-j-1}\Bigg)
		\\
		&=12\omega\beta^2\Bigg(2(\mu_0\mu_k-\gamma^k k^{k-1}) +\gamma^k \sum_{j=0}^{k}\frac{k!}
		{j!(k-j)!}j^{(j-1)^+}(k-j)^{(k-j-1)^+}\Bigg),
	\end{align*}
	where we adopt the convention $0^0=1$. Now, a combinatorial lemma due to Kahane \cite{Kahane} implies that
	\[\|D^{k+1}F'\|_{L^2(\R)} \leq 12\omega\beta^2\left(2(\mu_0\mu_k-\gamma^k k^{k-1}) + 4\gamma^k k^{k-1}\right).\]
	Using the induction hypothesis again and choosing $\gamma>0$ large enough so that $\mu_0\leq\frac{\gamma}{24\omega\beta^2}-1$, we deduce that
	\[\|D^{k+1}F'\|_{L^2(\R)} \leq 24\omega\beta^2(\mu_0+1)\gamma^k k^{k-1}\leq\gamma^{k+1}k^{k-1}.\]
	In conclusion,
	\[\mu_{k+1}=\max\{\mu_k,\|D^{k+1}F'\|_{L^2(\R)}\}\leq\max\{\gamma^k k^{k-1},\gamma^{k+1}k^{k-1}\}=\gamma^{k+1}k^{k-1}\leq\gamma^{k+1}(k+1)^k,\]
	which completes the proof.
\end{proof}

We are thus led to the following analyticity of $\eta$ as follows.

\begin{theorem}\label{thm:analyticity}
	Under the hypotheses of Lemma~\ref{lemma:recursivederivatives}, for every solution $u\in\boE(\R)$ to \eqref{TWc} with $c\in [0,\sqrt{2\sigma})$, there exists 
	$r>0$ such that  $\eta=1-|u|^2$ and $u$ have  analytic extensions to
	the strip $\boS_r =\{z\in\C : \abs{\Im z}<r\}$. If $c \in (0,\sqrt{2\sigma})$, 
	then $u$ is real   analytic on $\R$, in the sense that $\Re(u)$ and  $\Im(u)$
	are  real   analytic  on $\R$.
	
\end{theorem}

\begin{proof}
	We need to prove that the Taylor series expansion
	about any point $x_0\in \R$ converges with radius of convergence $r >0$ independent of $x_0$. Indeed, let $I_r=[x_0-r,x_0+r]$, then,
%	
%	Given $x_0\in\R$, we aim to prove that there exists an interval $I$ centered at $x_0$ such that
%%	 the series
%%	\begin{equation}\label{series}
%%		\sum_{k\in\N} \frac{D^k\eta(x_0)}{k!}(x-x_0)^k
%%	\end{equation}
%%	is convergent for every $x\in I$ and
%	\begin{equation}\label{series}
%		\sum_{k=0}^\infty \frac{D^k\eta(x_0)}{k!}(x-x_0)^k=\eta(x_0).
%	\end{equation}
%	In order to do so,
	 by  Taylor's theorem,
	\begin{equation}\label{taylor}\eta(x)-\sum_{k=0}^{n}\frac{D^k\eta(x_0)}{k!}(x-x_0)^k=\frac{D^{n+1}\eta(\zeta)}{(n+1)!}(x-x_0)^{n+1}
	\end{equation}
	for every $x\in I_r$ and for some $\zeta\in I_r$. Now we deduce from Lemma~\ref{lemma:recursivederivatives} that
	\[|D^k\eta(\zeta)|\leq\beta\mu_k\leq\beta\gamma^k k^{k-1}\]
	for every $k\in\N\setminus\{0\}$. 
	Since 
	$$
	\big(\frac{\gamma^k k^{k-1}}{k!}\big)^{1/k}\to \gamma e,\quad \text{as } k\to\infty,
	$$
	we conclude that  the left-hand side of \eqref{taylor} goes to zero as $n\to\infty$ and that the radius of convergence satisfies $r\geq (\gamma e)^{-1}$.
	In addition, $\eta$ has an
	analytic extensions to the strip $\boS_r$.
	
	In the case that $c\in (0,\sqrt{2\sigma})$, by Proposition~\ref{prop:etakidentities}, we have $\sup_\R \eta <1$, so that 
	$\rho=\abs{u}=\sqrt{1-u^2}$ is a real analytic function, as a composition of real analytic functions. Also,  $\theta$ given by \eqref{eq:deftheta} 
	is real analytic as the integral of a real analytic function.
	Consequently, $\Re(u)=\rho \cos(\theta)$ and $\Im(u)=\rho \sin(\theta)$ are 
	 real analytic functions on $\R$.
 \end{proof}
	
%	Moreover, by choosing $I=\left[x_0-\frac{1}{\gamma},x_0+\frac{1}{\gamma}\right]$, it follows that $\frac{\gamma^k k^{k-1}}{k!}|x-x_0|^k$ converges to zero as $k$ tends to infinity (uniformly) for every $x\in I$. Gathering all together, we deduce directly from \eqref{taylor} that the series in \eqref{series} is convergent and that its sum coincides with $\eta(x_0)$.

\begin{remark}
	As pointed out by Corollary 4.1.5 in \cite{BonaLi2}, the fact that  $\eta$ has an
	analytic extension to the strip $\boS_r$ implies the following exponential decay of its Fourier transform:
	$$
	\int_\R \abs{\hat \eta(\xi)}^2 e^{2\mu\abs{\xi}}d\xi< \infty, \quad \text{ for all }\mu \in (0,r).
	$$
\end{remark}

%%%%%%%%%%%%%%%%%%%%%%%%%%%%%%%%%%%%%%%%
We end this section by proving Corollary~\ref{thm:symmetry} as a consequence of  Theorem~\ref{thm:analyticity}.

%For $\gq\geq 0$, let us consider the following minimization curve:
%\[\Emin(\gq):=\inf\{E(v):v\in\boN\boE(\R), p(v)=\gq\}.\]
%Under suitable assumptions on $\boW$, it is known from \cite{deLaireMennuni} that, for each $\gq>0$, there exists $u\in\boN\boE(\R)$ such that 
%\begin{equation}\label{eq:minimizationproblem}
%	E(u)=\Emin(\gq),\quad p(u)=\gq.
%\end{equation}
%In such a case, it is also possible to prove that $u$ is, in fact, a solution to \eqref{TWc} for some $c>0$ depending on $\gq$.
%
%In this section we will prove that, given $\gq>0$, any solution $u\in\boN\boE(\R)$ to \eqref{eq:minimizationproblem} satisfies some symmetry properties. In fact, we will see that $|u|$ is symmetric about some point $x_0\in\R$. By translation invariance, we may assume that $x_0=0$, that is so say, $|u|$ is even. This result shows that the modulus of every solution to \eqref{TWc} obtained my solving \eqref{eq:minimizationproblem} is even, i.e.\ the solutions obtained in \cite{deLaireMennuni} have even modulus.
%
%
%\begin{theorem}
%	{\color{blue}Assume that $\boW$ satisfies...} Let $\gq>0$. Then, every solution $u=\rho e^{i\theta}\in\boN\boE(\R)$ to \eqref{eq:minimizationproblem} satisfies that, up to translations, $\rho$ is an even function and, up to multiplying $u$ by a constant of modulus one, $\theta$ is an odd function. 
%\end{theorem}

\begin{proof}[Proof of  Corollary~\ref{thm:symmetry}]
Let $\gq\in (0,\gq_*)$ and let  $u=\rho e^{i\theta}\in \boN\boE(\R)$ be 
the nontrivial solution to  \eqref{TWc}, given by Theorem~{1} in \cite{delaire-mennuni}, satisfying  
\begin{equation}
	\label{dem:u:min}
	E(u)=E_{\min}(\gq).
\end{equation}
%$p(u)=\gq$, given by Theorem~1 in \cite{delaire-mennuni}.

Arguing as in the proof of Proposition 3.12 in   \cite{delaire-mennuni}, we see that 
there exists $a_0\in\R$ such that $$\frac 1 2 \int_{a_0}^\infty (1-\rho^2)\theta'=\frac \gq 2,$$
which allows us 	to define the following function
	\[\tilde{u}(x):=\tilde{\rho}(x)e^{i\tilde{\theta}(x)}=\rho(x-a_0)e^{i(\theta(x-a_0)-\theta(-a_0))}.\]
	Notice that $\tilde{u}$ is nothing but $u$ multiplied by the constant of modulus one $e^{i\theta(-a_0)}$ and translated in the space variable, so $\tilde{u}$ is still satisfies \eqref{dem:u:min}, i.e.\ it is a solution to the minimization problem.  Moreover, $\tilde{u}$ satisfies that $\tilde{\theta}(0)=0$ and
	\begin{equation}\label{uppermomentum}
		\frac 1 2 \int_0^\infty (1-\tilde{\rho}^2)\tilde{\theta}'=\frac \gq 2.
	\end{equation}
	Furthermore,
	\begin{equation}\label{lowermomentum}
		\frac 1 2 \int_{-\infty}^0 (1-\tilde{\rho}^2)\tilde{\theta}'=p(v)-\frac 1 2 \int_0^\infty (1-\tilde{\rho}^2)\tilde{\theta}'=\frac \gq 2.
	\end{equation}
	For notational simplicity,  we continue to write $u$, $\rho$ and $\theta$ for $\tilde u$, $\tilde \rho$ and $\tilde \theta$.
By using the reflection operators $T^\pm$ and $S^\pm$ introduced in  the proof of Proposition 3.12 in   \cite{delaire-mennuni},  and the fact that $\rho$ and $\theta$ are continuous, it follows that the functions
 $$u^{\pm}=(T^\pm\rho)e^{iS^\pm\theta}$$ belong to $\boN\boE(\R)$. Bearing in mind \eqref{uppermomentum} and \eqref{lowermomentum}, we obtain that
$	p(u^{\pm})=\gq,$
	which implies that 
$\Emin(\gq)\leq E(u^\pm).$

On the other hand, as in Proposition 3.12 in   \cite{delaire-mennuni}, we get 
$$
E(u^+)+E(u^-)=2E_\k(u)+E_\p(u^+)+E_\p(u^-)\quad 
\text{ and }\quad 
	E_\p(u^+)+E_\p(u^-)\leq 2E_\p(u).
$$
Since  $u$ satisfies \eqref{dem:u:min}, we deduce that
	\[\Emin(\gq)\leq \frac{E(u^+)+E(u^-)}{2}\leq E(u)=\Emin(\gq).\]
	Hence,
	\[\frac{E(u^+)+E(u^-)}{2}= E(u)=\Emin(\gq).\]
	Observe that
	\[\Emin(\gq)\leq E(u^+)=2E(u)-E(u^-)\leq \Emin(\gq).\]
	In consequence, $E(u^+)=E(u^-)=E(u)=\Emin(\gq)$. This shows that $u^\pm$ and $u$ are solutions to the minimization problem \eqref{Emin} and therefore, $u^\pm$ and $u$ satisfy \eqref{TWc}, for some $c$ depending on $\gq$. By virtue of Theorem~\ref{thm:analyticity}, we have that $|u^\pm|^2$ and $|u|^2$ are real analytic functions. Thus, since $|u^+|=|u|$ in $\R_+$, then $|u^+|=|u|$ in $\R$. This proves that $\rho=|u|$ is even. On the other hand, from \eqref{eq:theta},  from the symmetry of $\rho$ and from the fact that $\theta(0)=0$, we derive
	\begin{equation*}
		\theta(x)=\frac c 2\int_0^x \left(\frac{1}{\rho(y)^2}-1\right)dy=-\frac c 2\int_0^{-x} \left(\frac{1}{\rho(y)^2}-1\right)dy=-\theta(-x).
	\end{equation*}
	This concludes the proof.
\end{proof}

%%%%%%%%%%%%%%%%%%%%%%%%%%%%%%%%%%%
\section{Proofs of the examples}
\label{sec:examples}

\begin{proof}[Proof of Theorem~\ref{thm:ex1}]
	The existence of a solution $u$ for every $c\in(0,\sqrt 2)$ is an immediate consequence of Theorem~\ref{thm:existenceallc}. Also, since $\wh \W_{\alpha,\beta}$ fulfills \eqref{eq:algdecaycond}, $\eta=1-\abs{u}^2$
	is real analytic by Theorem~\ref{thm:analyticity}.
	 	The nonexistence of finite energy solutions follows from Theorem~\ref{thm:ex2} and the fact that 
	$(\wh \W_{\alpha,\beta})'' (0)=4 \alpha\beta^{-2}(\beta-2\alpha)^{-1}\neq -1$.
	
It remains to prove the exponential decay. By explicit computations, we can find for some $\beta_1,\beta_2>0$,  depending 
	only on $\alpha$, $\beta$ and $c$ such that 
	\[\mathcal{L}_c(x)=\alpha_1 e^{-\beta_1|x|}+\alpha_2 e^{-\beta_2|x|},\quad\text{ for all } x\in\mathbb{R},
	\quad \text{with }
	\alpha_1=\frac{\beta^2-\beta_1^2}{2\beta_1(\beta_2^2-\beta_1^2)},\ \alpha_2=\frac{\beta_2^2-\beta^2}{2\beta_2(\beta_2^2-\beta_1^2)}.
	\]
Thus, $\mathcal{L}_c$ satisfies the condition \eqref{expLccondition} in Theorem~\ref{thm:exponentialdecay} with $m=\min\{\beta_1,\beta_2\}$ and $p=\infty$. 
\end{proof}

\begin{proof}[Proof of Theorem~\ref{thm:ex2}]
It is clear that \ref{H0} holds for the three potentials. Notice that  
$2-\cos(\lambda\xi)\geq 1$, for $\xi\in \R$, and
using the  elementary inequalities $e^{x}\geq 1+x$ and $\sin(x)/x\geq 1-x^2/6$, for $x\in\R$,  
\begin{align}
	e^{-\lambda \xi^2}\geq 1-{\lambda}\xi^2 \quad \text{ and }\quad 
	\frac{\sin(\lambda \xi)}{\lambda \xi}\geq 1-\frac{\lambda^2 \xi^2}{6}, \quad \text{ for all }\xi\in\R.
\end{align}
Hence,  \ref{h:lowerboundWbis} is satisfied, with 
 $(\sigma,\kappa)=(1,0)$, $(\sigma,\kappa)=(1,\lambda)$ and $(\sigma,\kappa)=(1,\lambda^2/3)$, in case (i), (ii) and (iii), respectively.
 In particular, in the three cases we have
 \begin{equation}
 	\label{dem:Mc}
 M_c(\xi)=\xi^2 +2\wh\W_\lambda(\xi)-c^2\geq 2-c^2+\xi^2(1-2\kappa)>0, \ \text{ for all }\xi\in\R \text{ and } c\in (0,\sqrt 2),
 \end{equation}
 and the existence of solutions is given by Theorem~\ref{thm:existenceae}.
 The analyticity of  $\eta=1-\abs{u}^2$ follows from Theorem~\ref{thm:analyticity}.

The nonexistence of finite energy solutions follows from Theorem~\ref{thm:nonexistence:intro} and the  fact that 
in the case (i) we have
$(\wh\W_\lambda)'' (0)=\lambda^2$, while in the case (ii),
$(\wh\W_\lambda)''(0)=-2\lambda$.

To prove the exponential decay, in view of \eqref{dem:Mc}, we deduce  that in all the cases  $\wh\W_\lambda$ can be extended as an 
 analytic function on $\C$. Hence, we only need to verify that 
for fixed $c\in(0,\sqrt 2)$ and $\lambda$, we can find 
a constant $\delta=\delta(c,\lambda)>0$ such that 
$M_{c}(z)=z^2 +2\wh\W_\lambda(z)-c^2$ does not vanish on the strip 
 $\boS_\delta:=\{z\in\C : \abs{\Im z}<\delta\}$
 and that $L_c(z)=(M_c(z))^{-1}$ satisfies the integrability condition in \eqref{cond:reed}. This will imply that $e^{\delta\abs{\cdot}}\boL_c \in L^2(\R)$, so that 
 the decay follows by invoking Theorem~\ref{thm:exponentialdecay} with $p=2$.
 
Let us show that there is  $\delta=\delta(c,\lambda)\in(0,1)$ such that 
\begin{equation}
	\label{cota:Mc}
	\abs{M_c(\xi+i w )}\geq \delta,\quad \text{ for all }\abs{w}\leq \delta,
	\text{ for all }\xi\in \R.
\end{equation}
 Arguing by contradiction, we get the existence of  sequences
 $\delta_n\in(0,1)$, $\abs{w_n}\leq 1$,  $\xi_n\in \R$, with 
 $\delta_n\to 0$, $w_n\to 0$, and 
  such that  
\begin{align}
	\label{dem:Re}
T(\xi_n,w_n)&=o_n(1), \quad \text{ with } T(\xi,w):=\xi^2-w^2+2 \Re(\wh\W_\lambda(\xi+iw))-c^2,\\
G(\xi_n,w_n)&=o_n(1),
 \quad \text{ with }
G(\xi,w) :=2\xi w +2\Im(\wh\W_\lambda(\xi+iw)).
\end{align}
By using the explicit expressions for $\wh\W_\lambda$, it is easy to check that 
in the case (i), we have
 \begin{align}
\abs{\wh\W_\lambda (\xi+iw)}=\abs{2-\cos(\lambda \xi)\cosh(\lambda w)+i \sin(\lambda \xi)\sinh(\lambda w)}
\leq 2 +\cosh(\lambda w)+\sinh(\lambda \abs{w}). 
 \end{align}
In the case (ii), we get $\abs{\wh\W_\lambda (\xi+iw)}=\abs{e^{-\lambda(\xi^2-w^2+2i\xi w)}}\leq e^{\lambda w^2}$, 
while in the case (iii),
$$\abs{\wh\W_\lambda (\xi+iw)}=\frac{\abs{ \sin(\lambda \xi)\cosh(\lambda w)+i \cos(\lambda \xi)\sinh(\lambda w)  }}{\lambda\sqrt{\xi^2+w^2}}\leq \cosh(\lambda w)+\frac{\sinh(\lambda \abs{w})}{\lambda \abs{w}}. $$
Hence, $\wh \W_\lambda$ is bounded on the strip $\boS_1$, that is, there is 
$K>0$ such that $\abs{\wh\W_\lambda (\xi+iw)}\leq K$, for all $\xi+iw \in \boS_1$.
 Therefore, we infer from \eqref{dem:Re} that $\{\xi_n\}$ is bounded, so that
 there are $\xi_*\in \R$ and  subsequence, that we do not relabel, such that 
  $\xi_n\to \xi_*$. In this manner, passing to the limit in \eqref{dem:Re}, we conclude that $T(\xi_*,0)=0$, i.e.\
$M_c(\xi_*)=0$, which contradicts \eqref{dem:Mc}. The proof of \eqref{cota:Mc} is completed.

By \eqref{cota:Mc}, the function
 $$L_c(\xi+ iw)=\frac{1}{M_c(\xi+ iw)}=
 \frac{1 }{T(\xi,w)+i G (\xi,w)}
 $$ defines an analytic function on the strip $\boS_\delta$. Also, for all $\abs{w}\leq \delta\leq 1$, we  infer the estimate
 $$
  \abs{L_c(\xi+ iw)}\leq\begin{cases}
\dfrac{1 }{\xi^2-3-2K}, &\text{ if }\xi^2\geq 4+2K,\\
  	\delta^{-1}, &\text{ otherwise}.
  \end{cases} 
 $$
 Consequently,
$ \sup_{\abs{w}\leq \delta }\norm{ L_c(\cdot +i w)}_{L^2(\R)}<\infty,$
which completes the proof of the exponential decay.

It is left to prove the existence of $u_c$ for every $c\in (0,\sqrt{2})$ in the case \ref{case:deltas} for $\lambda\leq\sqrt{2/3}.$ 
To do so, it is enough to verify that the hypotheses of Theorem~\ref{thm:existenceallc} hold.
It is clear that \ref{h:derivative} and \ref{W:infty} are satisfied. In order to check \ref{W:apriori}, let us denote $\mu_\lambda=-\frac{1}{4}\big(\delta_{-\lambda}+\delta_\lambda\big)$, so that $\W_\lambda=2(\delta_0+\mu_\lambda$). Thus $\mu_\lambda^+=0$ and $\|\mu_\lambda^-\|=1/2<1$, so that Proposition~\ref{prop:universalestimate} implies that \ref{W:apriori} holds with $V_0(c)=\sqrt{1+c^2/4}$.
	
Finally, we show that  \ref{h:restrictive} is fulfilled, at least for $\lambda\in (0,\sqrt{2/3}]$. 
Indeed, in this case, let us set $\mathfrak{s}=\min_{x\in\R}(\sin(x)/x)\in (-1,0)$ and $m_\lambda=-\mathfrak{s}\lambda^2 \in (0,2/3)$. Thus
\[\big(\wh\W\big)'(\xi)=\lambda\sin(\lambda\xi)\geq -m_\lambda\xi, \quad \text{ for all }\xi\geq 0.\]
Furthermore, with this choice of $m_\lambda$, we get, for all $c\in (0,\sqrt 2)$, 
$$m_\lambda V_0(c)^2\leq (3/2)m_\lambda<1.
$$
Hence, we can apply Theorem~\ref{thm:existenceallc}.
\end{proof}

\begin{remark}
A careful study of the functions $T(\xi,w)$ and $G(\xi,w)$ in the proof of 
 Theorem~\ref{thm:ex2}
should lead to 
the sharp exponential decay of the solitons.
\end{remark}

\begin{remark}
Notice that if $\W_\lambda$ is given by \eqref{ex:softcore}, then $(\wh\W_\lambda)''(0)=-\lambda^2/3$, so that $(\wh\W_\lambda)''(0)\neq -1$, for all
$\lambda\in (0,\sqrt 3)$. However, we cannot apply Theorem~\ref{thm:ex2} due to the change of sign of $\W_\lambda$.
\end{remark}

\begin{proof}[Proof of Theorem~\ref{thm:ex4}]
For $\kappa\in(0,1/2)$,	it is obvious that $\W_{\kappa}$ satisfies \ref{H0} and \ref{h:lowerboundWbis}, so we can use Theorem~\ref{thm:existenceae}.
For $\kappa=1/2$, we can apply Corollary \ref{cor:condition12}. Therefore, 
for  any $\kappa\in(0,1/2)$, 
 we conclude the existence of nontrivial solutions to $(S(\W_\kappa,c))$ in $\boN\boE(\R)$ for almost every $c\in (0,\sqrt{2})$. Moreover, since   $(\wh \W_{\kappa})'(\xi)= -2\kappa\abs{\xi}$ 
 for  $\abs{\xi}<1/\sqrt{\kappa}$, and  $(\wh \W_{\kappa})'(\xi)=0$, for 
 $\abs{\xi} > 1/\sqrt{\kappa}$, we can apply Theorem~\ref{thm:analyticity} to obtain that 
 $\eta=1-\abs{u}^2$ is real analytic.
 In addition, since $\W_{\kappa}$ fulfills  condition (i) if $\kappa=1/2$ in Theorem \ref{thm:nonexistence},  and condition (ii) otherwise, we get the nonexistence for $c=\sqrt 2$.

	It is left to prove the algebraic decay of $\eta$. Remark that we can apply Corollary~\ref{cor:algebraicdecay} with $s=1$, but we can get a better decay by computing explicitly $\boL_c$. In fact, since $L_c\in L^1(\R)$, then	
	\[\mathcal{L}_c(x)=\frac{1}{2\pi}\int_\mathbb{R} e^{ix\xi}L_c(\xi)d\xi=\frac{1}{2\pi}\int_{-\alpha}^{\alpha}\frac{\cos(x\xi)}{(1-2\kappa)\xi^2+2-c^2}d\xi+\frac{1}{2\pi}\int_{\{|\xi|>\alpha\}}\frac{\cos(x\xi)}{\xi^2-c^2}d\xi,\]
	where we have used that $L_c$ is even. Now, after  applying integration by parts twice, we get
	\[\boL_c(x)=\frac{1}{x^2}(A\cos(\alpha x)+g(x))\quad\text{ for all }x\not=0,
	\ \text{where } A=\frac{4\alpha\kappa}{\pi(\alpha^2-c^2)^2}
	\]
	 and
	\[g(x)=-\frac{2}{\pi}\int_\alpha^\infty\frac{(3\xi^2+c^2)\cos(x\xi)}{(\xi^2-c^2)^3}d\xi-\frac{2(1-2\kappa)}{\pi}\int_0^\alpha\frac{(3(1-2\kappa)\xi^2+2-c^2)\cos(x\xi)}{((1-2\kappa)\xi^2+2-c^2)^3}d\xi.\] 
	Observe that $g\in L^\infty(\R)$. 
	We are not interested in the value $\boL_c(0)$. We simply remark that, since $L_c\in L^1(\R)$, it follows that $\boL_c\in  L^\infty(\R)$ and therefore,
	$(1+|\cdot|^2)\boL_c\in L^\infty(\R).$
Consequently, the decay in  \eqref{eq:decay3} follows by  applying Theorem~\ref{thm:algebraicdecay} with $s=2$ and $p=\infty$.
\end{proof}

%%%%%%%%%%%%%%%%%%%%%%%%%%%%%%%%%%%
\begin{merci}
	The authors acknowledge support from the Labex CEMPI (ANR-11-LABX-0007-01).
	A.~de Laire was also supported by the ANR project ODA (ANR-18-CE40-0020-01). S.~L\'opez-Mart\'inez was also supported by PGC2018-096422-B-I00 (MCIU/AEI/FEDER, UE) and Junta de Andaluc\'ia FQM-116. S.~L\'opez-Mart\'inez would like to thank the members of the Laboratoire Paul Painlev\'e (Universit\'e de Lille) and of the team PARADYSE (Inria Lille - Nord Europe) for their support and hospitality during his postdoc stay, where this work was carried out. 
\end{merci}

\section*{Appendix}

We include here the proof of the deformation lemma.

\begin{proof}[Proof of Lemma~\ref{lemma:deformation}]
	For $j=1,2,3$, let us denote
	\[A_j=J_c^{-1}([\gamma-2\varepsilon,\gamma+2\varepsilon])\cap Z_{\delta_j}.\]
	Since these are closed sets in $H^1(\R)$, we may define a functional $\psi:H^1(\R)\to\R$ of class $\boC^1$ such that $0\leq\psi(v)\leq 1$ for every $v\in H^1(\R)$ and
	\[\psi(v)=\left\{
	\begin{array}{l l}
		1 & \text{ for all } v\in A_2,
		\\
		0 & \text{ for all } v\in H^1(\R)\setminus A_1.
	\end{array}
	\right.
	\]
	Let us now consider the vector field $\varphi:H^1(\R)\to H^1(\R)$ given by
	\[\varphi(v)=\left\{
	\begin{array}{l l}
		\displaystyle -\psi(v)\frac{J'_c(v)}{\|J'_c(v)\|_{H^{-1}(\R)}} & \text{ for all } v\in A_1,
		\\
		0 & \text{ for all } v\in H^1(\R)\setminus A_1.
	\end{array}
	\right.
	\]
	Clearly, $\varphi\in \boC^1(H^1(\R))$ (see Lemma~\ref{lemma:Jsmooth}) and $\|\varphi(v)\|_{H^1(\R)}\leq 1$ for every $v\in H^1(\R)$. 
	
	For any $v\in H^1(\R)$,  we consider the Cauchy problem
	\begin{equation*}
		\begin{cases}
			w'(t)=\varphi(w(t)),\quad\text{ for all } t\geq 0,
			\\
			w(0)=v.
		\end{cases}
	\end{equation*}
	The classical ODE theory,  the Cauchy problem has a unique solution $w(\cdot,v)\in H^1(\R)$ defined in $[0,+\infty)$. Let us show that $w(t,\NV)\subset\NV$ for every $t\geq 0$. Indeed, let $v\in\NV$. Clearly, $w(0,v)=v\in\NV$. Moreover, since $w(\cdot,v)$ is continuous and $v\in\NV$, there exists $\tilde{t}>0$ such that $w(t,v)\in \NV$ for every $t\in [0,\tilde{t})$. Let us assume by contradiction that $s\coloneqq\sup\{\tilde{t}>0: w(t,v)\in\NV\,\,\forall t\in [0,\tilde{t})\}<+\infty$. Then, $w(s,v)\in\partial\NV$. In particular, $w(s,v)\in H^1(\R)\setminus A_1$, so $\varphi(w(s,v))=0$. Actually, since $H^1(\R)\setminus A_1$ is open, then there exists $\tilde{s}\in (0,s)$ such that $w(t,v)\in H^1(\R)\setminus A_1$ for every $t\in [s-\tilde{s},s]$. Therefore, $\varphi(w(t,v))=0$ for every $t\in[s-\tilde{s},s]$. That is, $w'(t,v)=0$ for every $t\in[s-\tilde{s},s]$, so $w$ must be constant in $[s-\tilde{s},s]$ and, in consequence, $w(s-\tilde{s},v)=w(s,v)$. But this is a contradiction since, by definition of $s$, it is necessary that $w(s-\tilde{s},v)\in \NV$.

	On the other hand, for any $v\in H^1(\R)$ and $t\geq 0$, we have 
	\[\|w(t,v)-v\|_{L^\infty(\R)}\leq \|w(t,v)-v\|_{H^1(\R)}\leq \left\|\int_0^t\varphi(w(s,v)) ds\right\|_{H^1(\R)}\leq t.\]
	Observe that, in the previous inequality, we have used that the norm of the continuous embedding $H^1(\R)\subset L^\infty(\R)$ is equal to one (see Remark~\ref{remark:burenkov}). Hence, we deduce that $w(t,Z_{\delta_3})\subset Z_{\delta_2}$ for every $t\leq \delta_3-\delta_2$.
	
	Let us define $h:[0,1]\times \NV\to H^1(\R)$ by
	\[h(t,v)=w\left((\delta_3-\delta_2) t,v\right),\quad\text{ for all } (t,v)\in [0,1]\times \NV.\]
	We have already verified that $h([0,1]\times\NV)\subset\NV$. Furthermore, with this definition, items \ref{def1} and \ref{def2bis} are obviously satisfied. On the other hand, if $v\in \NV\setminus A_1$, then $\varphi(v)=0$, so $w(t)=v$ is the unique solution to the Cauchy problem and item \ref{def2} is satisfied too.
	
	As far as item \ref{def3} is concerned, let $v\in \NV$. Since $w(t,v)\in\NV$ for every $t\geq 0$, then the function $J_c(w(\cdot,v))$ is differentiable and
	\[\frac{d}{dt} J_c(w(t,v))=\langle J'_c(w(t,v)),w'(t,v)\rangle= \langle J'_c(w(t,v)),\varphi(w(t,v))\rangle\leq 0.\]
	Thus, $J_c(w(\cdot,v))$ is nonincreasing and item \ref{def3} holds true.
	
	Lastly, we check item \ref{def4}. Indeed, let $v\in J_c^{\gamma+\varepsilon}\cap Z_{\delta_3}$. If there exists $t\in \left[0,\delta_3-\delta_2\right)$ such that $J_c(w(t,v))<\gamma-\varepsilon$, then item \ref{def3} implies that $J_c\left(w\left(\delta_3-\delta_2,v\right)\right)<\gamma-\varepsilon$, so $w\left(\delta_3-\delta_2,v\right)\in J_c^{\gamma-\varepsilon}\cap Z_{\delta_2}$. Otherwise, for every $t\in\left[0,\delta_3-\delta_2\right)$, one has 
	\[\gamma-\varepsilon\leq J_c(w(t,v))\leq J_c(w(0,v))=J_c(v)\leq \gamma+\varepsilon.\]
	In particular, $w(t,v)\in A_2$ for every $t\in\left[0,\delta_3-\delta_2\right)$. In addition, by the definition of $\varphi$, we derive 
	\begin{align*}
		J_c\left(w\left(\delta_3-\delta_2,v\right)\right)&=J_c(v)+\int_0^{\delta_3-\delta_2} \frac{d}{dt} J_c(w(t,v))dt
		\\
		&=J_c(v)+ \int_0^{\delta_3-\delta_2} \langle J_c'(w(t,v)),\varphi(w(t,v))\rangle dt	
		\\	
		&=J_c(v)- \int_0^{\delta_3-\delta_2} \|J_c'(t,v))\|_{H^{-1}(\R)} dt	
		\\	
		&\leq \gamma+\varepsilon-(\delta_3-\delta_2)\frac{2\varepsilon}{\delta_3-\delta_2}=\gamma-\varepsilon.
	\end{align*}
	This concludes the proof.	
\end{proof}

We include also a technical lemma needed in the proof of Proposition~\ref{prop:key}. Even though the proof of the lemma is elementary, it is not straightforward and a sort of uniform continuity of the functional $J_c$ is required.

\begin{lemma}
\label{lemma:technical}
	Let $c>0$ and fix $\delta_c\in (0,1)$, $R_c>0$ and $\gamma_c>0$. For every $\alpha>0$ and $\delta\in (0,1)$, we consider the set 
	\[X_{\alpha,\delta}=J_c^{-1}((\gamma_c-\alpha,\gamma_c+\alpha))\cap Z_\delta,\]
	where $Z_\delta$ is defined by \eqref{Z}, i.e.
	\[Z_\delta=\{v\in\NV:\,\,\|v\|_{H^1(\R)}\leq R_c+1-\delta,\,\, v\leq 1-\delta\text{ on }\R\}.\]
	Let us also denote
	\[I_{\alpha,\delta}=\inf\{\|J_c'(v)\|_{H^{-1}(\R)}:\,\, v\in X_{\alpha,\delta}\}.\]
	Assume that there exists $\alpha>0$ such that $I_{\alpha,\delta_c}>0$. Then there exists $\bar{\delta}\in (0,\delta_c)$ such that $I_{\alpha,\bar{\delta}}>0$.
\end{lemma}

\begin{proof}
	Let us take $\delta\in (0,\delta_c)$. Observe that $Z_{\delta_c}\subset Z_\delta$. Recall that $J_c\in\boC^2(\NV)$, see Lemma~\ref{lemma:Jsmooth}. It is simple to check from \eqref{eq:Jderivative2} that $\|J''_c(v)\|\leq C$ for every $v\in Z_\delta$, where $C>0$ depends only on $R_c$ and $\delta$. Hence, since $Z_\delta$ is connected and convex, the Mean Value theorem implies that $J_c'$ is Lipschitz in $Z_\delta$, with Lipschitz constant denoted by $l_\delta>0$. In consequence, for every $\varepsilon>0$, if we take $\beta=\varepsilon/l_\delta$, then the following holds for  any $v,w\in Z_\delta$ satisfying $\|v-w\|_{H^1(\R)}<\beta$,
	\begin{equation*}
		\big|\|J'_c(v)\|_{H^{-1}(\R)}-\|J'_c(w)\|_{H^{-1}(\R)}\big|\leq\|J'_c(v)-J'_c(w)\|_{H^{-1}(\R)}\leq l_\delta\|v-w\|_{H^1(\R)}<\varepsilon.
	\end{equation*}
	In the previous inequality, we take $\varepsilon=I_{\alpha,\delta_c}/2$. Therefore, if $v\in Z_\delta$, $w\in X_{\alpha,\delta_c}$ and $\|v-w\|_{H^1(\R)}\leq\beta$, then
	\[\|J_c'(v)\|_{H^{-1}(\R)}>\|J_c'(w)\|_{H^{-1}(\R)}-\frac{I_{\alpha,\delta_c}}{2}\geq \frac{I_{\alpha,\delta_c}}{2}>0.\]
	We have proved that
	\begin{equation}
	\label{positiveinX}
		\|J'_c(v)\|_{H^{-1}(R)}\geq\frac{I_{\alpha,\delta_c}}{2}>0\quad\text{for all } v\in \boX,
	\end{equation}
	where
	\[\boX=\{v\in X_{\alpha,\delta}:\,\,\text{dist}(v,X_{\alpha,\delta_c})<\beta\}.\]
	Notice that
	\[X_{\alpha,\delta_c}\subset\boX\subset X_{\alpha,\delta}.\]
	Using \eqref{positiveinX}, the proof of the lemma will be finished as soon as we show that there exists $\bar{\delta}\in (\delta,\delta_c)$ such that
	\[X_{\alpha,\delta_c}\subset X_{\alpha,\bar{\delta}}\subset\boX\subset X_{\alpha,\delta}.\]
	In order to do so, let $\bar\delta\in(\delta,\delta_c)$ to be chosen later, and let $v\in X_{\alpha,\bar\delta}$. For some $\lambda>0$ to be chosen later too, we aim to prove that 
	\begin{equation}
	\label{distance}
		\lambda v\in X_{\alpha,\delta_c}\quad \text{ and }\quad \|v-\lambda v\|_{H^1(\R)}<\beta,
	\end{equation}
	which implies that $v\in\boX$. 
	
	On the one hand, simple computations show that a sufficient condition for $\lambda v\in Z_{\delta_c}$ and $\|v-\lambda v\|_{H^1(\R)}<\beta$ is
	\begin{equation}
	\label{sufcond}
		1-\frac{\beta}{R_c+1-\bar\delta}<\lambda<\frac{1-\delta_c}{1-\bar\delta}.
	\end{equation}
	Observe that $\lambda>0$ can be chosen so that \eqref{sufcond} holds whenever $\bar\delta$ is close enough to $\delta_c$.
	
	On the other hand, it is left to prove that $J_c(\lambda v)\in (\gamma_c-\alpha,\gamma_c+\alpha)$. Indeed, since $J_c$ is uniformly continuous in $Z_\delta$ and $J_c(v)\in(\gamma_c-\alpha,\gamma_c+\alpha)$, it follows that there exists $\lambda_0\in (0,1)$, independent of $v$, such that $J_c(\lambda v)\in (\gamma_c-\alpha,\gamma_c+\alpha)$ for every $\lambda\in (\lambda_0,1]$. Now we take $\bar\delta$ even closer to $\delta_c$ so that $\lambda_0<\frac{1-\delta_c}{1-\bar\delta}$. Thus, for every $\lambda\in (\lambda_0,1)$ satisfying \eqref{sufcond}, we have that \eqref{distance} holds. The proof is finished.
\end{proof}

%%%%%%%%%%%%%%%%%%%%%%%%%%%%%%%%%%%
\bibliographystyle{abbrv}

\end{document}